\documentclass[a4paper,12pt]{article}

\usepackage{amsmath}
\usepackage{amsthm}
\usepackage{amssymb}
\usepackage{latexsym}
\usepackage{graphicx}
\usepackage{stmaryrd}
\usepackage{mathrsfs}
\usepackage{enumerate}
\usepackage{cases}
\usepackage{color}

\newcommand{\BLACK}{\color{black}}

\definecolor{dGREEN}{rgb}{0.0,0.5,0.5}

\newcommand{\glalign}[2]{\lower.6ex\vbox{
\baselineskip\lineskip\ialign{$#1\hfil##\hfil$\crcr#2\crcr=\crcr}}}

\newcommand{\del}{\partial}

\newcommand{\Om}{\Omega}

\newcommand{\dlambda}{\,{\rm d}\lambda}

\newcommand{\dsigma}{\,{\rm d}\sigma}

\newcommand{\dtau}{\,{\rm d}\tau}

\newcommand{\ds}{\,{\rm d}s}

\newcommand{\dx}{\,{\rm d}x}
\newcommand{\dy}{\,{\rm d}y}

\renewcommand{\div}{\mbox{\rm div}\,}

\newcommand{\supp}{\mbox{\rm supp}\,}

\newcommand{\IR}{\mathbb{R}}
\newcommand{\IC}{\mathbb{C}}
\newcommand{\IN}{\mathbb{N}}

\newcommand{\R}{\mathcal{R}}
\newcommand{\D}{\mathcal{D}}
\newcommand{\HH}{\mathcal{H}}

\renewcommand{\Re}{\operatorname{Re}}

\def\eqn#1$$#2$${\begin{equation}\label#1#2\end{equation}}

\numberwithin{equation}{section}

\newtheorem{defi}{Definition}[section]
\newtheorem{thm}[defi]{Theorem}
\newtheorem{cor}[defi]{Corollary}
\newtheorem{prop}[defi]{Proposition}
\newtheorem{lem}[defi]{Lemma}
\newtheorem{rem}[defi]{Remark}
\newtheorem{assumption}[defi]{Assumption}

\def\eqn#1$$#2$${\begin{equation}\label#1#2\end{equation}}

\numberwithin{equation}{section}
\numberwithin{equation}{section}
\allowdisplaybreaks[4]

\begin{document}

\title{
\bf \large
The Stokes Operator on Exterior Domains\\ in Homogeneous Weighted Function Spaces:\\ From Weak Theory to $\mathscr H^\infty$-calculus to Fractional Domains}\BLACK
%$\mathscr H^\infty$-calculus of the Stokes Operator \\
%on Exterior Domains in Weighted Spaces}
% 
%\\  of non-localized impact\BLACK}
%\\ with non-local interaction\BLACK}
%the Muckenhoupt class} 
\author{{\normalsize
Reinhard Farwig\footnote{
Fachbereich Mathematik, Technische Universit\"at Darmstadt, Schlossgartenstr. 7, 64289 Darmstadt, \quad Germany, \texttt{farwig@mathematik.tu-darmstadt.de}} \;
and \;
Kazuyuki Tsuda\footnote{
Kyushu Sangyo University, 3-1 Matsukadai 2-chome,
Higashi-ku, Fukuoka,
813-8503 Japan, \texttt{k-tsuda@ip.}\texttt{kyusan-u.ac.jp}}   }\\[2ex]
%{\normalsize\it }
}
\date{{\normalsize Dedicated to our colleague Yoshikazu Giga}}
\maketitle
\vspace*{-4mm}
\begin{abstract}
\noindent
We consider the Stokes operator $A$ on smooth exterior domains $\Omega$ of $\IR^n$ in homogeneous Sobolev spaces $\widehat H^{\kappa,q}_w(\Omega)$ with radially symmetric Muckenhoupt weights $w\in \mathscr A_q$. A fundamental property is the existence of a bounded $\mathscr H^\infty$-calculus of the Stokes operator on weighted nonhomogeneous and homogeneous $L^q$ Sobolev spaces. This property implies the existence of uniformly bounded purely imaginary powers $A^{it}$, $t\in\IR$, and the characterization of domains of fractional powers $A^\theta$ equipped with nonhomogeneous  ($\| u\|_{L^q_w} + \|A^\theta u\|_{L^q_w}$) as well as homogeneous norm ($\|A^\theta u\|_{L^q_w}$) as complex interpolation spaces. 
%i.e., $\mathcal D(A^\theta) = [L^q_{\sigma,w},\mathcal D(A_{q,w}]_\theta$, $0<\theta<1$, 
The final aim is the identification with homogeneous spaces $\widehat{\mathcal D}((-\Delta_{q,w})^\theta) = [L^q_{w},\widehat{\mathcal D}((-\Delta_{q,w}]_\theta$ intersected by a space of solenoidal vector fields. Moreover, we obtain  weighted variational inequalities for weak solutions of the Stokes equations, \BLACK weighted $L^q$-$L^r$ decay estimates of the Stokes semigroup and $L^p$-maximal regularity on $L^q_{\sigma,w}(\Omega)$.    
\end{abstract}

\\


\section{Introduction}\label{S1}

In the analysis of the Navier-Stokes system 
$$ u_t -\Delta u + u\cdot\nabla u + \nabla p = f,\;\; \div u =0\;\textrm{ in }\Omega,\;\; u=0\;\textrm{ on  }\;\partial \Omega,\;\; u(0)=u_0$$
the Stokes operator $A = -\mathbb P\Delta$ plays a leading role. Here $\mathbb P$ is the Helmholtz projection which is known to be a bounded linear operator on various types of function spaces and domains $\Omega\subset \IR^n$. Assuming that the Stokes operator is sectorial and generates a bounded analytic semigroup $e^{-tA}$, $t\geq 0$, a mild solution of the Navier-Stokes system is given by Duhamel's formula 
$$ u(t) = e^{-tA}u_0 + \int_0^t e^{-(t-\tau)A} \mathbb P(-u\cdot\nabla u +f)(\tau)\dtau. $$
Important tools to solve this nonlinear integral equation are decay properties of the Stokes semigroup $e^{-tA}$ and in particular the theory of fractional powers $A^\theta$, $-1\leq \theta\leq 1$. Therefore, the domains $\D(A^\theta)$ must be compared to each other and to more classical functions spaces of Sobolev type, {\em e.g.} $H^{2\theta,q}(\Omega)$. 

The background for these tools is the existence of bounded imaginary powers $A^{is}$, $s\in\IR$, (property $BIP$) so that complex interpolation can be applied to the family of spaces $\D(A^\theta)$. Instead of defining $A^{is}$ as a limit of fractional powers $A^z$, $z\to is$, a powerful and much more flexible method is the bounded $\mathscr H^\infty$-calculus which allows to define $f(A)$ for a large class of bounded analytic functions $f$ as in Dunford's calculus, but on infinite contours with possible singularites at $\lambda=0$ and as $|\lambda|\to \infty$.
Given the property $BIP$ the domain $\D(A^\theta)$ is shown to equal the intersection $\D((-\Delta)^{2\theta}) \cap L^q_\sigma(\Omega)$, {\em i.e.} the intersection of the fractional Sobolev space $H^{2\theta,q}_0(\Omega)$ with the $L^q$ space of solenoidal vector fields.

In this approach the property $\lambda=0\in\rho(A)$, {\em i.e.,} $A$ has a bounded inverse, is very helpful and satisfied for bounded domains. However, for exterior domains, $0\in\sigma(A)$  and homogeneous Sobolev spaces such as $\widehat H^{1,q}(\Omega) = \{u\in H^{1,q}_{\rm loc}(\Omega): \nabla u\in L^q\}$ equipped with the norm $\|\nabla u\|_{L^q}$ naturally occur, but are more difficult to handle in estimates of both perturbation and nonlinear terms. A typical tool is the embedding 
$\widehat H^{1,q}(\Omega) \hookrightarrow L^{q^*}(\Omega)$ which holds with $q^*=\frac{nq}{n-q}$ when $q<n$. By analogy, $\widehat H^{2,q}(\Omega) \hookrightarrow L^{q^{**}}$ with $q^{**}=\frac{nq}{n-2q}$ which requires even that $q<\frac{n}{2}$. These conditions are very restrictive in estimates of the nonlinear term $u\cdot\nabla u$.

A possibility to gain more flexibility is the use of weighted function spaces. 
For an exterior domain a natural choice are radially symmetric power weights $\langle x\rangle^\ell = (1+|x|^2)^{\ell/2}$ and the weighted Lebesgue space $L^q_\ell(\IR^n)$ with norm $\|u\|_{L^q_\ell} = \big(\int |u|^q \langle x\rangle^{\ell q} \dx\big)^{1/q} < \infty$. 
By the $\mathscr A_q$ Muckenhoupt condition there must hold $-\frac{n}{q} < \ell <\frac{n}{q'}$. For the embeddings $\widehat H^{1,q}_\ell(\IR^n) \hookrightarrow L^{q^*}_\ell$ and $\widehat H^{2,q}_\ell(\IR^n) \hookrightarrow L^{q^{**}}_\ell$
the new restrictions are $1-\frac{n}{q} < \ell <\frac{n}{q'}$ and $2-\frac{n}{q} < \ell <\frac{n}{q'}$, respectively. 
Obviously, depending on $\ell$, the power $q$ is allowed to be larger than $n$ and even to be arbitrarily large whereas for $\ell=0$, {\em i.e.} no weight, $q<n$ and $q<\frac{n}{2}$ recur. 
We refer to \cite{Farwig-Tsuda-stab} for a recent application in the analysis of the Navier-Stokes equations about stability of stationary solutions in unbounded domains. Moreover, 
we can solve the periodically in time moving boundary  problem of viscous fluid flow in exterior domains, see \cite{Farwig-Tsuda-periodic-exterior-moving}.

\BLACK 

In view of these advantages it is our aim to analyze the Stokes operator $A$ on weighted homogeneous Sobolev spaces $\widehat H^{k,q}_\ell(\Omega)$ for an exterior domain. Moreover, the domain $\D(A)$ which is equipped with the graph norm $\|u\|_{L^q_\ell} + \|Au\|_{L^q_\ell}$ is to be changed to a homogeneous domain $\widehat \D(A)$, the closure of $\D(A)$ with respect to the norm $\|Au\|_{L^q_\ell}$.
For this setting we prove the existence of a bounded $\mathscr H^\infty$-calculus and the property $BIP$ to characterize homogeneous domains $\widehat\D(A^\kappa)$ by complex interpolation and embeddings into weighted Sobolev spaces. 
Finally we apply these results to the Stokes semigroup to derive spatially weighted $L^q$-$L^r$ decay estimates and $L^p$-maximal regularity. 

Let us recall several important results and ideas of proofs which are important for our aims. 
A relatively easy approach to the Helmholtz projection $ \mathbb P$ due to Simader and Sohr \cite{Sim_Sohr92} is based on a variational inequality for gradients to get the existence of the pressure term $\nabla \pi=(I-\mathbb P)u$ defined by a weak Neumann problem. 
That approach was generalized by Sohr and the first author of this article \cite{FS} to get the Helmholtz projection also in weighted $L^q$ spaces. 
The variational approach will be important also for a weak theory of the Stokes system in exterior domains with Dirichlet data, see Borchers and Miyakawa \cite [Sect. 3]{BoMi90},  Kozono and Sohr \cite{KoSo92}, Simader, Sohr and the first author of this article\cite{FaSiSo93} in $L^q$ \BLACK and Proposition \ref{lem-var-ineq-nabla_Omega0} below in weighted $L^q$ spaces.

The next important tool are resolvent estimates of the Stokes operator. 
Giga \cite{Giga81} proved the sectoriality and resolvent estimates of the Stokes operator in  bounded domains via the theory of pseudo-differential operators. 
The case of exterior domains in $\IR^n$ was solved by Solonnikov \cite{Sol77} when $|\lambda|\geq \epsilon>0$, see also \cite{Giga81}.
The limit case $\lambda\to 0$ for $n=3$ was included by Borchers and Sohr \cite{BoSo87}, and the case $n=2$ was solved by Borchers and Varnhorn \cite{BoVa93}. 
Another approach via multiplier theory and cut-off techniques was used by Sohr and the first author of this article in \cite{FS94}; the result covers bounded, perturbed half spaces and exterior domains. 
In particular for exterior domains the resolvent estimate was extended in \cite{FS} to $L^q$ spaces with radially symmetric weights. 

For sharp results the Stokes operator should possess the property $BIP$ of bounded imaginary powers $A^{is}\in \mathcal L(L^q_\sigma(\Omega))$, $s\in\IR$; 
here $L^q_\sigma(\Omega)$ denotes the $L^q$ space of solenoidal vector fields with vanishing normal component on the boundary $\partial\Omega$, and $\mathcal L$ denotes the set of linear bounded operators. 
By complex interpolation $BIP$ implies that 
\begin{equation}\label{DAtheta-interp}
 [L^q_\sigma(\Omega),\D(A)]_\theta = \D(A^\theta).
\end{equation}
Then, by a retraction-coretraction argument, the domain $\D(A^\theta)$ can be identified with a solenoidal subset of the fractional Sobolev space $H^{2\theta,q} \sim \D((-\Delta)^\theta)$ for $0<\theta<1$. 
The first results on $BIP$ for the Stokes operator on bounded domains are due to Giga \cite{Giga85} using the theory of pseudo-differential operators. 
The exterior domain case was considered by Giga-Sohr \cite{Giga-Sohr89}. 
A crucial tool is the comparison of the Stokes resolvent $(\lambda-A)^{-1}$ with the whole space resolvent $(\lambda-A_{\IR^n})^{-1}$ for $|\lambda|\to\infty$ and $\lambda\to 0$. 
Another tool is the representation of complex powers $A^z$ by contour integrals which do not converge in the operator norm, but hold pointwise on a dense subset of $L^q_\sigma(\Omega)$. 
Similar methods for exterior domains were used by Giga and Sohr in \cite{Giga-Sohr91} with their focus put on maximal regularity estimates proved via the theory of Dore and Venni. 
A further aim is to get estimates not only in $\D(A^\theta)$, $0<\theta<1$, but also in homogeneous spaces $\widehat\D(A^\theta)$.
%defined as the closure of $\D(A)$ with respect to the norm in $\D(A^\theta)$.
Although \eqref{DAtheta-interp} holds for homogeneous spaces as well, the general analysis for $\IR^n$ or an exterior domain is much more involved since $\widehat\D(A^\theta)$ is not contained in $L^q_\sigma(\Omega)$.

As an even stronger property the Stokes operator possesses a bounded $\mathscr H^\infty$-calculus. This property can be proved easily by multiplier theory for the whole space $\IR^n$; for the half space we refer to Desch, Hieber and Pr\"u{\ss} \cite{DHP01}. 
This property for bounded, bent half spaces and exterior domains was proved by Noll and Saal \cite{NollSaal}; their proof is based on ideas of \cite{Giga85} and \cite{DHP01} exploiting special techniques in half spaces during a cut off process. 
The advantage of the $\mathscr H^\infty$-calculus is to allow perturbation arguments with do not work for the property $BIP$. 
Moreover, under further assumptions (property $\alpha$) which are fulfilled for $L^q_\sigma$ spaces, the bounded $\mathscr H^\infty$-calculus implies the so-called $\mathcal R$-bounded $\mathscr H^\infty$-calculus and hence maximal regularity, see \cite{KuWe04}, \cite[p. 652]{NollSaal}.

The difficulties of homogeneous spaces $\widehat \D(A^\theta)$ or $\widehat H^{\kappa,q}(\Omega)$ occur when multiplication with further functions, {\em e.g.} via a partition of unity, or perturbation arguments are exploited. 
By definition, elements in homogeneous spaces are defined as equivalence classes and - on the whole space - unique up to polynomials of any degree, {\em cf.} the monograph by Triebel \cite[Sect. 5]{Triebel2010} and Bergh and L\"ofstr\"om \cite[Sect. 6.3, 6.4, 6.5]{BerghL}. 
To control the degree of those polynomials restrictions depending on the dimension $n$ must be put on exponents $\kappa,q$ and, if available, on weight functions. 
Another idea is to exclude polynomials completely, but to require that the spaces are Banach spaces; this restricts the class of admissible spaces a lot, see {\em e.g.} the monograph by Bahouri, Chemin and Danchin \cite{BCD} in which the set of tempered distributions is strongly controlled for frequencies $\xi$ as $|\xi|\to 0$, see also \cite[Chapter 3]{DHMT}. 
Similar problems and strategies are used for homogeneous Besov spaces $\widehat B^{\kappa}_{q,r}$. For more concrete results we refer to Gaudin
\cite{Gau-Tunis}, \cite{Gau-lip} and Danchin, Hieber, Mucha and Tolksdorf \cite{DHMT} who discuss for $\IR^n$ and $\IR_+^n$ homogeneous Sobolev and Besov spaces, subspaces of solenoidal vector fields, and interrelations by real and complex interpolation. 
In \cite{Gau-Tunis} Gaudin examines real and complex interpolation, duality, density  and trace results for homogeneous Sobolev spaces. %, whereas in  \cite{Gau-JEE} the focus is put on maximal regularity results including weights in time.
In \cite{Gau-lip} the author considers these questions on Lipschitz domains of half space type. 
The aim of \cite{DHMT} is the analysis of the Stokes operator with Neumann boundary condition, {\em i.e.}, the normal component of the stress tensor, $T(u,\pi)\cdot \textsl{n}$, vanishes on $\partial\IR_+^n$, and its maximal regularity. 
Special focus is put on $L^1$-maximal regularity which does not hold in reflexive spaces so that {\em e.g.} endpoint homogeneous Besov spaces with third parameter $1$ are to be considered.  
In particular, in the theory of inhomogeneous
Navier-Stokes equations and of free surface flow, the problem of $L^1$-maximal regularity occurs; for an analysis in homogeneous Besov spaces with third endpoint parameter we also refer to \cite{DanMu09},  \cite{Shimizu-O}. \BLACK

The first main goal of this article is to prove that the Stokes operator on smooth exterior domains possesses a bounded $\mathscr H^\infty$-calculus and hence the property $BIP$ in weighted spaces. The weighted Lebesgue spaces $L^q_\ell(\Omega)$ are equipped 
 with the norm 
 $$ \|u\|_{L^q_\ell} = \Big(\int_\Omega |u(x)|^q \,\langle x\rangle ^{q\ell} \dx \Big)^{1/q} <\infty $$ 
where $1<q< \infty$, and the weight satisfies the Muckenhoupt $\mathscr A_q$ condition 
$$ -\frac{n}{q} < \ell < \frac{n}{q'}\,.$$
The set $\HH_0(\Sigma_\phi)$ below is a set of bounded holomorphic functions on a sector $\Sigma_\phi\subset\IC$ with $|\arg\lambda|<\phi$ for $\lambda\in \Sigma_\phi$; for details see Subsect. \ref{S2.4} below.

\begin{thm}\label{theorem-H-infty-all}
Let $\Omega\subset\IR^n$, $n\geq 3$, be an exterior domain with boundary of class $C^3$. 
Further let $1<q< \infty$, and $A=A_{q,\ell} = -\mathbb P \Delta$ be the Stokes operator on $L^q_{\sigma, \ell }(\Omega)$ for $-\frac{n}{q} < \ell < \frac{n}{q'}$.

(i) There exists a constant $C=C(q,\ell,\phi)>0$ such that for any $\phi\in (0,\frac{\pi}{2})$, $f \in L^q_{\sigma, \ell }(\Omega)$ and all $h\in \HH_0(\Sigma_\phi)$
 \begin{align}\label{H-infty-all}
\|h(A)\|_{\mathcal{L}(L^q_{\sigma,\ell}(\Omega))} \leq C |h|_{\infty,\phi\BLACK}.    
\end{align}
In particular, $A$ possesses a bounded $\mathscr H^\infty$-calculus on $L^q_{\sigma.\ell}(\Omega)$ with $\mathscr H^\infty$-angle $\Phi_ {A}^\infty=0$. 

(ii) The Stokes operator has the property $BIP$ with power angle $\Theta_A=0$. Thus  for any $\theta>0$ there exist a constant $C(\theta)$ such that 
$$ \|A^{it}\|_{\mathcal L(L^q_{\sigma,\ell}(\Omega))} \leq C e^{\theta|t|},\quad t\in\IR. $$
\end{thm}
\BLACK

For details we refer to Theorem \ref{theorem-H-infty} and Corollary \ref{theorem-H-infty-2}. The problem is split into one for the whole space and one on bounded domains for which there exists a bounded $\mathscr H^\infty$-calculus due to Noll and Sall \cite{NollSaal}. 
For the whole space we use multiplier theory on $L^q_\ell(\IR^n)$ and the residue theorem. Perturbation terms are treated as in \cite{NollSaal} involving the Bogovski\u{\i} operator and special estimates of the pressure. 

The second goal concerns the characterization of homogeneous domains $\widehat \D(A^\theta)$ by the fractional Sobolev space $\widehat H^{2\theta,q}_0(\Omega)$ "intersected" by the condition $\div u=0$. 
Note that the seminal result 
$$ \D(A^\theta) = [L^q_\sigma(\Omega), \D(A)]_\theta = [L^q(\Omega), \D(-\Delta)]_\theta \cap L^q_\sigma(\Omega), $$ 
see \cite[Theorem 3]{Giga85} is not meaningful in the context of homogeneous spaces, since $\widehat\D(A^\theta) \not\subset L^q_\sigma(\Omega)$.
Moreover, the retraction-coretraction method based on the retraction $\mathcal R = A^{-1} \mathbb P(-\Delta)$ fails to work for pairs $q,\ell$ of interest since the formal adjoint $(-\Delta)\mathbb P A^{-1}$ can not be shown for $n\leq 4$ to be bounded for the conjugate pair $q',-\ell$. 
Indeed the analysis in homogeneous spaces requires restrictions, as the classical condition $q<\frac{n}{2}$ for a twofold Sobolev embedding. %namely $2-\frac{n}{q} < \ell < \frac{n}{q'}$. 
To be more precise, the conditions $2-\frac{n}{q} < \ell < \frac{n}{q'}$ and in the dual setting that $2-\frac{n}{q'} < -\ell < \frac{n}{(q')'} = \frac{n}{q}$ imply that $n>4$. 
Therefore, we follow ideas of Borchers and Miyakawa \cite{BoMi90},  restrict ourselves to $\theta \leq \frac12$ 
%\BLUE (I think that we need $\theta \leq \frac12$ since  Theorems \ref{theorem-H-infty}, \ref{angles},  Proposition \ref{Complex Interpolation} (i) are not enough to prove Prop. \ref{domain-of-fractional-power}.
%Indeed, on the BIP on the homogeneous domains, we need Cor. \ref{domains:A1/2-nabla_A1/2}, and thus $\theta \leq 1/2$. ) 
and base important estimates on variational inequalities in weighted homogeneous Sobolev spaces. 
The proof needs a cut off technique as in \cite{BoMi90}, \cite{KoSo92}.

\begin{thm}\label{A1/2-nabla_A1/2-all}
(i) For $1<q<\infty$ and $-\frac{n}{q}<\ell<\frac{n}{q'}$ there holds
\begin{equation}\label{A1/2-nabla-all}
    \big\|A_{q,\ell}^{1/2}\big\|_{L^q_\ell(\Omega)} \leq C\|\nabla u \|_{L^q_\ell(\Omega)}, \quad u\in \widehat H^{1,q}_{\sigma,\ell}(\Omega),
\end{equation}
with a constant $C>0$ independent of $u$.

(ii) Let $1<q<\infty$ and $1-\frac{n}{q} <\ell< \frac{n}{q'}$. Then
\begin{equation}\label{nabla-A1/2-all}
    \|\nabla u \|_{L^q_\ell(\Omega)} \leq C\big\|A_{q,\ell}^{1/2} u\big\|_{L^q_\ell(\Omega)}, \quad u\in \widehat\D(A_{q,\ell}^{1/2}),
    \end{equation}
with $C>0$ independent of $u$.
\end{thm}

Note that the restriction $1-\frac{n}{q} <\ell< \frac{n}{q'}$ in Theorem \ref{A1/2-nabla_A1/2-all} (ii) corresponds to the constraint $q<n$ when $\ell=0$. The proof of Theorem \ref{A1/2-nabla_A1/2-all} is based on the variational inequality 
\begin{align}\label{var-ineq-nabla-all}
  \|\nabla u\|_{L^q_\ell} \leq C\sup_{v\neq 0} \frac{\big|\langle \nabla u,\nabla v\rangle\big|}{\|\nabla v\|_{L^{q'}_{-\ell}}} \leq C\|\nabla u\|_{L^q_\ell},\quad  u\in \widehat \D(A_{q\ell}),
\end{align}
where $v\neq 0$ is running through $\widehat H^{1,q'}_{\sigma,-\ell,0}(\Omega)$, for details see   Proposition \ref{lem-var-ineq-nabla_Omega0} and Theorem \ref{var-ineq-sol} below. Whereas Proposition \ref{lem-var-ineq-nabla_Omega0} considers {\em a priori} estimates of the weak Stokes system on $\Omega$ yielding the variational inequality 
\eqref{var-ineq-nabla-all} with test functions $v\in \widehat H^{1,q'}_{-\ell,0}(\Omega)$,  Theorem \ref{var-ineq-sol} improves that result to \eqref{var-ineq-nabla-all} with $v\in \widehat H^{1,q'}_{\sigma,-\ell,0}(\Omega)$ via de Rham's theory. 
%The crucial point in \eqref{var-ineq-nabla-all} is the fact that $v$ is restricted to solenoidal vector fields. To get this improvement compared to allow all $v\in \widehat H^{1,q'}_{-\ell,0}(\Omega)$ we construct in Proposition \ref{Bog-ext} a Bogovski\u{\i} operator $\mathbb B$ for exterior domains on weighted spaces, {\em i.e.,} $\mathbb B: L^q_\ell(\Omega) \to \widehat H^{1,q}_{\ell,0}(\Omega)$,  $\div \mathbb Bg=g$ for $g\in L^q_\ell(\Omega)$, and 
%$$ \|\mathbb Bg\|_{\widehat H^{1,q}_{\ell,0}(\Omega)} \leq C \|g\|_{L^q_\ell(\Omega)}. $$  
%For the construction we extend ideas from the monograph of Galdi \cite{Galdi-steady} to weighted spaces.

Now the main characterization of fractional domains $\widehat \D(A^\theta)=[L^q_{\sigma,\ell}(\Omega),\widehat \D(A)]_\theta$ reads as follows:

\begin{prop}\label{domain-of-fractional-power-all} 
 Let $1< q\leq r<\infty$, let $2\theta-\frac{n}{q} < \ell < \frac{n}{q'}$ with  $0<\theta<\frac12$ \BLACK and assume that
$$ \ell-\ell' = 2\theta + \frac{n}{r} -\frac{n}{q}\geq 0.
$$
%\RED OLD \BLACK Then, for $0\leq \theta \leq 1$, there holds with $r$ defined by $\frac{1}{r} = \frac{1}{q} - \frac{2\theta}{n}$, the identity
%Then, for $0\leq \theta \leq 1$, there holds with $p$ defined by $\frac{1}{p} = \frac{1}{q} - \frac{2\theta}{n}$, the identity
 Then the Stokes operator $A=A_{q,\ell} = -\mathbb P \Delta$ on $L^q_{\sigma, \ell }(\Omega)$ satisfies  
\begin{align*} 
\widehat{\D}(A^\theta) = [L^q_{\sigma,\ell}, \widehat{\D}(A)]_{\theta} =  [L^q_{\ell}, \widehat{\D}(-\Delta_{q,\ell})]_\theta \cap L^r_{\sigma,\ell'}(\Omega) = \widehat{\D}((-\Delta_{q,\ell})^\theta) \cap L^r_{\sigma,\ell'}(\Omega) 
\end{align*}  
with the norm equivalence 
\begin{equation*} 
 \|A^{\theta} x\|_{L^q_{\ell}(\Omega)} \simeq \|x\|_{[L^q_{\sigma, \ell}, \widehat{\D}(A)]_{\theta} } \simeq \|x\|_{\widehat{H}^{2\theta,q}_{\ell}(\Omega)}. 
\end{equation*}
\end{prop}

 As application we mention the $L^q_\ell$-$L^r_{\ell'}$ decay of the Stokes semigroup $e^{- tA_{q,\ell}}$ and the $L^p$-maximal regularity of $A_{q,\ell}$.
\BLACK

\begin{cor}\label{Lp-Lq At-all} 
Let $1< q\leq r<\infty$, let $\kappa-\frac{n}{q} < \ell < \frac{n}{q'}$  with  $\kappa\in(0,n)$, %(see Cor. 3.12) 
\BLACK %(By Cor. 3.12 this condition is not needed, see also the proof of Cor. 3.12)
and assume that
$$ \ell-\ell' = \kappa + \frac{n}{r} -\frac{n}{q}\geq 0.
$$
Then there holds the weighted $L^q$-$L^r$-estimate
$$ \|e^{- tA_{q,\ell}}u\|_{L^r_{\ell'}(\Omega)} \leq C t^{-\frac{n}{2}\big(\frac{1}{q}-\frac{1}{r}\big) -\frac{\ell-\ell'}{2}} \|u\|_{L^q_{\ell}(\Omega)}, \quad u\in L^q_{\sigma,\ell }(\Omega),$$
with a constant $C>0$ independent of $t>0$. 
\end{cor}

%\vspace*{1ex}

%\begin{thm}\label{domain-of-fractional-power-all} \RED (not yet correct) \BLACK 
%Let $1<q< \infty$ with $3 \leq n$ and $A_q = -\mathbb P_q \Delta$ be the Stokes operator on $L^q_{\sigma, \ell }(\Omega)$ for $2-\frac{n}{q} < \ell < \frac{n}{q'}$ with $0\leq \ell$. 
%For $0<\theta < \frac{1}{2q}$
%\begin{align}\label{BIP-homog-ext-domain-all}
%\widehat{\D}(A_q^\theta) = [L^q_{\sigma, \ell}, \widehat{\D}(A_q)]_{\theta} \RED = [L^q_{\sigma, \ell}, \widehat{\D}(\Delta_\Omega)]_\theta = \widehat{H}^{2\theta}_{q,\ell}(\Omega) \cap L^q_{\sigma,\ell}(\Omega)\BLACK 
%\end{align}
%\end{thm}

\begin{prop}\label{maxreg-all}
Let $1<q< \infty$ and  $-\frac{n}{q} < \ell < \frac{n}{q'}$. Then $A_{q,\ell}$ possesses for each $1<p<\infty$ maximal $L^p$-regularity on $L^q_{\sigma,\ell}(\Omega)$, {\em i.e.,} for $f\in L^p(\IR_+;L^q_{\sigma,\ell}(\Omega))$ the instationary Stokes system $u_t -\Delta u + \nabla\pi = f,\; \div u=0,\; u(0)=0,\; u\big|_{\partial\Omega} = 0$ possesses a unique solution $(u,\nabla\pi)$ such that 
    \begin{align*} 
     \|u_t\|_{L^p(\IR_+;L^q_\ell)}  + \|A_{q,\ell}u\|_{L^p(\IR_+;L^q_\ell)}  
     \leq C\|f\|_{L^p(\IR_+;L^q_\ell)}.
    \end{align*}
Under the assumption  $2-\frac{n}{q} < \ell < \frac{n}{q'}$ also the estimate
    \begin{align*} 
    \|\nabla^2 u\|_{L^p(\IR_+;L^q_\ell)}  + \|\nabla\pi\|_{L^p(\IR_+;L^q_\ell)} & \leq C\|f\|_{L^p(\IR_+;L^q_\ell)} 
    \end{align*}
holds. In case of the whole space both estimates hold for all $-\frac{n}{q} < \ell < \frac{n}{q'}$.
\end{prop}

The article is structured as follows. In Sect.~2 we introduce notation, several functions spaces with and without weights, subspaces of solenoidal vector fields, recall important results from interpolation theory, applied to classical operators (multiplier, Calder\'on-Zygmund, Helmholtz projection, Stokes operators extended to weighted spaces, {\em etc.}). 
Further, in Subsect.~\ref{S2.2} we construct in Proposition \ref{Bog-ext} a Bogovski\u{\i} operator on exterior domains in weighted homogeneous spaces. 
Subsect.~\ref{S2.4} discusses for the Stokes operator on the whole space the bounded $\mathscr H^\infty$-calculus, its property $BIP$ and consequences of complex interpolation. 
The weak Stokes system on exterior domains in weighted homogeneous spaces in Subsect.~\ref{S3.1} is an important tool since it yields a variational inequality for gradients of solenoidal fields. 
Then Subsect.~\ref{S3.2} on the bounded $\mathscr H^\infty$-calculus of the Stokes operator in weighted $L^q$ spaces is central for all further results. For applications, {\em e.g.} the characterization of the fractional domain $\widehat\D(A^\theta)$, we need that norms based on $\nabla u$ are similar to norms of $A^{1/2}$ on weighted spaces. 
 This result, obvious in $L^2$, is extended in Subsect.~\ref{S3.3} to weighted spaces $\widehat H^{1,q}_{\sigma,\ell,0}(\Omega)$, see Theorem \ref{A1/2-nabla_A1/2}. It is \BLACK the starting point for the full characterization of $\widehat \D(A^\theta)$ in Subsect.~\ref{S3.4} where also further fundamental estimates on embeddings and decay of the Stokes semigroup will be proved.\\

\section{Preliminaries}\label{S2}

\subsection{Basic function spaces, weights, and operators}\label{S2.1}

{\bf Preliminaries:}
Exterior domains in $\IR^n$ are denoted by $\Omega$. Without loss of generality we assume that $0\in \overline{\Omega}^c$ where $\Omega^c$ is the complement of $\Omega$ in $\IR^n$. 
Let $R>0$ denote a radius such that $\Omega^c\subset B_R$ where  $B_R=B_R(0)$ is the open ball with center $0$ and radius $R$. Set $\Omega_R := \Omega\cap B_R$.
Then 
$L^q(\Omega)$, $1\leq q \leq \infty$, denotes the usual Lebesgue space with norm $\|\cdot\|_{L^q(\Omega)} = \|\cdot\|_{L^q}$.
Note that we frequently omit the symbol $\Omega$ in $L^q$, $\|\cdot\|_{L^q}$, {\em etc.}, of the underlying domain when it is known from the context. 
Moreover, we omit the subscript $n$ of $L^q(\Omega)^n$ to denote a corresponding space of vector fields; the same applies to matrix-valued fields.
In addition, $H^{k,q}(\Omega)$ is the standard Sobolev space of order $k\in\IN_0$ and with norm $\|\cdot\|_{H^{k,q}(\Omega)}$.
We denote the spaces of test functions and of solenoidal test functions by 
$C^{\infty}_0(\Omega)$  and $C^{\infty}_{0,\sigma}(\Omega):=\{\varphi \in C^{\infty}_{0}(\Omega): \div \varphi=0\}$, respectively. Moreover,
$L^q_{\sigma}(\Omega):= \overline{C^{\infty}_{0,\sigma}(\Omega)}^{\|\cdot\|_{L^q(\Omega)}}$, $1<q<\infty$, is the $L^q$ space of weakly solenoidal vector fields, $u$, with vanishing normal component, $u\cdot \textsl{n}$, on $\partial\Om$.

%Exterior domains in $\IR^n$ are denoted by $\Omega$. Without loss of generality we assume that $0\in \overline{\Omega}_0^c$. Let $R>0$ denote a radius such that $\Omega^c\subset B_R$ where $\Omega^c$ denotes the complement of $\Omega$ in $\IR^n$ and $B_R=B_r(0)$ denotes the open ball with center $0$ and radius $R$. Then $\Omega_R = \Omega\cap B_R$.  

For general Banach spaces $X,\,\|\cdot\|_X$ and $Y,\,\|\cdot\|_Y$ and a linear operator $A:\D(A)\subset X \to Y$ let $\D(A)$ denote the domain of $A$, $\mathcal R(A)$ and $\mathcal N(A)$ its range and null space, respectively. Further, $\rho(A)$ and $\sigma(A)$ are its resolvent set and spectrum, respectively. For $u\in \D(A)$ the graph norm $\|u\|_A$ is defined by $\|u\|_X + \|Au\|_X$. \BLACK The set of bounded linear operators from $X$ to $Y$ is $\mathcal L=\mathcal L(X,Y)$, its norm, $\|\cdot\|_{\mathcal L(X,Y)}$, is also denoted simply by $\|\cdot\|$, if no confusion can occur. The dual space to $X$ is denoted by $X^*$.

\vspace*{2ex} 
\noindent
{\bf Muckenhoupt weights:}
We define the Muckenhoupt class and weighted Lebesgue spaces as follows.

\begin{defi}\label{Muckenhoupt-class}
Let $1<q <\infty$.

(i) A weight function $0 \leq w \in L^1_{\rm loc}(\mathbb{R}^n)$ belongs to the Muckenhoupt class $\mathscr{A}_q(\IR^n) $ if $w$ satisfies 
$$
\sup_Q \bigg(\frac{1}{|Q|}\displaystyle\int_Q w \dx\bigg)\bigg(\frac{1}{|Q|}\displaystyle\int_Q w^{-1/(q-1)}\dx\bigg)^{q-1} \leq C <\infty
$$
for all cubes $Q \subset \mathbb{R}^n$, where $|Q|$ denotes the Lebesgue measure of $Q$. 

(ii) The weighted $L^q$ space with weight $w\in \mathscr{A}_q(\IR^n)$ is defined by 
$$ \mathbb{L}^q_w(\Omega) = \Bigg\{u\in L^1_{\rm loc}(\Omega): \|u\|_{\mathbb L^q_w(\Omega)} = \Big(\int_\Omega |u|^q w(x)\dx \Big)^{1/q}<\infty \Bigg\}. $$

(iii) For weight functions of radially symmetric type $w(x) = \langle x\rangle^{q\ell} = \big(1+|x|^2\big)^{ q\ell/2}$ we introduce the weighted space 
$$
L^q_\ell (\Omega) := \bigg\{ u \in L^1_{\rm  loc} (\bar{\Omega}): \|u\|_{L^q_\ell (\Omega)}= \Big(\displaystyle\int_{\Omega} |u|^q \langle x\rangle^{q\ell} \dx \Big)^{1/q} <\infty\bigg\}.  
$$
\end{defi}

Note that in Definition \ref{Muckenhoupt-class} (iii)
we replace the index $w$ by $\ell$ and use the weight $\langle x\rangle^{q\ell}$ rather than $\langle x\rangle^{\ell}$ in the above integrand. 
Similarly, weighted Sobolev spaces $\mathbb H^{k,p}_{w}$, $H^{k,p}_{\ell}$, {\em etc.}~will be defined below. 
We recall that 
$$ w =\langle x\rangle^{q\ell}\BLACK \in \mathscr A_q(\IR^n) \quad\textrm{ if and only if} \quad
-\frac{n}{q} < \ell < n\Big(1-\frac{1}{q}\Big) = \frac{n}{q'};$$  
see {\em e.g.} \cite[Lemma 2.3]{FS}. Since the focus of this article is on exterior domains $\Omega
\subset \IR^n$ and weights with singularities at a finite point in $\overline \Omega$ are not adequate, we pose the following simplifying assumption for weights on $\Omega$:

\begin{assumption}\label{weight-Omega} 
For any weight $w\in \mathscr A_q(\IR^n)$ there holds $w(x) \sim 1$ in an open neighborhood of $\Omega^c$. In other words, there exists $R>0$ and $c_0>0$ such that $\Omega^c \subset B_R$ and $1/c_0 < w(x) < c_0$ in $\overline  B_R$. Moreover, let $\mathscr A_q = \mathscr A_q(\IR^n)$.
\end{assumption}

\vspace*{1ex}

We prepare several fundamental estimates on weighted spaces. Especially, the following operator estimates play an important role.

\begin{prop}\label{thm-hoermander}
Let $1<q<\infty$ and $w\in \mathscr A_q $. 

(i) Calder\'on-Zygmund operators are bounded from $\mathbb L^q_w(\IR^n) \to \mathbb L^q_w(\IR^n)$.

(ii) Hörmander-Mikhlin operators defined by a multiplier function $m$ such that 
$$ \sup_{0\neq \xi\in\IR^n} |\xi|^{|\alpha |} |\nabla_\xi^\alpha m(\xi)| \leq C<\infty $$
for all multi-indices $\alpha\in\IN_0^n$ with $|\alpha|\leq n$ are bounded from $ \mathbb L^q_w(\IR^n) \to \mathbb L^q_w(\IR^n)$.
\end{prop}

For (i) we refer to \cite[Theorem IV 3.1]{G-C and RdF}. (ii) is found in \cite[Theorem 2]{KurtzWheeden} or \cite[Theorem IV 3.9]{G-C and RdF}.

\vspace{2ex}

\noindent
{\bf Homogeneous Sobolev spaces:}
In view of the unboundedness of domains we have to deal with homogeneous function spaces, their density  and interpolation properties. For the whole space $\IR^n$ homogeneous Riesz potential spaces (Sobolev spaces) $\widehat{\mathbb{H}}^{\kappa,q}_w(\IR^n)$ are defined as distributions in the space $Z'(\IR^n) = \mathcal S'(\IR^n)/\Pi(\IR^n)$ where $\mathcal S'(\IR^n)$ denotes the space of Schwartz' tempered distributions and $\Pi(\IR^n)$ the set of all polynomials over $\IR^n$. Then  for any $\kappa\in\IR$, $1<q<\infty$ and $w\in\mathscr A_q$
$$
\widehat{\mathbb{H}}^{\kappa,q}_w(\IR^n) = \Big\{f\in Z'(\IR^n): \|f\|_{\widehat H^{\kappa,q}_w} = \Big(\int_{\IR^n} \big|\mathcal F^{-1}(|\xi|^\kappa \hat f(\xi))(x)\big|^q \,w(x) \dx \Big)^{1/q} <\infty\Big\}. 
$$
For more details we refer to \cite[Chapter 5]{Triebel2010} where these spaces are introduced  - without weights - as the class $\dot F^{\kappa;q,2}_w(\IR^n)$. 
If $\kappa=k\in\IN$,  Proposition \ref{thm-hoermander} shows that $\widehat{\mathbb{H}}^{k,q}_w(\IR^n)$ coincides with the classical homogeneous Sobolev space of functions $f$ such that - with equivalent norms - $\sum_{|\alpha|=k} \| \nabla^\alpha f\|_{q,w}<\infty$. 
The drawback of the spaces 
$\widehat{\mathbb{H}}^{\kappa,q}_w(\IR^n)$  
is the fact that elements are uniquely determined only up to polynomials. 
However, by embedding theorems it can be shown that in the case $1<q<n$, $\kappa=1$, and $w=1$ elements of $\widehat{\mathbb{H}}^{\kappa,q}(\IR^n)$ can be identified with functions in $L^{q^*}(\IR^n)$, $q^*=nq/(n-q)$; for an analogue in weighted spaces see Propositions \ref{embed-ext-weight} and \ref{Hardy-Rellich} below. 
In view of interpolation of spaces $L^r_v(\IR^n)$ and $\widehat{\mathbb{H}}^{\kappa,q}_w(\IR^n)$ we identify $L^r_v(\IR^n)$, $1\leq r<\infty$, with $(L^r_v(\IR^n) + \Pi(\IR^n))/\Pi(\IR^n)$.

Similar results hold for homogeneous spaces defined on an  exterior domain $\Omega \subset\IR^n$. Recall the definition
$$ \widehat{\mathbb{H}}^{k,q}_w(\Omega) := \big\{u: u=v\big|_{\Omega},\, v\in \widehat{\mathbb{H}}^{k,q}_w(\IR^n)\big\}, \quad 
\|u\|_{\widehat{\mathbb{H}}^{k,q}_w(\Omega)}: = \inf_{v,\, u=v|_{\Omega}} \|v\|_{\widehat{\mathbb{H}}^{k,q}_w(\IR^n)}, $$
for $k\in\IN$ and, in case of vanishing boundary values,
$$  \widehat{\mathbb{H}}^{k,q}_{w,0}(\Omega) := \overline{ C^\infty_0(\Omega)}^{\widehat{\mathbb H}^{k,q}_{w}(\Omega)}. $$
Moreover,  for $\kappa=k+\theta$ where $k\in \IN_0$ and $0<\theta<1$ let 
\begin{equation*}
\widehat{\mathbb H}^{\theta,q}_{w,0}(\Omega) := [\widehat{\mathbb H}^{k,q}_{w,0}(\Omega), \widehat{\mathbb H}^{k+1,q}_{w,0}(\Omega)]_\theta.
\end{equation*}  
%
%
%\RED ((Note: In the definition above the norm is defined by the infimum of {\bf all} suitable extensions from $\Omega$ to $\IR^n$. Another possibility is to use the norm of $\widehat{\mathbb{H}}^{k,q}_w(\IR^n)$ where first of all $u\in C^\infty_0(\Omega)$ is simply extended by $0$ to $\IR^n$.  Obviously, $\overline{C^\infty_0(\Omega)}^{\|\cdot\|_{\widehat{\mathbb{H}}^{\kappa,q}_w(\Omega)}} \supset \overline{C^\infty_0(\Omega)}^{\|\cdot\|_{\widehat{\mathbb{H}}^{\kappa,q}_w(\IR^n)}}$.  But are they equivalent? It seems to be that this problem is explained by [Triebel, §4.3.2, pp. 317-320]. There a space $\tilde H^{s,q}$ is defined which coincides with $H^{s,q}_0$ unless $s \in 1/q +\IZ$)) \BLACK \normalsize

Concerning spaces of solenoidal vector fields we define $\mathbb L^q_{\sigma,w}(\Omega)$ and $L^q_{\sigma,\ell}(\Omega)$ as the closure of $C^\infty_{0,\sigma}(\Omega)$ with respect to the norm of $\mathbb L^q_{\sigma,w}(\Omega)$ and $L^q_\ell(\Omega)$, respectively.
By analogy, we introduce the subspace of solenoidal vector fields in $\widehat{\mathbb{H}}^{1,q}_{w,0}(\Omega)$ by
\begin{align}\label{spaces}
 \widehat{\mathbb{H}}^{1,q}_{\sigma,w,0}(\Omega) & :=  \overline{C^\infty_{0,\sigma}(\Omega)}^{\|\nabla \cdot\|_{L^q_w(\Omega)}},
\end{align} 
 and, motivated by complex interpolation, 
\begin{equation}\label{spaces2}
\widehat{\mathbb{H}}^{\theta,q}_{\sigma,w,0}(\Omega) := [\mathbb L^q_{\sigma,w}(\Omega), \widehat{\mathbb{H}}^{1,q}_{\sigma,w,0}(\Omega)]_\theta.
\end{equation}
\BLACK
Obviously there holds the identity
$$
\mathbb L^q_{\sigma,w}(\Omega) \cap \widehat{\mathbb H}^{1,q}_{\sigma,w,0}(\Omega)  = \widehat{\mathbb H}^{1,q}_{\sigma,w,0}(\Omega). $$
%\quad \widehat{\mathbb H}^{1,q}_{\sigma,w,0}(\Omega) =  \overline{C^\infty_{0,\sigma}(\Omega)}^{\|\nabla\cdot\|_{\mathbb L^q_w(\IR^n)}},$$
%where the last norm is the norm of $\widehat{\mathbb H}^{1,q}_{\sigma,w,0}(\IR^n)$ used as a quotient norm. 
Homogeneous spaces with weights $\langle\cdot\rangle^{\ell  q}$ are defined analogously. For instance
$$\widehat H^{1,q}_{\sigma,\ell,0}(\Omega) := \overline{C^\infty_{0,\sigma}(\Omega)}^{\|\nabla\cdot\|_{L^q_\ell(\Omega)}}. $$

\begin{prop}\label{dense Htheta}
Let $\Omega\subset \IR^n$ be an exterior domain as above. Then, for each $1<q<\infty$, $0<\theta<1$ and $w\in\mathscr A(\IR^n)$,  the set $C^\infty_{0,\sigma}(\Omega)$ is dense in the interpolation space $[\mathbb L^q_{\sigma,w}(\Omega), \widehat{\mathbb{H}}^{1,q}_{\sigma,w,0}(\Omega)]_\theta$, {\em i.e.} by definition in $\widehat{\mathbb{H}}^{\theta,q}_{\sigma,w,0}(\Omega)$. 

By analogy, $C^\infty_{0,\sigma}(\Omega)$ is dense in $[\mathbb L^q_{w}(\Omega), \widehat{\mathbb{H}}^{1,q}_{w,0}(\Omega)]_\theta = \widehat{\mathbb{H}}^{\theta,q}_{w,0}(\Omega)$.  
\end{prop}

\begin{proof}
By \cite[Theorem 1.9.3 (c)]{Triebel} the set $\mathbb L^q_{\sigma,w}(\Omega) \cap \widehat{\mathbb H}^{1,q}_{\sigma,w,0}(\Omega)$ is dense in the interpolation space $[\mathbb L^q_{\sigma,\ell}(\Omega), \widehat{\mathbb H}^{1,q}_{\sigma,\ell,0}(\Omega)]_\theta$.
Moreover, since $C^\infty_{0,\sigma}(\Omega) $ is dense in $\mathbb H^{1,q}_{\sigma,w,0}(\Omega)$, we obtain by (3.68) and (3.67) that $C^\infty_{0,\sigma}(\Omega)$ is dense in $\widehat{\mathbb H}^{\theta,q}_{\sigma,w,0}(\Omega)$.
\end{proof}

\small
%\RED (({\bf Conclusion:} The space $\widehat H^{\theta,q}_{\sigma,\ell,0}(\Omega)$ can be defined by complex interpolation
%and, equivalently, by density of $C^\infty_{0,\sigma}(\Omega)$. I guess or hope that similar results for $1<\theta<2$ can be obtained easily. There is also the method of reiteration. \BLACK))\normalsize
%
\begin{lem}\label{hatH0-dense}
Let $1<q<\infty$, $w\in\mathscr  A_q(\IR^n)$.

(i) The set $C^\infty_{0}(\IR^n)$ is dense in $\widehat{\mathbb{H}}^{1,q}_{w}(\IR^n)$.

(ii) The set $\Delta C_0^\infty(\IR^n)$ is dense in $\mathbb L^q_w(\IR^n)$.

(iii) The set $C^\infty_{0}(\IR^n)$ is dense in $\widehat{\mathbb{H}}^{\kappa,q}_{w}(\IR^n)$ for each $0\leq\kappa\leq 2$.  
\end{lem}

\begin{proof}
For (i) and (ii) we refer to \cite[Corollary 4.3]{FS} 
and \cite[Lemma 4.1 (iii)]{FS}, respectively. 

(iii) First we consider the case when $\kappa=2$. Given $u \in \widehat{\mathbb{H}}^{2,q}_{w}(\IR^n)$ so that $\Delta u\in \mathbb L^q_w(\IR^n)$ there exists due to (ii) a sequence $(u_k)_k\subset C^\infty_0(\IR^n)$ such that $\Delta u_k\to \Delta u$  in $\mathbb L^q_w(\IR^n)$. Applying Riesz transforms  Proposition \ref{thm-hoermander} implies that $\nabla^2 u_k\to \nabla^2 u$ in $\mathbb L^q_w(\IR^n)$ as $k\to\infty$. 

Now let $0<\kappa<2$. We exploit that $-\Delta$ has the property $BIP$ of bounded imaginary powers; a proof via the bounded $\mathscr H^\infty$-calculus of $-\Delta$ on $\mathbb L^q_w(\IR^n)$ can  be performed by analogy to the proof for the Stokes operator as in § \ref{S2.4}. 
In particular, we get the identity $[\mathbb L^q_w(\IR^n),\widehat H^{2,q}_w(\IR^n)]_{\kappa/2} = \widehat{\mathbb H}^{\kappa,q}_w(\IR^n)$ for each $\kappa\in (0,2)$.
Since $C_0^\infty(\IR^n)$ is dense in $\mathbb H^{2,q}_w(\IR^n) = \mathbb L^q_w(\IR^n) \cap \widehat{\mathbb H}^{2,q}_w(\IR^n)$ and the latter intersection is dense in $\widehat{\mathbb H}^{\kappa,q}_w(\IR^n)$ by \cite[Theorem 1.9.3 (b)]{Triebel}, we get that $C_0^\infty(\IR^n)$ is dense in $\widehat{\mathbb H}^{\kappa,q}_w(\IR^n)$.
\end{proof}

\vspace*{2ex}

\noindent
{\bf Estimates in homogeneous Sobolev spaces:}
An essential tool for inequalities in weighted Sobolev spaces are the general Gagliardo-Nirenberg estimates from which also Propositions \ref{embed-ext-weight} and \ref{Hardy-Rellich} can be derived.

\vspace*{1ex}

\begin{prop}\label{DS} {\em (Duarte and Silva \cite[Theorem 1.7]{Du-Si})} 
Let $1 < p, q, r <  \infty$, $0\leq k_0 <k$, $\frac{k_0}{k}\leq \theta\leq 1$, and $\alpha\in (-\frac{n}{p}, \frac{n}{p'})$, $\beta\in (-\frac{n}{q}, \frac{n}{q'})$, $\gamma >-\frac{n}{r}$, 
satisfy  
\begin{equation}\label{theta-pqr-alpha-beta-N}
\theta\Big(\frac{1}{p}-\frac{k}{n}\Big) + (1-\theta)  \frac{1}{q} +\frac{\gamma}{n} \leq \frac{1}{r} -\frac{k_0-\gamma}{n} \leq  \theta\Big(\frac{1}{p}-\frac{k-\alpha}{n}\Big) + (1-\theta)\Big(\frac{1}{q}+\frac{\beta}{n}\Big)
\end{equation} 
and 
$$ \gamma \leq \theta \alpha + (1-\theta)\beta $$
with 
$$
\frac{1}{r} \leq \frac{\theta}{p}+\frac{1-\theta}{q}.
$$
Then there holds for any Schwartz function $f \in \mathcal S(\mathbb{R}^n)$ %\RED (Note the correction of weight $\langle x\rangle^\alpha$ instead of $|x|^\alpha$ as in \cite{Du-Si}) \BLACK
\begin{align}\label{GNI-weighted-N}  
\|(-\Delta)^{k_0/2} f \|_{L^r_\gamma} 
\leq C\|(-\Delta)^{k/2} f\|_{L^p_\alpha}^\theta \, \|f\|_{L^q_\beta}^{1-\theta}.
\end{align}   
\end{prop}

\vspace{1ex}

{\rm To apply Proposition \ref{DS} also to exterior domains we need special extension operators that take into account estimates in homogeneous function spaces. 
Chua \cite{Chua} proved the extension property for homogeneous weighted Sobolev spaces for general unbounded $(\varepsilon,\infty)$-domains, including smooth exterior domains. The crucial point is that homogeneous norms are conserved under these operators.}

%In Lemma \ref{extension} below the symbol $\mathbb H_w$ stands for weighted Sobolev spaces $H$ with weight $w$ rather than $\langle\cdot\rangle^{sq}$, {\em cf.} $\mathbb L$ and $L$ in Definition \ref{Muckenhoupt-class} (ii) and (iii).

\vspace*{1ex}

\begin{lem}\label{extension} {\rm(\cite[Theorem 1.5]{Chua}, \cite[Theorem 2.2]{FroehII})}
Let $1<q <\infty$, $w\in \mathscr{A}_q$, and $k_1<\ldots<k_N\in\mathbb N_0$. Further let $\Omega\subset\IR^n$ be an exterior Lipschitz domain. \BLACK Then there exists a linear extension operator $E: \bigcap_{i=1}^N \widehat{\mathbb{H}}^{k_i,q}_{w}(\Omega)\rightarrow  \bigcap_{i=1}^N \widehat{\mathbb{H}}^{k_i,q}_{w}(\mathbb{R}^n)$ such that 
$$
\|\nabla^{k_i} Eu \|_{\mathbb{L}^q_w(\mathbb{R}^n)} \leq C_i \|\nabla^{k_i} u\|_{\mathbb{L}^q_w(\Omega)}
$$
for all $i=1, \ldots,N$ and $u \in \bigcap_{i=1}^N \widehat{\mathbb{H}}^{k_i,q}_{w}(\Omega)$. 
\end{lem}

\vspace*{1ex}

As a consequence of Proposition \ref{DS}  and Lemma \ref{hatH0-dense} (i), \BLACK we have the following inequalities of Hardy, Rellich and Sobolev on weighted spaces.

\vspace*{1ex}

\begin{prop}\label{embed-ext-weight}{\em \cite[Corollary 3.8]{FS}, \cite[Proposition 2.2]{Farwig-Tsuda-stab}} 
Let $\Omega\subset \IR^n$ be either an exterior domain of class $C^1$ or let $\Omega=\IR^n$. Assume $1 < q\leq r < \infty$ and for $\ell, \ell'\in \IR$ that $1-\frac{n}{q} < \ell < \frac{n}{q'}$ and
\begin{equation}\label{E3.12}
\ell-\ell' = 1 + \frac{n}{r} -\frac{n}{q}\geq 0.
\end{equation}
Then the embedding
$\widehat H_\ell^{1,q}(\Omega)\hookrightarrow L_{\ell'}^r(\Omega)$ is continuous, {\em i.e.,} for each representative
$ u\in \widehat H_\ell^{1,q}(\Omega)$ there exists a constant $\mathcal C\in\IR$ such that 
\begin{equation}\label{E3.12'} 
\|u-\mathcal C\|_{L^r_{\ell'}}  \leq c \|\nabla u\|_{L^{q}_\ell} 
\end{equation}
where $c=c(\ell,\ell', r,q,\Omega)>0$ is independent of $u$. In particular, there holds the embedding 
\begin{equation}\label{embedHhat-L}
    \widehat H^{1,q}_{\ell,0}(\Omega) \hookrightarrow L^r_{\ell'}(\Omega).
\end{equation}
\end{prop} 

\vspace*{1ex}

\begin{prop}\label{Hardy-Rellich}{\em \cite[Lemma 2.4]{Farwig-Tsuda-stab}}  
Let $\Omega\subset \IR^n$, $n\geq 3$, be an exterior Lipschitz domain, and let $q \in (1,\infty)$.

(i) (Hardy's inequality) Let $u \in H^{1,q}_{\ell}(\Omega)$ with  $1-\frac{n}{q} <\ell< \frac{n}{q'}$.
Then there holds  
\begin{equation}\label{Hardy-ineq}
\Big\|\dfrac{u}{1+|x|} \Big\|_{L^q_{\ell}(\Omega)} 
\leq C\|\nabla u \|_{L^q_{\ell}(\Omega)}. 
\end{equation}

(ii) (Rellich's inequality) 
Let $u \in H^{2,q}_{\ell}(\Omega)$  and $2-\frac{n}{q} < \ell < \frac{n}{q'}$.   
Then  
\begin{equation}\label{Rellich-ineq}
\Big\|\dfrac{u}{(1+|x|)^2} \Big\|_{L^q_{\ell}(\Omega)} 
\leq C\|\nabla^2 u \|_{L^q_{\ell}(\Omega)}. 
\end{equation}
  
%(iii) (Fractional Sobolev inequality) Let $1<p<q<\infty$, $\beta=\frac{n}{2}\big(\frac1p-\frac1q\big)$ and  $-\frac{n}{q} < s < \frac{n}{p'}$. 
%(not: $2-\frac{n}{p} < s < \frac{n}{p'}$ since $A$ is not yet used). 
%Then for $u\in \widehat H^{2\beta}_p(\IR^n)$ there holds the Sobolev embedding estimate
%\begin{equation}\label{Sob-ineq}
%\|u\|_{L^q_s(\IR^n)} \leq C\|(-\Delta)^\beta u\|_{L^p_s(\IR^n)}.
%\end{equation}

(iii) (Sobolev inequality) Let $1< q\leq r<\infty$, let $\kappa\in(0,2]$ \BLACK satisfy $\kappa-\frac{n}{q} < \ell < \frac{n}{q'}$, \BLACK and assume that
\begin{equation}\label{fract-ell}
\ell-\ell' = \kappa + \frac{n}{r} -\frac{n}{q}\geq 0.
\end{equation}
Then there holds the fractional Sobolev embedding 
\begin{equation}\label{fractSob}
\widehat H^{\kappa,q}_{\ell}(\IR^n) \hookrightarrow L^r_{\ell'}(\IR^n)
\end{equation} 
in the sense that for every representative $u\in \widehat H^{\kappa,q}_{\ell}(\IR^n)$ \BLACK there exists a polynomial $\pi_k$ of degree at most $k=\lfloor \kappa \rfloor -1$ such that 
\begin{equation} \label{pi_k} 
\|u+\pi_k\BLACK\|_{L^r_{\ell'}(\IR^n)} \leq C \|u\|_{\widehat H^{\kappa,q}_{\ell}(\IR^n)}.
\end{equation}

For an exterior domain $\Omega\subset \IR^n$ the fractional Sobolev embedding $\widehat H^{\kappa,q}_{\ell,0}(\Omega) \hookrightarrow L^r_{\ell'}(\Omega)$ and \eqref{pi_k} hold with $\pi_k = 0$.
\end{prop}

\begin{proof}
All assertions follow from Proposition \ref{DS} with $\theta=1$ so that condition \eqref{theta-pqr-alpha-beta-N} simplifies with $k=\kappa$, $\gamma=\ell'$, $\alpha=\ell$ to 
$$  \frac{1}{p}-\frac{\kappa}{n}  + \frac{\ell'}{n} \leq \frac{1}{r} + \frac{\ell'}{n} \leq \frac{1}{p}-\frac{\kappa-\ell}{n}   
$$
which can be written in the form $\ell-\ell' \geq  \kappa + \frac{n}{r} -\frac{n}{q}\geq 0$, {\em i.e.} \eqref{fract-ell} with "=" replaced by "$\geq$". 
For (iii), assuming \eqref{fract-ell} and $\kappa-\frac{n}{q} < \ell < \frac{n}{q'}$, we easily deduce that $\ell'\leq\ell$ and $\ell'>-\frac{n}{r}$. %as well as $\ell'<\frac{n}{r'}$ since $q\leq r$. 
The polynomial $\pi_k$ vanishes when $u\in \mathcal S(\IR^n)$.

Now let $\in \widehat H^{\kappa,q}_{\ell}(\IR^n)$ be arbitrary. Then Lemma \ref{hatH0-dense} yields a sequence $(u_m)_m\subset C_0^\infty(\IR^n)$ such that $u_m\to u$ in $\widehat H^{\kappa,q}_{\ell}(\IR^n)$, and  \eqref{fractSob}, proved already on $\mathcal S(\IR^n)$, \BLACK
%Proposition \ref{DS} 
applied to $u_m-u_{m'}$, shows that
$$ \|u_m-u_{m'}\|_{L^r_{\ell'}(\IR^n)} \leq C \|u_m-u_{m'}\|_{\widehat H^{\kappa,q}_{\ell}(\IR^n)}.$$
We easily conclude that $(u_m)$ converges in  $L^r_{\ell'}(\IR^n)$ to some function $\tilde u\in L^r_{\ell'}(\IR^n)$ \BLACK and that $(-\Delta)^{\kappa/2}\tilde u = (-\Delta)^{\kappa/2}u$. 
Hence there exists a polynomial $\pi_k$ of degree $k=\lfloor \kappa \rfloor -1$ such that $\tilde u=u+\pi_k$, thus proving \eqref{pi_k}.
%Hence Proposition \ref{DS} implies \eqref{fractSob}. 

Finally, for  an exterior domain and $u\in \widehat H^{\kappa,q}_{\ell,0}(\Omega)$ we find by definition a sequence $(u_m)_m \subset C^\infty_0(\Omega) \subset\mathcal S(\IR^n)$ such that \eqref{pi_k} holds for each $u_m$ with $\pi_k=0$. Passing to the limit $m\to \infty$, also in the definition of $\nabla$ in the sense of distributions, we find a unique representative $u\in H^{1,q}_{\rm loc}(\overline \Omega)$ such that $u|_{\partial\Omega}=0$ and $\pi_k=0$ in \eqref{pi_k}. \BLACK

(i) and (ii) follow with (iii) when $\ell=\ell'$ and $\kappa=1$ or $=2$, respectively. The polynomial $\pi_K$ vanishes since by assumption $u$ is contained in a nonhomogeneous Sobolev space. \BLACK  
\end{proof}
\vspace{2ex}

Note that in \cite[Lemma 2.4]{Farwig-Tsuda-stab}) inequalities \eqref{Hardy-ineq},  \eqref{Rellich-ineq} are proved only when $\ell\geq 0$. However, since the proof is mainly based on Proposition \ref{DS} applied to $Eu$ rather than $u$ where $E$ is an extension operator of Lemma \ref{extension}, the condition $\ell\geq 0$ can be ignored. 

\vspace*{2ex}

\noindent{\bf Helmholtz projections:}
For bounded and exterior domains $\Omega\subset\IR^n$ with boundary of class $C^1$  
we recall the Helmholtz projection, {\em i.e.,} the projection $\mathbb{P}_q\colon L^q(\Omega)\to L^q_{\sigma}(\Omega) \subset L^q(\Omega)$, $1<q<\infty$, such that the kernel of $\mathbb{P}_q$ equals the space $G_q(\Omega)$ of all weak gradient fields in $L^q(\Omega)$ and the range $\mathcal R(\mathbb P_q) = L^q_\sigma(\Omega)$. 
The adjoint of $\mathbb P_q$ coincides with the Helmholtz projection $\mathbb{P}_{q'}: L^{q'}(\Omega) \to L^{q'}(\Omega)$ where $\frac1q+\frac1{q'}=1$. 
Since $\mathbb{P}_q u = \mathbb{P}_ru$ for all vector fields $u\in L^q(\Omega)\cap L^r(\Omega)$, $1<q,r<\infty$, and a similar result holds in weighted spaces, we will omit indices such as $q$ and $q,w$ or $q,\ell$ in most cases for $\mathbb{P}.$
\vspace*{1ex}

In weighted function spaces we have the following results:

\begin{lem}\label{Helmh}{\rm \cite[Theorem 1.3, Corollary 4.4]{FS}}
For an exterior domain $\Omega \subset\IR^n$ with boundary of class $C^1$, all $1<q<\infty$ and weights $w\in \mathscr A_q$ there exists the bounded Helmholtz projection 
$$ \mathbb P = \mathbb P_{q,w}: \mathbb L^q(\Omega) \to \mathbb L^q_{\sigma}(\Omega) $$
with null space $\mathcal N(\mathbb P) = \nabla \widehat{\mathbb H}^{1,q}_{w}(\Omega)$ and range $\mathcal R(\mathbb P) = \mathbb L^q_{\sigma}(\Omega)$. The adjoint of $\mathbb P_{q,w}$ equals $\mathbb P_{q',w'}$ on $L^{q'}_{w'}(\Omega)$ with $w'=w^{-1/(q-1)}$.

Similar results hold for the whole space $\IR^n$. 
\end{lem}

\vspace*{1ex}
To discuss interpolation properties of homogeneous and nonhomogeneous function spaces we recall the retraction-coretraction principle, see \cite{Triebel}.

\begin{defi}
Let $A,B$ be complex Banach spaces and $R\in \mathcal L(A,B)$. Then $R$ is called a {\em retraction} if there exists an operator $S\in \mathcal L(B,A)$ such that $RS=I$ on $B$. In this case, $S$ is called a {\em coretraction} corresponding to $R$.
\end{defi}

\vspace*{1ex}

\begin{thm}\label{retract-co}
Consider interpolation couples of Banach spaces, $\{A_0,A_1\}$ and $\{B_0,B_1\}$. 
Assume that there exists a retraction $R: A_j\to B_j$ with coretraction $S:B_j\to A_j$, $j=1,2$. 
Let $(\cdot,\cdot)$ denote an interpolation functor, applied to $\{A_0,A_1\}$ and $\{B_0,B_1\}$.
Then $S$ is an isormorphism,  
$$ S:(B_0,B_1) \to SR((A_0,A_1)),$$
$ S((B_0,B_1)) =  SR((A_0,A_1)),$ and 
$$ (B_0,B_1)  = R ((A_0,A_1))\quad \textrm {with equivalent norms}.$$
Moreover, $(SR) \big|_{(A_0,A_1)}$ is a projection, and  $\mathcal R \big((SR) \big|_{(A_0,A_1)}\big)$ is a complemented subspace of  $(B_0,B_1)$.

In particular, if $P$ is bounded projection on $A_j$, $j=0,1$, then 
\begin{equation}\label{interpol-P}
   (PA_0,PA_1) = P((A_0,A_1)). 
\end{equation}  
\end{thm}

\begin{lem}\label{interp-hatH-Rn in q}
Let $1<q_0<q_1<\infty$, $0\leq \theta\leq 1$ and $\frac1q = \frac{1-\theta}{q_0} + \frac{\theta}{q_1}$. Further assume $-\frac{n}{q_1} <\ell< \frac{n}{q_0'}$. Then there holds
$$ [\widehat H^{1,q_0}_\ell(\IR^n), \widehat H^{1,q_1}_\ell(\IR^n)]_\theta = \widehat H^{1,q}_\ell(\IR^n).$$
\end{lem}

\begin{proof} 
To apply \eqref{interpol-P} we define the projection (retraction) 
$$ P=R=I-\mathbb P_{q,\ell}: L^q_\ell(\IR^n) \to \nabla \widehat H^{1,q}_\ell(\IR^n),\;\; u\mapsto \nabla p = (I-\mathbb P_{q,\ell})u, $$ 
%with corresponding coretraction $S:\nabla \widehat H^{1,q}_\ell(\IR^n) \to  L^q_\ell(\IR^n)$ where $S(\nabla p)=\nabla p$. 
Then Theorem \ref{retract-co} implies  for $0< \theta< 1$ that 
$$ [\nabla \widehat H^{1,q_0}_\ell(\IR^n), \nabla\widehat H^{1,q_1}_\ell(\IR^n) ]_\theta 
= P \big([ L^{q_0}_\ell(\IR^n), L^{q_1}_\ell(\IR^n)]_\theta\big)
= P L^q_\ell(\IR^n) = \nabla\widehat H^{1,q}_\ell(\IR^n).$$
Considering $\nabla$ as an isomorphism from $\widehat H^{1,q}_\ell(\IR^n)$ to $\nabla \widehat H^{1,q}_\ell(\IR^n)$, the lemma is proved. \BLACK
\end{proof}

\vspace*{1ex}

\begin{lem}\label{interp-hatH-Rn in k}
For any $1<q<\infty$, $w\in \mathscr A_q(\IR^n)$ and $0\leq \kappa_0<\kappa_1<\infty$ there holds with $\kappa = (1-\theta)\kappa_0 + \theta\kappa_1$, $0\leq \theta\leq 1$ and equivalent norms 
\begin{align}\label{interp hatH}
 [\widehat H^{\kappa_0,q}_w(\IR^n),\widehat H^{\kappa_1,q}_w(\IR^n)]_\theta & = \widehat H^{\kappa,q}_w(\IR^n),\\
 \label{interp hatH sigma}
 [\widehat H^{\kappa_0,q}_{\sigma,w}(\IR^n),\widehat H^{\kappa_1,q}_{\sigma,w}(\IR^n)]_\theta & = \widehat H^{\kappa,q}_{\sigma,w}(\IR^n). 
\end{align}
\end{lem}

\begin{proof}
Since $\widehat H^{\kappa,q}(\IR^n) = \D((-\Delta_{q,w})^{\kappa/2})$, \eqref{interp hatH} follows from complex interpolation and the property $BIP$ of $-\Delta$, {\em cf.} Corollary \ref{cor:Rn-domains} for the Stokes operator. Concerning \eqref{interp hatH sigma} we apply \eqref{interp hatH} and Theorem \ref{retract-co} with $\mathbb P_{q,w}$ as retraction.
\end{proof}

\subsection{Bogovski\u{\i} operators on weighted homogeneous spaces}\label{S2.2}

Let us recall several well known properties of Bogovski\u{\i} operators on bounded domains and then move on to exterior domains.

\begin{prop}\label{Bog}
Let $D \subset \IR^n$ denote a bounded domain of class $C^1$, and let $1<q<\infty$. 
(i) There exists a linear bounded Bogovski\u{\i} operator %\RED (it seems to better to call a bounded domain $D$, say, since $\Omega_1$ is confusing.) \BLACK
$$ \mathbb B: L^q_{(m)}(D) \to H^{1,q}_0(D) $$
such that $\div \mathbb Bg=g$ for $g\in L^q_{(m)}(D) := \big\{h\in L^q(D): \int_{D} h\dx=0 \big\}$. 
Moreover, if $g\in H^{1,q}_0(D)$, then $u = \mathbb Bg\in H^{2,q}_0(D)$ and $\|\mathbb Bg\|_{H^{2,q}(D)} \leq C \|g\|_{H^{1,q}(D)}$. 

(ii) If $\partial D\in C^2$, then the operator $B$ extends to a bounded operator 
$$ \mathbb B: H^{-1,q}_0(D)\to L^q(D)$$
where $H^{-1,q}_0(D) := (H^{1,q'}(D))^*$, {\em i.e.}, the dual to the Sobolev space $H^{1,q'}(D)$ allowing for nonzero traces.

(iii) For any $p\in L^q(D)$ there holds $\nabla p\in H^{-1,q}(D)$ and the estimate 
\begin{equation}\label{p_m}
\Big\|p-\frac{1}{|D|} \int_D p \dx\Big\|_{L^q(D)} \leq C \|\nabla p\|_{H^{-1,q}(D)}. \end{equation}
\end{prop}

\begin{proof}
For the proof of (i), (ii) we refer to \cite[Theorems III.3.1, III.3.3, and III.3.5]{Galdi-steady}. 

 Concerning (iii) we follow \cite[Proposition 3.8 (i)]{BoMi90} and note that by (i) 
$$ \div\! : H^{1,q'}_0(D) \to L^{q'}_{(m)}(D) \subset L^{q'}(D) $$ 
is a bounded surjective operator with closed range in $L^{q'}(D)$. Hence the closed range theorem implies that 
$$ \nabla: L^q_{(m)}(D) \to H^{-1,q}(D) $$
is a bounded injective operator with closed range. Hence the open mapping theorem implies that $\nabla$ is even an isomorphism from $L^q_{(m)}(D)$ to $\nabla L^q_{(m)}(D) \subset H^{-1,q}(D)$. Therefore, we obtain \eqref{p_m} for $p\in L^q_{(m)}(D)$ and similarly for $p\in L^q(D)$ when subtracting its integral mean on \BLACK $D$. 
\end{proof}

\vspace*{1ex}

There exist also Bogovski\u{\i} operators on exterior domains such that $\nabla \mathbb Bg$ can be estimated by $g$; in this case there is no integral mean condition. 
For the result yielding the estimate $\|\mathbb Bg\|_{\widehat H^{1,q}(\Omega)}\leq C \|g\|_{L^q(\Omega)}$ we refer to \cite[Theorem III.3.6]{Galdi-steady}. 
 In our case a corresponding result in weighted spaces is more interesting. We note that 
Proposition \ref{Bog-ext} below will not be used in the following. %((DELETE: \BLACK will be used in the proof of Proposition \ref{lem-var-ineq-nabla_Omega0} on a  variational inequality for $\nabla u$ in $L^q_\ell(\Omega)$, see \eqref{var-ineq-nabla}.))

\begin{prop}\label{Bog-ext}
Let $\Omega\subset \IR^n$ be an exterior domain with $\partial\Omega\in C^1$, let $1<q<\infty$ and $-\frac{n}{q} < \ell <\frac{n}{q'}$. 
Then there exists a linear bounded operator
$$ \mathbb B: L^q_\ell(\Omega) \to \widehat H^{1,q}_{\ell,0}(\Omega) $$
such that $\div \mathbb Bg=g$ for $g\in L^q_\ell(\Omega)$ and 
$$ \|\mathbb Bg\|_{\widehat H^{1,q}_{\ell,0}(\Omega)} \leq C \|g\|_{L^q_\ell(\Omega)}. $$

 Moreover, for any $p\in L^q_\ell(\Omega)$ there holds $\nabla p\in H^{-1,q}_\ell(\Omega)$ and the estimate \BLACK
\begin{equation}\label{p_ext}
\|p\|_{L^q_\ell(\Omega)} \leq C \|\nabla p\|_{H^{-1,q}_\ell(\Omega)}. \end{equation}
%
%Moreover, if $g\in H^{1,q}_0(\Omega_1)$, then $u=Bg\in H^{2,q}_0(\Omega_1)$ and $\|Bg\|_{H^{2,q}(\Omega_1)} \leq C \|g\|_{H^{1,q}(\Omega_1)}$. 
%
%\RED DELETE (ii): \BLACK (ii) Let $F$ be a bounded linear functional on $\widehat H^{1,q}_{\ell,0}(\Omega)$ vanishing on $\widehat H^{1,q}_{\sigma,\ell,0}(\Omega)$. %
%Then there exists a unique $\pi\in L^{q'}_{-\ell}(\Omega)$ (up to an additive constant) such that $F=-\nabla \pi$. Moreover, there exists $C=C(q,\ell,\Omega)>0$ such that 
%\begin{equation}\label{pi-nablapi}
% \|\nabla \pi\|_{\widehat H^{-1,q'}_{-\ell}(\Omega)} \leq \|\pi\|_{L^{q'}_{-\ell}(\Omega)} \leq C\|\nabla \pi\|_{\widehat H^{-1,q'}_{-\ell}(\Omega)}.    
%\end{equation}
\end{prop}

\begin{proof}
Given $g\in L^q_\ell(\Omega)$ we denote the extension by $0$ to $\IR^n$ also by $g$. Then the equation 
$$ \Delta \psi = g \quad \textrm{in }\;\IR^n$$
possesses a solution $\psi\in \widehat H^{2,q}_\ell(\IR^n)$
satisfying 
\begin{equation}\label{est-psi2}
    \|\nabla^2 \psi\|_{L^q_\ell(\IR^n)} \leq C \|g\|_{L^q_\ell (\IR^n)}; 
\end{equation}
actually, $\psi$ is defined by its Fourier transform $|\xi|^{-2} \hat g(\xi)$ and estimated by the H\"ormander-Mikhlin multiplier theorem in weighted spaces, see Theorem \ref{thm-hoermander}. 

To construct $\mathbb B$ and to prove that it maps $L^q_\ell(\Omega)$ into $\widehat H^{1,q}_{\ell,0}(\Omega)$ we first assume that 
\begin{equation}\label{g-psi-C0infty}
     \psi\in C^\infty_0(\Omega),\;\; g=\Delta \psi\in C^\infty_0(\IR^n). 
\end{equation}
Later, a density argument will complete the proof. Further, to fulfill the integral mean condition 
\begin{equation}\label{int-psi}
\int_{\Omega_{2R}} \nabla\psi\dx = 0
\end{equation}
on the bounded domain $\Omega_{2R}$, we add an adequate linear polynomial $\pi_1$ to $\psi$ and set $\tilde\psi = \psi+\pi_1$. This will have an impact neither on $\nabla^2\tilde\psi$ nor on the integral 
$$ \int_{\partial\Omega} -(\nabla\tilde\psi)\cdot \textsl{n} \dsigma = \int_{\partial\Omega} -\nabla\psi\cdot \textsl{n} \dsigma = 0; $$
%\RED ((The strange step with $\pi_1$ is important in the following and due to Galdi. If $\ell>1-n/q$ (i.e. $q<n$ if there is no weight) this step is not necessary since $\nabla\psi\in L^r_{\ell'}$. But if $-n/q<\ell\leq 1-n/q$, there is no embedding. Galdi wrote the proof in 2 steps, I unified both cases somehow. Later it will be important to get the result for "all" admissible $q,\ell$)) \BLACK
for the latter integral we used Gauss' theorem and that $\div\nabla\psi=0$ in $\Omega^c$. 
%
%Moreover, with $\nabla^2\psi \in L^q_{\rm loc}$ also $\nabla\psi \in L^q_{\rm loc}$, and 
Due to \eqref{int-psi}, now holding for $\tilde \psi$, Friedrichs' inequality implies that 
\begin{equation}\label{nabla-psi-nabla^2}
\int_{\Omega_{2R}} |\nabla\tilde\psi|^q\dx \leq C\int_{\Omega_{2R}} |\nabla^2\tilde\psi|^q \dx = C\int_{\Omega_{2R}} |\nabla^2\psi|^q \dx\leq C \|g\|^q_{L^q_\ell(\Omega)}. 
\end{equation}

It is important to note that in the special case when additionally $\ell>1-\frac{n}{q}$ we may put $\nabla\tilde\psi = \nabla\psi$ since  by the weighted Hardy inequality \eqref{Hardy-ineq} 
the condition $\nabla^2\psi\in L^q_\ell(\IR^n)$ implies $\nabla\psi\in L^q_{\ell'}(\IR^n)$ where $\ell'=\ell-1$ satisfying $-\frac{n}{q} < \ell' < -1+\frac{n}{q'}$. 
Since $\|\nabla\psi\|_{L^q_{\ell'}(\IR^n)} \leq C \|\nabla^2\psi\|_{L^q_{\ell}(\IR^n)}$, the estimate of $\nabla\psi$ by $g$ in \eqref{nabla-psi-nabla^2} holds in this case as well.

Let $v\in H^{1,q}_0(\Omega_{2R})$ with $\Omega_{2R}=\Omega\cap B_{2R}(0)$ be a solution of the equation 
\begin{equation}\label{div-problem} 
\div v = 0\; \textrm{ in } \Omega_{2R},\quad v=-\nabla\tilde\psi\;\textrm{  on }\;\partial\Omega,\quad v=0\;\textrm{ on }\;\partial B_{2R}.
\end{equation}
Such a vector field can be found by a Bogovski\u{\i} operator $\mathbb B(\Omega_{2R})$ on $\Omega_{2R}$ by extending in a first step the boundary values $-\nabla \tilde\psi$ and $0$ on $\partial\Omega$ and $\partial B_{2R}$, respectively, to a function $w\in H^{1,q}(\Omega_{2R})$. 
Note that we get the estimate $\|w\|_{H^{1,q}(\Omega_{2R})} \leq c \|\nabla\tilde\psi\|_{H^{1,q}(\Omega_{2R})}$. 
Moreover, $\div w$ satisfies 
$$ \int_{\Omega_{2R}} \div w\dx = \int_{\partial\Omega} -\nabla\tilde\psi\cdot  \textsl{n} \BLACK   \dsigma = 0. $$
%Note that the boundary integral $\int_{\partial\Omega_0} -\nabla\psi\cdot \textsl{n} \dsigma$ does not change when we add a constant $\mathcal C$ to $\nabla\psi$.
Then we solve the divergence problem 
$$ \div v' = -\div w\; \textrm{ in } \Omega_{2R},\quad v' = 0\;\textrm{  on }\;\partial\Omega\, \cup\,\partial B_{2R}$$
by the operator $\mathbb B(\Omega_{2R})$ satisfying the estimate $\|v'\|_{H^{1,q}(\Omega_{2R})} \leq c\|w\|_{H^{1,q}(\Omega_{2R})}$.
Summarizing the above local estimates, we find a solution $v:=v'+w$  to \eqref{div-problem}, extend $v$ by $0$ to $\Omega$ and get due to \eqref{nabla-psi-nabla^2}, \eqref{est-psi2} that 
%, \eqref{est-psi1}
\begin{align*}
  \|v\|_{H^{1,q}_\ell(\Omega)} & \leq C\|\nabla\tilde\psi\|_{H^{1,q}(\Omega_{2R})} \leq C \|\nabla^2 \tilde\psi\|_{L^{q}_\ell(\IR^n)} 
  \leq C \|g\|_{L^q_\ell(\Omega)}.  
\end{align*}

Adding $v$ and $\nabla\tilde\psi$ we get a linear operator 
$$ \mathbb B: L^q_\ell(\Omega)\to \widehat H^{1,q}_{\ell,0}(\Omega),\quad g\mapsto \mathbb Bg = v+\nabla\tilde\psi, $$ satisfying $\div \mathbb Bg=g$, $\|\mathbb Bg\|_{\widehat H^{1,q}_{\ell}(\Omega)} \leq C \|g\|_{L^q_\ell(\Omega)}$ and $\mathbb Bg\big|_{\partial\Omega}=0$.

In the special case when $\ell>1-\frac{n}{q}$ and $\nabla\tilde\psi = \nabla\psi\in \widehat H^{1,q}_{\ell,0}(\IR^n)$ it follows that $\mathbb Bg\in \widehat H^{1,q}_{\ell,0}(\Omega)$. 
Indeed, we choose a cut off function $\phi\in C^\infty_0(\Omega_{3R})$ such that $\phi(x)=1$ for $x\in \Omega_{2R}$ and write 
$$ \mathbb Bg = \phi (v + \nabla\psi) + (1-\phi)(v +  \nabla\psi).$$
Obviously, $(1-\phi)(v+\nabla\psi) = (1-\phi)\nabla\psi \in C^\infty_0(\Omega) \subset \widehat H^{1,q}_{\ell,0}(\Omega)$, since $\psi\in C^\infty_0(\Omega)$. 
Moreover, $\phi(v + \nabla\psi) \in \widehat H^{1,q}_{\ell,0}(\Omega_{3R})$ so that there exists a sequence $(\Phi_k)_k\subset C^\infty_0(\Omega_{3R})$ converging to $\phi(v + \nabla\psi)$ in $H^{1,q}_0(\Omega_{3R})$.  
Then $\Phi_k + (1-\phi)\nabla \psi \in C^\infty_0(\Omega)$ converges to $\mathbb Bg$ in $\widehat H^{1,q}_{\ell,0}(\Omega)$. 
%Consequently, $\mathbb Bg\in \widehat H^{1,q}_{\ell,0}(\Omega)$.

If $-\frac{n}{q} < \ell \leq 1-\frac{n}{q}$ and $\nabla\tilde\psi = \nabla\psi +\nabla\pi_1 = \nabla\psi +\pi_0$ with a constant vector $\pi_0\in \IR^n$, it remains to show that $1-\phi \in \widehat H^{1,q}_{\ell,0}(\Omega)$. This result is proved in Lemma \ref{D^1} below. Thus $(1-\phi) \pi_0 \in \widehat H^{1,q}_{\ell,0}(\Omega)$ and hence $\mathbb Bg\in \widehat H^{1,q}_{\ell,0}(\Omega)$.

Now the proof that $\mathbb Bg \in \widehat H^{1,q}
_{\ell,0}(\Omega)$ is complete for all admissible pairs $q,\ell$ and $g \in C^\infty_0(\Omega)$.
A final density argument by Lemma \ref{hatH0-dense}  (ii) completes the proof for general $g\in L^q_\ell(\Omega)$. 

The estimate \eqref{p_ext} follows the lines of the proof of Proposition \ref{Bog}; however, the integral mean condition is not defined and \BLACK unnecessary.
\end{proof}

%\RED DELETE this proof of (ii): \BLACK (ii) The divergence can be considered as a linear bounded operator 
%$$\div: \widehat H^{1,q}_{\ell,0}(\Omega) \to L^q_\ell(\Omega).$$
%By part (i) the range $\mathcal R(\div) = L^q_\ell(\Omega)$ is closed, and, obviously, \RED $\mathcal N(\div)=\widehat H^{1,q}_{\sigma,\ell,0}(\Omega)$ (in the sense of density of $C^\infty_{0,\sigma}(\Omega)$ or as the larger space $\{u\in \widehat H^{1,q}_{\ell,0}(\Omega): \div u=0\}$ or is it the same?)\BLACK. The adjoint operator $(\div)^*=-\nabla$ is a bounded map 
%$$ -\nabla: L^{q'}_{-\ell}(\Omega) \to \widehat H^{-1,q'}_{-\ell}(\Omega) = \big(\widehat H^{1,q}_{\ell,0}(\Omega)\big)^*$$
%for which by the closed range theorem 
%$\mathcal R(-\nabla)= \mathcal N(\div)^\perp = \widehat H^{1,q}_{\sigma,\ell,0}(\Omega)^\perp$. Here ${}^\perp$ denotes the annihilator. 
%Since $-\nabla$ is injective due to the surjectivity of $\div$, we conclude that for $F \in \widehat H^{1,q}_{\sigma,\ell,0}(\Omega)^\perp$ there exists a unique $\pi\in L^{q'}_{-\ell}(\Omega)$ such that $-\nabla\pi = F$.

%By the open mapping theorem, the map $-\nabla: L^{q'}_{-\ell}(\Omega) \to \widehat H^{1,q}_{\sigma,\ell,0}(\Omega)^\perp$ is an isomorphism. The first part of the estimate \eqref{pi-nablapi} is trivial by a formal integration by parts. The second part of \eqref{pi-nablapi} follows from the isomorphism $\pi \mapsto F:=-\nabla\pi$.
 
\begin{lem}\label{D^1}
Let $-\frac{n}{q} < \ell \leq 1-\frac{n}{q}$\,. Then the set
$$ \mathcal D^1(\Omega) := \big\{\phi\in C^\infty(\Omega):  \phi|_{\partial\Omega}=0,\BLACK\, \nabla\phi\in C^\infty_0(\overline\Omega) \big\} $$
is contained in $\hat H^{1,q}_{\ell,0}(\Omega)$.  An analogous result holds for the whole space $\IR^n$. \BLACK
\end{lem}

\begin{proof}
Assume that $\phi(x)=c=const$ for $|x|\geq 3R$. \BLACK For the proof choose a family of smooth cut off functions $\chi_k, k\in\IN,$ such that $\chi_k(x)=1$ in $B_k$, but $\chi_k(x)=0$ for $|x|\geq 2k$ such that $|\nabla\chi_k|\sim \frac1k$. 
Thus $\nabla (\chi_k(c-\phi)) - \nabla (c-\phi) = (\nabla\chi_k) (c-\phi) $ since $(\chi_k-1) \nabla(c-\phi) = 0$ for $k>3R$.
Now 
$$ \int |\nabla\chi_k|^q \langle x\rangle^{\ell q} \dx \sim \frac1{k^q} k^{\ell q+n} = k^{q \big(\ell-1+\frac{n}{q}\big)} \to 0 \quad \textrm{as }\; k\to\infty $$
provided that $\ell < 1-\frac{n}{q}$. If $\ell = 1-\frac{n}{q}$, then $(\nabla\chi_k)\subset L^q_\ell$ is a bounded sequence; moreover, it converges weakly in $L^q_\ell$ to $0$. Hence Mazur's lemma yields the existence of a sequence of convex combinations of $(\nabla\chi_k)_k$ converging strongly in $L^q_\ell$ to $0$.
\end{proof}
\BLACK

\subsection{Stokes operators}\label{S2.3}

For a smooth exterior domain $\Omega\subset \IR^n$ let 
$$ A_q = A = -\mathbb P\Delta: \D(A_q) = H^{2,q}(\Omega)\cap H^{1,q}_0(\Omega) \cap L^q_\sigma(\Omega) $$ 
denote the Stokes operator.
By definition of the graph norm and the triangle inequality we see that $\mathcal D(A) = \mathcal D(I+A)$ with equivalent norms. 
For a bounded domain the Stokes operator $A$ has a bounded inverse $A^{-1}$ and there holds $0\in\rho(A)$ so that $\|u\|_A \sim \|Au\|_{L^q}$. However, for an exterior domain as well as for the whole space there holds $0\in \sigma(A)$, $A$ is densely defined with dense range $\mathcal  R(A)\subsetneq L^q_\sigma$.
This disadvantage yields several problems on adequate {\em a priori} estimates and properties of underlying function spaces, leading to severe restrictions on admissible exponents such as $q$. 
A second possibility to discuss the Stokes operator on domains where $0\in\sigma(A)$ is to consider homogeneous function spaces. 
More generally, for a closed densely defined operator $A$ on a Banach space $X, \|\cdot\|_X$ we define an extension $\dot A$ with domain $\widehat{\D}(A)$, the completion of $\D(A )$ with respect to the norm $\|A \,\cdot\|$. To be more precise,
%\begin{equation}\label{defhomogDA}  
%\widehat{\D}(A) := \textrm{completion of } \D(A ) \textrm{ with respect to } \| \cdot\|_{A} = \|\dot A \,\cdot\|.
%\end{equation}
\begin{align}\label{Def-hatAx}
\begin{aligned}
\widehat\D(A) & = \{\dot x=(x_k)_{k\in\IN}\subset \D(A): (Ax_k)_k \;\textrm{ is a Cauchy sequence in }\; X\} \\ 
\dot A\dot x & := \lim_{k\to\infty} A x_k \quad \textrm{for}\; \dot x=(x_k)_{k\in\IN} \in \widehat\D(A), \quad \|\dot x\|_{\dot A} = \lim_{k\to\infty} \|Ax_k\|. 
\end{aligned}
\end{align}
\BLACK
%\begin{align}\label{Def-hatAx}
%\begin{aligned}
%x \in \widehat\D(A)\;&\; \textrm{if and only if} \;\;\exists (x_k)\subset \D(A) \;\;\textrm{such that}\;\; (Ax_k)  \;\;\textrm{converges in}\;\; X, \\
%\dot Ax & := \lim_{k\to\infty} A x_k, \quad \D(\dot A) := \widehat\D(A) . 
%\end{aligned}
%\end{align}
By definition of completion, $(\widehat{\D}(A), \|\cdot\|_{\dot A})$ is a Banach space, but in general $\widehat{\D}(A) \not\subset X$. 
On the other hand, under the assumption that $A$ is sectorial (see Definition \ref{op} (i) below), 
\begin{equation}\label{Ahat_isom}  
\dot A: \D(\dot A) = \widehat{\D}(A) \to X\quad \textrm{is an isomorphism},
\end{equation}
{\em cf.} \cite[p. 750]{KaKuWe06}  for the general case and Lemma \ref{conv-u-D-u} below for the Stokes operator on weighted spaces. For further general results on $\widehat\D(A)$ and properties of fractional powers of $\dot A$ we refer to \cite[Appendix E]{KuWe04}.
 \BLACK 

Under the assumption of sectoriality on $A$ we may  define the fractional powers $A^\theta$, $0<\theta<1$, as densely defined closed operators with domain $\D(A^\theta)\subset X$. Exploiting the property of bounded imaginary powers ($BIP$), see Definition \ref{op} (ii), we get for the domains $\D(A^\theta)$, $0<\theta<1$,  
of  $A^\theta$
the identity  
$$\D(A^\theta) = [L^q_{\sigma,\ell}(\Omega_0), \D(A)]_\theta. $$
Since $\D(A)= \D(I+A)$ where $I$ denotes the identity, this relation extends to $I+A$.

Concerning the Stokes operator on weighted spaces  $L^q_{\sigma,\ell}(\Omega)$ with $1<q<\infty$ and $-\frac{n}{q} <\ell<\frac{n}{q'}$ we define
$$ A_{q,\ell} = -\mathbb P_{q,\ell}\Delta_{q,\ell}: \D(A_{q,\ell}) = H^{2,q}_\ell(\Omega)\cap H^{1,q}_{\ell,0}(\Omega) \cap L^q_{\sigma,\ell}(\Omega) \to L^q_{\sigma,\ell}(\Omega); $$
here $\Delta_{q,\ell}$ denotes the Dirichlet Laplacian on $L^q_\ell(\Omega)$ with domain $H^{2,q}_\ell(\Omega)\cap H^{1,q}_{\ell,0}(\Omega)$.
We mention that $A_{q,\ell}$ is a closed, injective and densely defined operator on $L^q_{\sigma,\ell}(\Omega)$ with $0\in\sigma(A_{q,\ell})$. Many times we simply write $A=-\mathbb P\Delta$.

By analogy to \eqref{Def-hatAx} we define its homogeneous version, {\em i.e.},
\begin{equation}\label{defhomogDA_s}  
\D(\dot A^\theta_{q,\ell}) := \widehat{\D}(A^\theta_{q,}\ell) := \textrm{completion of } \D(A^\theta_{q,\ell}) \textrm{ with respect to } \|A^\theta_{q,\ell}\,\cdot\|.
\end{equation}
Then Proposition \ref{domain-of-fractional-power} admits the generalization 
\begin{equation}\label{Dhat(Atheta)}
\D(\dot A^\theta) = [X, \D(\dot A)]_\theta,\quad 0<\theta<1,
\end{equation} 
to homogeneous spaces $\D(\dot A^\theta)$,  with equivalent norms; for more details of the abstract theory we refer to \cite{KaKuWe06}, \cite{KuWe04} and
Proposition \ref{Complex Interpolation} below.
For more concrete results on the Stokes operator and proofs see \cite{FS} and Sect. \ref{S3} below. The whole space case is considered in Corollary \ref{cor:Rn-domains}. The definitions $A_{q,\ell}$, $\widehat\D(\dot A_{q,\ell}^\theta)$  {\em etc.} can be extended to arbitrary weights $w\in \mathscr A_q(\IR^n)$ yielding $\dot A_{q,w}$, $\widehat\D(\dot A_{q,w}^\theta)$ {\em etc.}

\begin{thm}\label{res-weighted} {\em \cite[Theorems 1.5 and 5.5]{FS}}
Let $n \geq 3$, $\Omega \subset \mathbb{R}^n$ be an exterior domain with boundary of class $C^2$, let $q \in (1, \infty)$ and $w\in \mathscr{A}_q$. 
%Further let $\Sigma_\omega$ denote the complex sector $\Sigma_\omega =\{0\neq \mu\in\IC: |{\rm arg}\, \mu|<\omega\}$ with $0<\omega<\pi$,
\begin{enumerate}
\item[{\rm (i)}]
For $\lambda\in\Sigma_\omega$, $0<\omega<\pi$, satisfying $|\lambda|\geq\delta>0$ the Stokes resolvent problem
$\lambda u + A_{q,w} u =f$, $f\in \mathbb L^q_w(\Omega)$, 
has a unique solution $u\in {\D}(A_{w,q})$ satisfying the resolvent estimate
\begin{align}\label{equ:rse-w}
	\|\lambda u\|_{\mathbb{L}^q_w(\Omega)}  + \|A_{q,w} u\|_{\mathbb{L}^q_w(\Omega)}  \leq c_{\omega,\delta} \|f\|_{\mathbb{L}^q_w(\Omega)}.
\end{align}
\item[{\rm (ii)}]
If $w(x)=(1+|x|)^\ell$ and $-n < \ell < n(q-1)$, 
 {\em i.e.} $w\in \mathscr{A}_q$, the constant $c_{\omega,\delta}$ can be replaced by $c_{\omega}>0$, a constant independent of $\delta$. 
\item[{\rm (iii)}]
Let %$x_0 \in \mathbb{R}^n \setminus \overline\Omega$ and 
$ w^{\ell/q} \langle x\rangle^{-\gamma \ell} \in  \mathscr{A}_\ell,$ 
where 
$\gamma=n\big(\frac{2}{n}+\frac{1}{\ell}-\frac{1}{q}\big) \geq 0 $, $n\geq 3$ %(\RED $n\geq 3$ is mentioned in [8, Thm. 5.5]) 
\BLACK and $\ell \geq q$. Then 
the resolvent estimate 
\begin{align}\label{equ:rse-w2}
	\|\lambda u\|_{\mathbb{L}^q_w(\Omega)}  +  |\lambda|^\frac12 \|\nabla u\|_{\mathbb{L}^q_w(\Omega)}  +\|\nabla^2 u\|_{\mathbb{L}^q_w(\Omega)}  \leq c \|(\lambda+A_{q,w})u\|_{\mathbb{L}^q_w(\Omega)}
\end{align}
holds for $u\in\mathcal D(A_{q,w})$ and all $\lambda\in \Sigma_\omega$ uniformly.  
In particular, if  $w=\langle \cdot\rangle^\ell$   and
$$
2q-n < \ell < n(q-1), $$
then \eqref{equ:rse-w2} is satisfied. 
\item[{\rm (iv)}] For $u\in\mathcal D(A_{q,w})$ and $w=\langle\cdot\rangle^\ell$, $2q-n < \ell < n(q-1)$, there holds the estimate
\begin{align}\label{equ:nabla2-A}
	\|\nabla^2 u\|_{\mathbb{L}^q_w(\Omega)}  \leq c \|A_{q,w}u\|_{\mathbb{L}^q_w(\Omega)}.
\end{align}
\end{enumerate}
\end{thm}

Note that (iv) is obtained from \eqref{equ:rse-w2} by letting $\lambda$ tend to $0$.\\[-1ex]

On the whole space we have the following more detailed generalized resolvent estimate.

\begin{prop}\label{FS-weighted}
%{\rm (\cite[Theorem 4.5]{Farwig-Sohr})}
%\label{resolvent-est-inhomogeneous-condition}
%
Let $1<q < \infty$ with $-\frac{n}{q} < \ell < \frac{n}{q'}$.
For every $f \in L^q_\ell (\IR^n)$, $g= H^{1,q}_\ell(\IR^n) \cap \widehat{H}^{-1,q}_\ell (\IR^n)$ and $\lambda\in\Sigma_\omega$, $0<\omega<\pi$, the generalized Stokes problem 
\begin{align}\label{generalized-stokes-problem}
\lambda u -\Delta u + \nabla p =f,  \ \ \div u=g \mbox{ in } \IR^n
\end{align}
has a unique solution $(u,p) \in H^{2, q}_\ell (\IR^n) \times \widehat{H}^{1,q}_\ell(\IR^n)$ satisfying
\begin{align}\label{St-res-wRn}
\begin{aligned}
\|\lambda u\|_{L^{q}_\ell(\IR^n) } & \leq 
C \big( \|f\|_{L^q_\ell(\IR^n)} + \|\lambda g\|_{\widehat{H}^{-1,q}_\ell(\IR^n)} \big),\\[1ex]
\||\lambda|^{1/2} \nabla u\|_{L^{q}_\ell(\IR^n) } &   \leq C \big( \|f\|_{L^q_\ell(\IR^n)} + \||\lambda|^{1/2} g\|_{L^q_\ell(\IR^n)} \big),\\[1ex] 
\|\nabla^2 u\|_{L^{q}_\ell(\IR^n) } & \leq 
C\big(\|f\|_{L^q_\ell(\IR^n)}  + \|\nabla g\|_{L^q_\ell(\IR^n)} \big),\\
\|\nabla p\|_{L^{q}_\ell(\IR^n) } & \leq 
C\big(\|(f, \nabla g)\|_{L^q_\ell(\IR^n)} + \|\lambda g\|_{\widehat{H}^{-1,q}_\ell(\IR^n)} \big)
\end{aligned} 
\end{align}
uniformly in $\lambda$. 
%\RED (Actually, by the explicit formula, $\|\nabla^2 u \|_{L^{q}_\ell(\IR^n) } \leq C \|f, \nabla g\|_{L^q_\ell(\IR^n)}$, but for $\nabla p$ we only get (due to $(\lambda-\Delta)u$ that 
%$$ \|\nabla p\|_{L^{q}_\ell(\IR^n) }  \leq 
%C\big(\|(f, \nabla g)\|_{L^q_\ell(\IR^n)} + \|\lambda g\|_{\widehat{H}^{-1,q}_\ell(\IR^n)} \big)$$
\end{prop}

\begin{proof}
Parts of \eqref{St-res-wRn} are proved in \cite[Theorem 4.5]{FS} which is based on  Proposition \ref{thm-hoermander} and the explicit solution formula 
$u = (\lambda-\Delta)^{-1} \mathbb Pf - (-\Delta)^{-1}\nabla g$,  $\;\nabla p = -(-\Delta)^{-1}\nabla\div f + \nabla(-\Delta)^{-1}(\lambda-\Delta)g$. 
Recall that $\mathbb P_{q,\ell}$ is bounded on $L^q_\ell(\IR^n)$, see Lemma \ref{Helmh}. 
Indeed, $\eqref{St-res-wRn}_{1,3,4}$ are proved in \cite{FS}. 
Moreover, $\eqref{St-res-wRn}_{2}$ follows from multiplier theory. 
%Finally, $\eqref{St-res-wRn}_{4}$ is a consequence of \eqref{generalized-stokes-problem} and $\eqref{St-res-wRn}_{1,3}$.
\end{proof}

 Now we prove in a detailed way that $\dot A:\widehat \D(A) \to L^q_{\sigma,\ell}(\Omega)$ is an isomorphism.

\begin{lem}\label{conv-u-D-u} \BLACK
Let $n \geq 3$, $\Omega \subset \mathbb{R}^n$ be an exterior domain with boundary of class $C^2$ and $1<q<\infty$.

(i) Let $w\in\mathscr A_q$. Then the Stokes operator 
$A: \D(A_{q,w}) \to \R(A_{q,w}) \subsetneq L^{q}_{\sigma,w}(\Omega)$ extends to an isomorphism
$$ \dot A: \widehat\D(A_{q,w}) \to L^{q}_{\sigma,w}(\Omega).$$

(ii)
Let $2-\frac{n}{q}<\ell<\frac{n}{q'}$ and  define the exponents $r,s,\ell_r,\ell_s$ such that 
\begin{equation}\label{r,s,lr,ls}
\ell-\ell_r = 1+\frac{n}{r} - \frac{n}{q}\geq 0, \quad \ell-\ell_s = 2+\frac{n}{s} - \frac{n}{q}\geq 0. 
\end{equation}
%((In the former version with $r=\frac{nq}{n-q}$, $s=\frac{nq}{n-2q}$ we needed $q<n/2$)). 
\BLACK 
Then for any $f\in L^q_{\sigma,\ell}(\Omega)$ there exists a unique $\hat u\in {\D}(\dot A)$ such that $\dot A \hat u = f$. Moreover, there exists $p\in \widehat H^{1,q}_\ell(\Omega)$ and a unique representative 
%\begin{equation}\label{U0-res}
$\hat U\in \widehat H^{2,q}_\ell(\Omega) \cap \widehat H^{1,r}_{\ell_r,0}(\Omega) \cap L^s_{\sigma,\ell_s}(\Omega)$
%\end{equation}
of $\hat u$ such that 
\begin{equation}\label{Stokes-hat}
 -\Delta \hat U + \nabla p = f, \quad \div \hat U=0\textrm { in }\;\Omega, \;\;\hat U=0\textrm { on }\;\partial\Omega. 
\end{equation} 
In particular, $\dot A: \widehat{\D}(A_{q,\ell}) \to L^q_{\sigma,\ell}(\Omega)$
is an isomorphism.  
%\begin{equation}\label{mathcal-P}
%\mathcal P:\widehat{\D}(A_q) \to L^s_\sigma(\Omega)
%\end{equation}
%is a linear bounded  operator.
\end{lem}

\begin{proof} \BLACK
(i) For $\hat u\in \widehat\D(A_{q,w})$ there exists, by definition \eqref{Def-hatAx} a sequence $(u_k)_k\subset \D(A_{q,w})$ such that $u_k\to \hat u$, {\em i.e.}, $A_{q,w} u_k\to \dot A_{q,w} \hat u$ in $L^q_{\sigma,w}(\Omega)$ as $k\to \infty$. Since $\R(A_{q,w})$ is dense in $L^q_{\sigma,w}(\Omega)$, we get that $\dot A_{q,w}$ is surjective.

Next let $\dot A_{q,w}\hat u =0$ for $\hat u \in \widehat\D(\dot A_{q,w}).$ By definition, there exists $(u_k)\subset \D(A_{q,w})$ such that $A_{q,w} u_k\to \dot A_{q,w} \hat u =0$ in $L^q_{\sigma,w}(\Omega)$. Again, by definition, $\hat u=0$.
%\BLUE(I think that here we also used that $A$ is closed. 
%$u_k \rightarrow \hat{u}$ in $\widehat{D}(A)$, $A: D(A) \subset \widehat{D}(A) \rightarrow L^q_{\sigma,w}$, $Au_k \rightarrow 0$, and thus $A\hat{u} =0$ which implies that $\hat{u}$ is $0$ in the homogeneous domain $\widehat{D}(A)$. 
%)
\BLACK 

(ii)
The proof is not based on part (i), but on an explicit approximation procedure.

For $f\in L^q_{\sigma,\ell}(\Omega)$ and an arbitrary sequence of resolvent parameters $(\lambda_k) \subset  \Sigma_\omega = \{\lambda\in\IC: |\arg \lambda| < \omega, \lambda\neq 0\}$, $0<\omega<\pi$,  converging to $0$ as $k\to\infty$ we find solutions $(u_k)\subset \D(A_{q,\ell})$ of the resolvent problem
\begin{equation}\label{u_k-resolvent}
\lambda_k u_k -\Delta u_k + \nabla p_k=f,\;\; \div u_k=0,\;\; u_k\big|_{\partial\Omega} = 0. 
\end{equation}
By Theorem \ref{res-weighted} (iii)
%\cite[Theorems 1.2 and 5.5]{FS} 
the sequences $(\lambda_k u_k)$, $(\lambda_k^{1/2} u_k)$ and $(\nabla^2 u_k)$ are uniformly bounded in $L^q_\ell(\Omega)$.
By Sobolev embeddings, see Proposition \ref{Hardy-Rellich} (iii)  we may choose arbitrary exponents $r,s,\ell_r,\ell_s$ satisfying \eqref{r,s,lr,ls} 
%with $\ell=\ell'$, 
such that $(u_k) \subset L^s_{\sigma,\ell_s}(\Omega)$ and $(\nabla u_k) \subset L^r_{\ell_r}(\Omega)$ \BLACK are bounded. 
Hence there exist  $\dot U \in L^s_{\ell_s}(\Omega),\, \dot U_1\in L^r_{\ell_r}(\Omega)$ \BLACK and $\dot U_2, \dot V, \dot V_1\in L^q_\ell(\Omega)$ such that, for a subsequence labeled by $k$ again,
\begin{align}\label{weak1}
 u_k \rightharpoonup \dot U \textrm{ in } L^s_{\sigma,\ell_s}(\Omega),\ \ &   \nabla u_k\rightharpoonup \dot U_1 \textrm{ in } L^r_{\ell_r}(\Omega),\,\ \nabla^2 u_k \rightharpoonup \dot U_2\textrm{ in }L^q_\ell(\Omega) \BLACK \\
\lambda_k u_k \rightharpoonup \dot V \textrm{ in } L^q_{\sigma,\ell}(\Omega),\ \ &  (\lambda_k)^{1/2} \nabla u_k\rightharpoonup \dot V_1 \textrm{ in } L^q_\ell(\Omega).\label{weak2}
\end{align}
Hence $\nabla^2 \dot U= \dot U_2$ and $\nabla \dot U_1= \dot U_2$. Moreover, combining \eqref{weak1}, \eqref{weak2},  we obtain that $\dot V=0$, $\dot V_1=0$. 
Since by \cite[Theorems 1.2 and 5.5]{FS} also $(\nabla p_k)$ is bounded in $L^q_\ell(\Omega)$, there exists $p$ such that $\nabla p_k\rightharpoonup \nabla p\in L^q_\ell(\Omega)$ as $k\to \infty$. 
Thus we obtain from \eqref{u_k-resolvent} in the limit that
$$ \int_\Omega f\phi \dx = -\int_\Omega \Delta \dot U\,\phi \dx +\int_\Omega \nabla p\, \phi \dx,\quad \phi\in C^\infty_0(\Omega), $$
{\em i.e.,} \eqref{Stokes-hat} is satisfied.

Since $\nabla^2 u_k \rightharpoonup \dot U_2=\nabla^2 \dot U$, also $A_q u_k = -\mathbb P\Delta u_k \rightharpoonup -\mathbb P\Delta \dot U$. 
Then Mazur's lemma yields the existence of a sequence of convex combinations of the $u_k's$, say $(\tilde u_k)$, such that $-\mathbb P\Delta \tilde u_k \to -\mathbb P\Delta \dot U$ in $L^q_\ell(\Omega)$. This proves $\dot U\in \widehat \D(A_{q,\ell})$ and $\dot A_{q,\ell} \dot U=f$.

Concerning uniqueness consider a solution $\dot u\in \widehat \D(A_{q,\ell})$ of the equation $\dot A_{q,\ell} \dot u =0$, {\em cf.} %\eqref{U0-res} and 
\eqref{Stokes-hat} such that $\dot u\in L^s_{\sigma,\ell}(\Omega)$ and $\nabla \dot u \in L^r_\ell(\Omega)$.
Introduce a cut off function $\chi\in C_0^\infty(\IR^n)$ such that $\chi(x)=1$ for $|x|<R$ and $\supp\nabla \chi\subset B_{2R} \setminus \overline{B_R}$ where $\IR^n\setminus \Omega \subset B_R$. 
Then we write $\dot u=\chi \dot u + (1-\chi)\dot u$, consider inhomogeneous Stokes equations with nonzero divergences on $\Omega\cap \overline{B_{2R}}$ for $\chi \dot u$ and on $\IR^n$ for $(1-\chi)\dot u$, respectively, apply Bogovski\u{\i}'s operator on $B_{2R} \setminus\overline{B_R}$, and elliptic regularity results to both parts to get in a finite number of steps that $\nabla \dot u\in L^2(\Omega)$. Testing the equation $-\mathbb P\Delta \dot u = 0$ with $\dot u$ we obtain that $\dot u=0$. We omit further details of these standard arguments.  
\end{proof}

We note that Lemma \ref{conv-u-D-u} (i) holds not only for $\dot A$, but also for fractional powers $\dot A^\theta$ for any $\theta>0$, see \cite[Proposition 15.23, a)]{KuWe04}.
\BLACK

\subsection{$\mathscr H^\infty$-calculus, $BIP$ and the Stokes operator on $\IR^n$}\label{S2.4}

We recall the definition of the $\mathscr H^\infty$-calculus of a general sectorial operator $A$ on a Banach space $(X, \|\cdot\|_X)$ as follows. 
%For any $0<\theta<\pi$, let $\Sigma_\theta \subset \IC$ be the open sector 
%$$ \Sigma_\theta = \{z\in\IC: |\arg z| <\theta, z\neq 0\}.$$
On the open sector $\Sigma_\theta\subset \IC$, $0<\theta<\pi$, we consider two spaces of holomorphic functions,
\begin{align*} %\label{HH0}
\HH_0(\Sigma_\theta) & = \Big\{f: \Sigma_\theta\to \IC: \exists\, \alpha, \beta>0: \sup_{z\in \Sigma_\theta: |z|\leq 1} \frac{|f(z)|}{|z|^\alpha}  + \sup_{z\in \Sigma_\theta: |z|\geq 1} |z|^{\beta} |f(z)|<\infty \Big\},\\[1ex]
\HH^\infty(\Sigma_\theta) & = \big\{f: \Sigma_\theta\to \IC: |f|_{\infty,\theta} = \sup_{z\in \Sigma_\theta}  |f(z)|<\infty \big\} . %\label{HHinfty}
\end{align*}

\begin{defi}\label{op}
Let $A:\mathcal D(A)\subset X\to X$ be a closed injective operator with dense domain $\D(A)$ and dense range $\R(A)$ in $X$. Then the domains $\mathcal D(A^\theta)$ of fractional powers will be equipped with the graph norms 
$$ \|u\|_{A^\theta} = \|u\|_X + \|A^\theta u\|_X. $$

(i) We say that $A$ is {\em sectorial} if there exists a {\em spectral angle} $\Phi_A\in \big[0,\frac{\pi}{2}\big)$ such that for each $\phi\in (\Phi_A,\pi)$
\begin{equation}\label{classSEC}
\rho(-A) \supset \Sigma_{\pi-\phi}, \quad \sup_{z\in\Sigma_{\pi-\phi}} \|z(z+A)^{-1}\|<\infty. \end{equation}
In this case, there exists a least constant $S_A(\phi)$ such that 
$ \|z(z+A)^{-1}\| \leq S_A(\phi), \;z\in\Sigma_{\pi-\phi}$.
%
%for any $\phi\in (\Phi_A,\pi)$ there exists least constants $S_A(\phi)$  and $\hat S_A(\phi)$ such that 
%$$ \|z(z+A)^{-1}\| \leq S_A(\phi)\quad \textrm{and }\; \;\|A(z+A)^{-1}\| \leq \hat S_A(\phi), \;\;z\in\Sigma_{\pi-\phi}, $$
%respectively.

(ii) The sectorial operator $A$ has {\em bounded (purely) imaginary powers} if $A^{it}\in \mathcal L(X)$ for all $t\in\IR$ and there exists a {\em power angle} $\Theta_A$ such that for any $\theta>\Theta_A$ there exists a least constant $J_A(\theta)\geq 0$ such that
\begin{equation}\label{classBIP} 
\|A^{it}\|_{\mathcal L(X)} \leq J_A(\theta) e^{\theta|t|}, \quad t\in\IR. 
\end{equation}
We also say that $A$ has property $BIP$.

(iii) The sectorial operator $A$ is said to have a {\em bounded $\mathscr H^\infty$-calculus} or, for short, to be of {\em class $\mathscr H^\infty(X)$} if there exists an angle $\Phi_A^\infty$ such that $\Phi_A\leq \Phi_A^\infty < \frac{\pi}{2}$ and for each $\theta>\Phi_A^\infty$ there exists a least constant $H_A^\infty(\theta)\geq 0$ such that 
\begin{equation}\label{classHinfty} 
\|h(A)\|_{\mathcal L(X)} \leq H_A^\infty(\theta) |h|_{\infty,\theta}, \quad h\in \HH_0(\Sigma_\theta), 
\end{equation}
where a functional calculus $h(A)$ of $A$ is defined by 
\begin{align}\label{Atslambda2} 
\begin{aligned}
h(A) = 
\frac{1}{2\pi i}\int_{\Gamma_\omega} h(\lambda)
(\lambda-A)^{-1}\,{\rm d}\lambda,
\end{aligned} 
\end{align}
using the contour $\Gamma_{\omega}=\Gamma_{\omega +}\cup \Gamma_{\omega -}$,
$\Gamma_{\omega \pm }: \lambda=\rho e^{\pm i\omega}, \, 0 \leq \rho <\infty$, with $\Phi_A<\omega<\theta<\frac{\pi}{2}$; to be more precise, $\Gamma_\omega=(\infty,0]e^{+i\omega} \cup [0,\infty)e^{-i\omega}$. 
\end{defi}

\begin{rem}
    {\rm 
%(i) For a sectorial operator $A$ the set $\mathcal D(A) \cap\mathcal R(A)$ is dense in $X$, see \cite[Proposition 1.2]{DHP2}. Using the graph norms $\|u\|_{A^\theta}$ we will consider only non-homogeneous Banach spaces rather than the homogeneous case where $\|u\|_{A^\theta}= \|A^\theta u\|_X.$ \RED ((Do we really exclude the homogeneous case? Delete (i)?)) \BLACK

(i) In Definition \ref{op} (i), (iii) the conditions $\Phi_A < \frac{\pi}{2}$ and $ \Phi_A^\infty < \frac{\pi}{2}$ imply that $-A$ generates a bounded, analytic $C^0$-semigroup. In the general case, only $\Phi_A^\infty < \pi$ is assumed.

(ii) In \eqref{Atslambda2} the contour integral of $h(A)$, $h\in\HH_0(\Sigma_\theta)$, is absolutely convergent. Under the assumption \eqref{classHinfty} this formula can be extended to all $h\in \HH^\infty(\Sigma_\theta)$ with the help of the operator $A +2I + A^{-1}$ and its inverse on $\mathcal D(A) \cap\mathcal R(A)$, see {\em e.g.} \cite[p. 656]{NollSaal}. In other words, the estimate \eqref{classHinfty} for all $h\in \HH_0(\Sigma_\theta)$ extends to all $h\in \HH^\infty(\Sigma_\theta)$.
}
\end{rem}

We recall the following facts about the relation between the existence of the $\mathscr H^\infty$-calculus, the property {\em BIP} and fractional powers of operators, see \cite[(2.15), (2.16)]{DHP2}:
%\begin{thm}\label{angles}{\rm(\cite[(2.15), (2.16)]{DHP2}.)  }
% The following inclusions hold:
\begin{equation}\label{classesSEC}  
\mathscr H^\infty(X)\, \Rightarrow \textit{ BIP } \Rightarrow \textrm{ sectoriality}\quad\textrm{and }\; \Phi_A^\infty \geq \Theta_A\geq \Phi_A.
\end{equation}
%\end{thm}

\begin{prop}\label{Complex Interpolation}{\em(}{\rm \cite[Proposition 4.3]{FT-lin-half}, \cite[Proposition 2.2]{KaKuWe06}}{\em )} 
Let $A:\D(A)\subset X \to X$ be a sectorial operator with property {\em BIP} such that
\begin{equation}\label{Ait}
		\| A^{is}\| \leq J_A(\phi)\, e^{\phi |s|}, \quad \phi>\Theta_A, \;s\in \IR.
\end{equation}

(i) For all $0<\theta<1$ there holds 
\begin{equation}\label{CI}
	\widehat{\D}(A^\theta) = [X,\widehat{\D}(A)]_\theta
\end{equation}
and for $u\in \widehat{\D}(A^\theta)$
\begin{equation}\label{Aitest}
		\frac{1}{J_A(\phi) e^{\phi}}\|A^\theta u\| \leq \|u\|_{[X,\widehat{\D}(A)]_\theta} \leq J_A(\phi)e^{\phi} \|A^\theta u\|.
\end{equation}

(ii) Concerning non-homogeneous spaces, for all $0<\theta<1$, 
\begin{align}\label{BIP-nonhomogeneous-domain-abstr}
{\D}(A^\theta) =  [X, {\D}(A)]_{\theta} 
\end{align}
and 
$\| u\|_X + \|A^{\theta} u\|_X \simeq  \|u\|_{A^{\theta}}%\quad (not: \simeq \|A^{\theta}u\|_X)
$.
\end{prop}

For the proof of (ii) see {\em e.g.} Triebel \cite[Theorem 1.15.3]{Triebel} when $0\in\rho(A)$ and Denk, Hieber and Pr\"u{\ss} \cite[Theorem 2.5]{DHP2} when $0\notin\rho(A)$ holds.

\vspace{2ex}

Let $1<q< \infty$, $n \geq 3$, and let $\mathbb P = \mathbb P_{q,\ell}$ and $A=A_{q,\ell} = -\mathbb P \Delta$ be the Helmholtz projection and Stokes operator on $L^q_{\sigma, \ell }(\Omega)$, respectively, for $-\frac{n}{q} < \ell < \frac{n}{q'}$.  
To consider the characterization and embeddings of domains of fractional powers $A_{q,\ell}^\theta$, $0<\theta<1$, with  weight exponent $\ell$,  we will show the $\mathscr H^\infty$ calculus of $A_{q,\ell}$ on the weighted space $L^q_{\sigma, \ell }(\Omega)$.

In this Subsection \ref{S2.4} we consider the case of the whole space, $\Omega = \IR^n$; for a similar result on an exterior domain we refer to Subsect. \ref{S3.2}. Given any Muckenhoupt weight $w\in \mathscr A_q(\IR^n)$  we recall from
%from Definition \ref{Muckenhoupt-class} the notation $\mathbb L^q_w(\IR^n)$.  By 
\cite[Theorem 1.4]{FS} that $A=A_{q,w}$ is sectorial with spectral angle $\Phi_A=0$. Moreover, for the Stokes resolvent problem 
$$ \lambda u - \Delta u + \nabla p = f,\; \div u =0 \;\textrm{ on }\IR^n,$$
$f\in \mathbb L^q_w(\IR^n)$, we apply $\mathbb P$ to see that 
 $u = (\lambda+A)^{-1}\mathbb P f$. \BLACK This problem admits in Fourier space the symbol 
$$ M_\lambda(\xi) = \Big\{\Big(\delta_{jk} - \frac{\xi_j\xi_k}{|\xi|^2}\Big) \frac{1}{\lambda+|\xi|^2}\Big\}_{j,k=1}^n, $$
where the first term $ \delta_{jk} - \frac{\xi_j\xi_k}{|\xi|^2}$ is the symbol of the Helmholtz projection $\mathbb P$. It is immediate to see by Proposition \ref{thm-hoermander} that $M_\lambda \in\mathcal L (\mathbb L^q_w(\IR^n))$ for $\lambda\in\IC\setminus (-\infty,0]$ and hence $\sigma(A)=[0,\infty)$. 
For the $\mathscr H^\infty$-calculus which is based on Dunford's calculus  and the Stokes resolvent $(\lambda-A)^{-1}$ \BLACK we consider for any $h\in \HH_0(\Sigma_\theta)$  the contour integral \eqref{Atslambda2}, {\em i.e.,}
\begin{align*}
h(A)  = \mathbb P_{q,w} \mathcal F^{-1}\Big(\frac{1}{2\pi i} \int_{\Gamma_\omega} \frac{h(\lambda)}{
\lambda-|\xi|^2}\,{\rm d}\lambda\Big) \mathcal F.
\end{align*}

Now it suffices to prove that 
$$ m(\xi) = \frac{1}{2\pi i}\int_{\Gamma_\omega} \frac{h(\lambda)}{
\lambda-|\xi|^2}\, {\rm d}\lambda $$
\BLACK defines a bounded multiplier operator on $\mathbb L^q_w(\IR^n)$.
Actually, $m$ is the symbol of $h(-\Delta)$ in the $\mathscr H^\infty$-calculus of the Laplacian on $\IR^n$ which  is well-known to be bounded on $ L^q(\IR^n)$, $1<q<\infty$, see {\em e.g.} \cite{DHP2}. 
To extend this result to weighted spaces $\mathbb L^q_w$ we use the residue theorem to see that $m(\xi) = h(|\xi|^2)$. 
Indeed, the contour is approximated by the closed contour $\Gamma_\omega \cap B_R$, where $|\xi|^2<R\to\infty$, and the arc $\gamma_{R,\omega}: Re^{i[-\omega,\omega]}$ in the half plane $\Re\lambda>0$, $0<\omega<\theta$.
By the assumption on $h\in\HH_0$ the contour integrals $\int_{\Gamma_\omega\setminus B_R} \frac{h(\lambda)}{
\lambda-|\xi|^2}\dlambda$ and $\int_{\gamma_{R,\omega}}\frac{h(\lambda)}{\lambda-|\xi|^2} \dlambda$ converge to $0$ as $R\to\infty$. 
In particular, 
\begin{equation}\label{m:k=0}
    |m(\xi)|= |h(|\xi|^2)| \leq |h|_{\infty,\theta}.
\end{equation}
To apply the H\"ormander-Mikhlin multiplier theorem to $m(\xi)$ on $\mathbb L^q_w(\IR^n)$, see Proposition \ref{thm-hoermander}, it suffices to prove that there exists a constant $c=c_m>0$ such that
\begin{equation}\label{m:k=n}
|\xi|^k |\nabla^k m(\xi)| \leq c_m,\quad \xi\in\IR^n, \quad k=0,1,\ldots,n.
\end{equation}
The case $k=0$ is given by \eqref{m:k=0}. For $k=1$ note that $|\xi| \, |\nabla_\xi m(\xi)| = 2|\xi| \, |h'(|\xi|^2)| \, |\xi|$ so that it suffices to consider $\lambda h'(\lambda)$, $\lambda>0$. 
More general, for $\lambda\in \Sigma_{\omega-\varepsilon}$ with $0 < \varepsilon<\omega$ and $0<r\sim \epsilon|\lambda|$ such that the disc $D_r(\lambda)$ lies to the right of $\Sigma_\omega$, Cauchy's integral formula yields the estimate 
\begin{align*}
    |\lambda h'(\lambda)| & = \Big|\frac{\lambda}{2\pi i}\int_{\partial D_r(\lambda)} \frac{ h(\lambda')}{
(\lambda'-\lambda)^2}\, {\rm d}\lambda'\Big|\\
&  \leq \frac{c}{r} |\lambda| |h|_{\infty,\omega}\\
& \leq c_\varepsilon |h|_{\infty,\omega}.
\end{align*}
Estimates for $|\xi|^k |\nabla^k m(\xi)|$, $2\leq k\leq n$, are obtained by analogy. Thus \eqref{m:k=n} is proved, and consequently $h(A)$ is bounded on $\mathbb L^q_{\sigma,w}(\IR^n)$  with norm bounded by $C_{q,w} |h|_{\infty,\omega}.$ The admissible angle is $\omega-\varepsilon>0$, but since $\varepsilon>0$ can be chosen arbitrarily small and $\Phi_A=0$, we finally obtain that $\Phi_A^\infty=0$.
%\BLUE (I agree the proof. By Theorem \cite[Theorem 2]{KurtzWheeden} %\RED ( I do not understand this citation), \BLUE  the proof is reduced to the $H^\infty $ calculus for the Laplacian on the whole space. ) \BLACK 
Hence we proved the following result: \BLACK

\vspace{1ex}
 
\begin{prop}\label{HinftyRn}
Let $1<q<\infty$ and $w\in \mathscr A_q(\IR^n)$. Then the Stokes operator $A_{q,w}$ possesses a bounded $\mathscr H^\infty$-calculus on $\mathbb L^q_{\sigma,w}(\IR^n)$ with $\mathscr H^\infty$-angle $\Phi_A^\infty=0$.
\end{prop}

\vspace{1ex}

As consequence of 
Propositions \ref{Complex Interpolation} and
\ref{HinftyRn} we get the important

\begin{cor}\label{cor:Rn-domains}
Let $1<q< \infty$ and $n\geq 3$, and let $A=A_{q,\ell}$ be the Stokes operator on $L^q_{\sigma, \ell }(\mathbb{R}^n)$ for $-\frac{n}{q} < \ell < \frac{n}{q'}$. 
For $0<\theta < 1$ there holds for homogeneous domains
\begin{align}\label{BIP-homogeneous-domain-Rn}
\widehat{\D}(A^\theta) =  [L^q_{\sigma, \ell}, \widehat{\D}(A)]_{\theta} 
\end{align}
with the norm equivalence 
$$
 \|A^{\theta} u\|_{L^q_{\ell}(\mathbb{R}^n)} \simeq \|u\|_{[L^q_{\sigma, \ell}, \widehat{\D}(A]_{\theta} }. 
 %\simeq \|u\|_{\widehat{H}^{2\theta}_{q,\ell}(\mathbb{R}^n)}. 
$$
However, if $\mathcal D(A^\theta)$ is equipped with the graph norm, then 
\begin{align}\label{BIP-nonhomogeneous-domain-Rn}
{\D}(A^\theta) =  [L^q_{\sigma, \ell}, {\D}(A)]_{\theta} 
\end{align}
and 
$\| u\|_{L^q_{\ell}(\mathbb{R}^n)} + \|A^{\theta} u\|_{L^q_{\ell}(\mathbb{R}^n)} \simeq \|u\|_{[L^q_{\sigma, \ell}, {\D}(A)]_{\theta} }.$
\end{cor}

\vspace{1ex}

Since $\D(A) = \D(I+A)$ with equivalent norms and since with $A$ also $I+A$ has a bounded $\mathscr H^\infty$-calculus, see \cite[Proposition 2.11 (iv)]{DHP2}, \eqref{BIP-nonhomogeneous-domain-Rn} can be extended to   
$$ {\D}((I+A)^\theta) =  \big[L^q_{\sigma, \ell}, {\D}(I+A_q)\big]_{\theta} = \big[L^q_{\sigma, \ell}, {\D}(A)\big]_{\theta} = {\D}(A^\theta). $$
The same results holds for \eqref{BIP-nonhomog-Rn} below.

To work with Corollary \ref{cor:Rn-domains} concretely, we have to estimate $\|A^\theta u\|_{L^q_w}$ by $\|(-\Delta)^\theta u\|_{L^q_w}$, {\em i.e.,} by a norm in a homogeneous  Riesz \BLACK potential space. 
Here we recall the homogeneous spaces $\widehat{\mathbb H}^{2\theta,q}_w(\IR^n)$, $\theta\in\IR$, for which the norm $\|\mathcal F^{-1}(|\xi|^{\theta} \mathcal F u)\|_{\mathbb L^q_w}$ is finite, and their subspaces, 
$\widehat{\mathbb H}^{2\theta,q}_{\sigma,w}(\IR^n)$, of solenoidal vector fields, see Subsect. \ref{S2.1}.

Concerning non-homogeneous Bessel potential spaces we introduce by analogy the spaces $H^{2\theta,q}_w(\IR^n)$ and $\mathbb H^{2\theta,q}_{\sigma,w}(\IR^n)$ via the norm $\|\mathcal F^{-1}((1+|\xi|^2)^{\theta/2} \mathcal F u)\|_{ \mathbb L^q_w} $.
 
\vspace{1ex}

\begin{cor}\label{cor:Rn-domains-char}
Let $1<q< \infty$ and $n\geq 3$, and let $A=A_{q,\ell}  = -\mathbb P \Delta_{q,\ell}$ be the Stokes operator on $L^q_{\sigma, \ell }(\mathbb{R}^n)$ for $-\frac{n}{q} < \ell < \frac{n}{q'}$. 
For $0<\theta < 1$ there hold the characterizations
\begin{align}\label{BIP-nonhomog-Rn}
\D(A^\theta) =  \big[L^q_{\sigma, \ell}, \D(A)\big]_{\theta} = \big[L^q_{\ell}, \D(-\Delta)\big]_{\theta} \cap L^q_{\sigma}(\IR^n) =  H^{2\theta,q}_{\sigma,\ell}(\IR^n)
\end{align}
and
\begin{align}\label{BIP-homogeneous-domain-Rn-char}
\widehat{\D}(A^\theta) =  \big[L^q_{\sigma, \ell}, \widehat{\D}(A)\big]_{\theta} = \big[L^q_{\ell}, \widehat{\D}(-\Delta)\big]_{\theta} \cap L^1_{\rm loc,\sigma}(\IR^n) = \widehat H^{2\theta,q}_{\sigma,\ell}(\IR^n).
\end{align}
\end{cor}

\vspace{1ex}

\begin{proof} 
Concerning the non-homogeneous setting we will apply Theorem \ref{retract-co}. Hence it suffices to find a retraction
$R: \D(-\Delta) \to \D(A_{q,\ell})$ which can be extended to a bounded operator $R_0:L^q_\ell \to L^q_{\sigma,\ell}$. In the case of the whole space $ \Omega=\IR^n$ where $\mathbb P_{q,\ell}$ commutes with partial derivatives we choose $R=\mathbb P_{q,\ell}$  and as corresponding coretraction $S=j$, the trivial injection $j: \D(A) \to \D(-\Delta)$. To be more precise, we may follow \cite[Chapter 7]{Giga85} and define $R=(I+A)^{-1}\mathbb P(I-\Delta)$ which in Fourier space has the symbol
$$ \frac{1}{1+|\xi|^2}\Big(\delta_{jk}-\frac{\xi_j\xi_k}{|\xi|^2}\Big)_{j,k=1}^n (1+|\xi|^2) =  \delta_{jk}-\frac{\xi_j\xi_k}{|\xi|^2}, $$
{\em i.e.} $R=\mathbb P_{q,\ell}$.
As coretraction we take the trivial injection $S_0=j: L^q_{\ell,\sigma} \to L^q_{\ell}$. Now Theorem \ref{retract-co} implies that 
$$
\big[L^q_{\sigma, \ell}, \D(A)\big]_{\theta} = \mathbb P_{q,\ell} \big[L^q_{\ell}, \D(-\Delta)\big]_{\theta} =  \big[L^q_{\ell}, \D(-\Delta)\big]_{\theta} \cap L^q_{\sigma,\ell}(\IR^n). $$ 

In the homogeneous case we proceed by analogy. Let 
$R=\mathbb P_{q,\ell}$ which in phase space coincides with the multiplier operators 
$$ \frac{1}{|\xi|^2}\Big(\delta_{jk}-\frac{\xi_j\xi_k}{|\xi|^2}\Big)_{j,k=1}^n |\xi|^2 =  \Big(\delta_{jk}-\frac{\xi_j\xi_k}{|\xi|^2}\Big)_{j,k=1}^n .$$
With $S=j$ we arrive by Theorem \ref{retract-co} and \eqref{BIP-homogeneous-domain-Rn} at
$$ \big[L^q_{\sigma, \ell}, \widehat{\D}(A)\big]_{\theta} = \mathbb P_{q,\ell} \big(\big[L^q_{\ell}, \widehat{\D}(-\Delta)\big]_{\theta} \big) = \big[L^q_{\ell}, \widehat{\D}(-\Delta)\big]_{\theta} \cap L^1_{\rm loc,\sigma}.$$
where $[L^q_{\ell}, \widehat{\D}(-\Delta)]_{\theta}  = \widehat H^{2\theta,q}_{\ell}(\IR^n)$.
 %
%Following Y. Giga \cite[xxx]{Giga-fract}} let
%$$ R = A_{q,\ell}^{-1} \mathbb P_{q,\ell}(-\Delta): \widehat\D(\Delta)\to  \widehat\D(A).$$
%Since the symbol in Fourier space of $R$ equals 
%$$ \frac{1}{|\xi|^2} \Big(\delta_{ij}-\frac{\xi_i\xi_j}{|\xi|^2}\Big) |\xi|^2 = \delta_{ij}-\frac{\xi_i\xi_j}{|\xi|^2} $$
%we get that $R=\mathbb P_{q,\ell}$. Obviously $R$ is bounded on $\widehat\D(-\Delta)$ as well as on $L^q_{\ell}$, and $S=j$, the trivial injection from $L^q_{\ell,\sigma}$ into $L^q_{\ell}$ is a coretraction to $R$. 
%
\end{proof}

\vspace{1ex}

Note that in \eqref{BIP-homogeneous-domain-Rn-char} the intersection of $[L^q_{\ell}, \widehat{\D}(-\Delta)]_{\theta}$ is not taken with a space $L^r_{\sigma,\ell}$ since in general $[L^q_{\ell}, \widehat{\D}(-\Delta)]_{\theta}$ will not be embedded into any weighted $L^r(\IR^n)$ space. 
We refer to Proposition \ref{domain-of-fractional-power} \BLACK below  for conditions on $n,q,\ell$ so that the intersection with a space of type $L^r_{\sigma,\ell'}(\IR^n)$ will hold, see also Lemma \ref{conv-u-D-u} (ii).

\begin{section}{Main Results and Proofs}\label{S3}
We consider an exterior domain $\Omega\subset \IR^n$, $n\geq 3$, with $C^3$-boundary.  The high regularity $\partial\Omega\in C^3$ is used only in Theorem \ref{theorem-H-infty} and Corollary \ref{theorem-H-infty-2}. \BLACK
%and return to radial weights $\langle x\rangle^{\ell q}$ and the notation $L^q_\ell(\Omega_0)$. 
Let $1<q< \infty$, and let $\mathbb P = \mathbb P_{q,\ell}$ and $A=A_{q,\ell} = -\mathbb P \Delta$ denote the Helmholtz projection and Stokes operator on $L^q_{\sigma, \ell }(\Omega_0)$, respectively,  where $-\frac{n}{q} < \ell < \frac{n}{q'}$.
%\BLUE  $\ell\geq 0$\BLACK. In Corollary \ref{theorem-H-infty-2} we will consider the case when $-\frac{n}{q} < \ell < \frac{n}{q'}-2$, \BLUE $\ell\leq 0$ ?. \RED In most results $\ell\geq 0$ and $\ell\leq 0$ play no role. Therefore, in these introductory lines I would omit it.) \BLACK

\vspace*{1ex}

\subsection{Weak Stokes problems on exterior domains with weights}\label{S3.1}

In the following we prove important {\em a priori} estimates for the stationary weak $L^q_\ell$ Stokes problem in the whole space, bounded domains $D$ and the exterior domain $\Omega$. For later use it will be important to consider also the nonhomogeneous problem where $g=\div u \neq 0$ is prescribed. The weak Stokes problem for any $f\in \widehat H^{-1,q}_{\ell}(D)$ and  $g\in L^q_{\ell,(m)}(D)=\{h\in L^q_{\ell}(D): \int_D h \dx =0\}$: %((Here and in Lemma 3.1 we forgot the condition $\int_D g\dx=0$ for bounded domains $D$)) 
\BLACK  reads as follows: 
\begin{align}\label{weak Stokes}
 -\Delta u + \nabla p =f,\quad \div u=g\; \textrm{ in } D,\quad u=0\; \textrm{ on } \partial D,
\end{align}
where the velocity field $u\in \widehat H^{1,q}_{\ell,0}(D)$ and an associated pressure $p\in L^q_\ell(D)$ satisfy 
\begin{align}\label{St-weakRn}
\langle \nabla u,\nabla v\rangle - \langle p,\nabla v\rangle = \langle f, v\rangle, \quad -\langle u,\nabla \psi\rangle = \langle g,\psi\rangle  
\end{align}
for all $v\in \widehat H^{1,q'}_{-\ell,0}(D)$ and $\psi\in L^{q'}_{-\ell}(D)$. Of course, in the case of $\Omega=\IR^n$ there is no boundary condition for $u$. We note that the norm of $f\in \widehat H^{-1,q}_{\ell}(D)$ is defined as
$$ \|f\|_{\widehat H^{-1,q}_{\ell}(D)} = \sup_{v\neq 0} \frac{|\langle f, v\rangle\big|}{\|\nabla v\|_{L^{q'}_{-\ell}(D)}} $$
where the supremum is taken over all $0\neq v\in \widehat H^{1,q'}_{0,-\ell}(D)$. For further results on the weak Stokes equations in various settings we refer to \cite{BoMi90, Cat, FaSiSo93, KoSo92}.

\begin{lem}\label{lem-var-ineq-nabla-Rn-bdd}
Let $1<q<\infty$ and $-\frac{n}{q} <\ell< \frac{n}{q'}$.

(i) For any $f\in \widehat H^{-1,q}_\ell(\IR^n)$, $g\in L^q_\ell(\IR^n)$ the weak Stokes problem \eqref{weak Stokes} has a unique solution $u\in \widehat H^{1,q}_{\ell}(\IR^n)$, $p\in L^q_\ell(\IR^n)$. Moreover, there exists a constant $C=C(q,\ell,n)>0$ such that the estimate 
\begin{align}\label{var-ineq-nabla_Rn-fg}
  \|\nabla u\|_{L^q_\ell(\IR^n)} + \|p\|_{L^q_\ell(\IR^n)} \leq C\big(\|f\|_{\widehat H^{-1,q}_\ell(\IR^n)} + \|g\|_{L^q_\ell(\IR^n)}\big)
\end{align} 
holds. In particular, if $g=0$ and $u\in \widehat H^{1,q}_{\sigma,\ell}(\IR^n)$, there holds the variational inequality 
\begin{align}\label{var-ineq-nabla_Rn}
  \|\nabla u\|_{L^q_\ell(\IR^n)} \leq C\sup_{v\neq 0} \frac{|\langle \nabla u,\nabla v\rangle\big|}{\|\nabla v\|_{L^{q'}_{-\ell}(\IR^n)}} \leq C\|\nabla u\|_{L^q_\ell(\IR^n)}, 
\end{align}
where $v$ is running through all of $\widehat H^{1,q'}_{\sigma,-\ell}(\IR^n)$.  
If additionally $f\in \widehat H^{-1,q_1}_{\ell_1}(\IR^n)$, $g\in L^q_{\ell_1}(\IR^n)$ where $1<q_1<\infty$ and $-\frac{n}{q_1} <\ell_1< \frac{n}{q_1'}$, then even $u\in \widehat H^{1,q_1}_{\ell_1}(\IR^n)$, $p\in L^{q_1}_{\ell_1}(\IR^n)$. \BLACK

(ii) For any $\nabla p\in \widehat H^{-1,q}_{\ell}(\IR^n)$ there exists  $\pi\in L^q_\ell(\IR^n)$ such that $\nabla p = \nabla\pi$. \BLACK
 
(iii) Let $D \subset\IR^n$ be a bounded domain of class  $C^{1}$. %((not $C^2$)). \BLACK 
Then for any $f\in \widehat H^{-1,q}_\ell(D)$, $g\in L^q_{\ell,(m)}(D)$ there exists a unique weak solution $u\in \widehat H^{1,q}_{\ell,0}(D)$, $p\in L^q_\ell(D)$ with $\int_D p\dx = 0$ of the Stokes problem \eqref{weak Stokes}. Moreover, there exists $C=C(q,\ell,D)>0$ such that  
\begin{align}\label{var-ineq-nabla_bdd-fg}
  \|\nabla u\|_{L^q_\ell(D)}+ \|p\|_{L^q_\ell(D)} \leq C\big(\|f\|_{\widehat H^{-1,q}_\ell(D)} + \|g\|_{L^q_\ell(D)}\big).
\end{align} 
Furthermore,   if $g=0$, \BLACK then
\begin{align}\label{var-ineq-nabla-bdd}
  \|\nabla u\|_{L^q(D)} \leq C\sup_{v\neq 0} \frac{|\langle \nabla u,\nabla v\rangle\big|}{\|\nabla v\|_{L^{q'}(D)}} \leq \|\nabla u\|_{L^q(D)},
\end{align}
where $v$ is running through all of $\widehat H^{1,q'}_{\sigma,0}(D)$.  
\end{lem}

The variational estimates \eqref{var-ineq-nabla_Rn} and \eqref{var-ineq-nabla-bdd} will not be used in the following, but are proved for the sake of completeness. However, an estimate of this type for exterior domains will be a crucial tool, see \eqref{nablau A12u} below. Note that in \eqref{var-ineq-nabla_Rn},  \eqref{var-ineq-nabla-bdd} the supremum is taken over solenoidal vector fields $v\in \widehat H^{1,q'}_{\sigma,-\ell,0}$ rather than over $v\in \widehat H^{1,q'}_{-\ell,0}$.

\begin{proof} 
(i) For any $f\in \widehat H^{-1,q'}_{\ell}(\IR^n)$ and $g\in L^q_\ell(\IR^n)$ the weak Stokes problem \eqref{weak Stokes}
possesses the weak solution in $\widehat H^{1,q}_{\ell}(\IR^n)$ with associated pressure $p\in L^q_\ell(\IR^n)$   given by 
$$ u= (-\Delta)^{-1}f - (-\Delta)^{-1}\nabla p, \quad  p=g-(-\Delta)^{-1} \div f.$$
Hence %the multiplier theorem on weighted spaces $L^q_\ell(\IR^n)$, see 
Proposition \ref{thm-hoermander}  proves the estimate
$\|p\|_{L^q_\ell(\IR^n)} \leq C \big(\|g\|_{L^q_\ell(\IR^n)} + \|f\|_{\widehat H^{-1,q}_\ell(\IR^n)}\big) $ and 
$$ \|\nabla u\|_{L^q_\ell(\IR^n)} \leq C \big(\|f\|_{\widehat H^{-1,q'}_{\ell}(\IR^n)} + \|p\|_{L^q_\ell(\IR^n)}\big) 
\leq C \big(\|f\|_{\widehat H^{-1,q'}_{\ell}(\IR^n)} + \|g\|_{L^q_\ell(\IR^n)}\big).$$ 

Concerning uniqueness for $f=0$, $g=0$ in the sense of Schwartz' distributions note that $p$ is harmonic, hence a polynomial; then the condition $p\in L^q_\ell(\IR^n)$ with $\ell q>-n$ implies that $p=0$. Hence $u$ is harmonic as well, so that $\nabla u \in L^q_\ell(\IR^n)$ must vanish. Hence $u$ is unique up to  constants. 

To prove \eqref{var-ineq-nabla_Rn} we consider $u\in \widehat H^{1,q}_{\sigma,\ell}(\IR^n)$ as a weak solution with right hand side $f:=-\Delta u$ and $g=0$ so that by \eqref{var-ineq-nabla_Rn-fg}
\begin{equation}\label{var-nosigma}
\|\nabla u\|_{L^q_\ell(\IR^n)} \leq C \|f\|_{\widehat H^{-1,q'}_{\ell}(\IR^n)} = C\sup_{v\neq 0} \frac{|\langle \nabla u,\nabla v\rangle|}{\|\nabla v\|_{L^{q'}_{-\ell}(\IR^n)} } 
\end{equation}
where the supremum is taken over all $0\neq v\in \widehat H^{1,q'}_{-\ell}(\IR^n)$. 
To pass from $v\in \widehat H^{1,q'}_{-\ell}(\IR^n)$ to solenoidal test functions, $v\in\widehat H^{1,q'}_{\sigma,-\ell}(\IR^n)$, we use the Helmholtz decomposition  
\begin{equation}\label{Helmh-Rn-Hhat} 
v = [v+\nabla(-\Delta)^{-1}\div v] - \nabla[(-\Delta)^{-1}\div v] =  v_0 + \nabla \pi,\end{equation}
where $v_0, \nabla \pi \in H^{1,q'}_{-\ell}(\IR^n)$ and $\div v_0=0$. 
Since by \eqref{Helmh-Rn-Hhat} $\nabla v = \nabla v_0 + \nabla:\nabla \pi$, we get that
$$\langle \nabla u,\nabla v\rangle = \langle \nabla u,\nabla v_0\rangle + \langle \nabla u,\nabla:\nabla \pi\rangle . 
$$
By definition 
$u\in \widehat H^{1,q}_{\sigma,\ell}(\IR^n) = \overline{C^\infty_{0,\sigma}(\IR^n)}^{\|\nabla \cdot\|_{L^q_\ell(\IR^n)}}$
so that we may assume $u\in C^\infty_{0,\sigma}(\IR^n)$. In that case $\langle \nabla u,\nabla:\nabla \pi\rangle = 0$ by an integration by parts.
Hence a density argument shows that $\langle \nabla u,\nabla v\rangle = \langle \nabla u,\nabla v_0\rangle $ for each $u\in \widehat H^{1,q}_{\sigma,\ell}(\IR^n)$. 
Finally, since $\|\nabla v_0\|_{L^{q'}_{-\ell}(\IR^n)} \leq c\|\nabla v\|_{L^{q'}_{-\ell}(\IR^n)}$, \eqref{var-nosigma} yields \eqref{var-ineq-nabla_Rn}.

The proof of the additional regularity is based on the above proof of uniqueness for solutions in Schwartz' class $\mathcal S'$. \BLACK 

 (ii) By (i) the weak Stokes problem $-\Delta u +\nabla \pi =\nabla p$, $\div u=0$, on $\IR^n$ possesses a unique solution $u\in \widehat H^{1,q}_{\sigma,\ell}(\IR^n)$, $\pi\in L^q(\IR^n)$. Applying the divergence operator, we get $\Delta\pi=\Delta p$ and hence $\Delta^2 u=0$. Since $\nabla u\in L^q_\ell(\IR^n) \subset \mathcal S'(\IR^n)$, we see that $\nabla u$ is polynomial. But $\nabla u\in L^q_\ell$ shows that $\nabla u=0$. Thus $\nabla\pi=\nabla p$. \BLACK

(iii) By \cite{BoMi90, Cat, KoSo92} the weak Stokes problem \eqref{weak Stokes} possesses for any $f\in \widehat H^{-1,q}(D)$ and $g\in L^q(D)$ a unique solution $u\in \widehat H^{1,q}_{0}(D) = H^{1,q}_{0}(D)$, $p\in L^q(D)$.
%such that $-\Delta u + \nabla p=f$, $\div u=g$ and $u=0$ on $\partial D$. 
Moreover, $(u,p) $ satisfies the estimate \eqref{var-ineq-nabla_bdd-fg}.

Concerning \eqref{var-ineq-nabla-bdd} we start with $u\in \D(A)$, set $h = -\Delta u \in L^q(D)$ and consider $h$ as a functional in $H^{1,q'}_{\sigma,0}(D)^*$. 
By Hahn-Banach's theorem there exists $f\in H^{-1,q}(D)$ such that $h=f$ on $H^{1,q'}_{\sigma,0}(D)$ and $\|f\|_{H^{-1,q}(D)} = \|h\|_{H^{1,q'}_{\sigma,0}(D)^*}$. 
Since $f-h\big|_{H^{1,q'}_{\sigma,0}(D)} = 0$, \cite[Theorem III.5.3]{Galdi-steady} implies that $f-h=\nabla \pi$ where $\pi\in L^{q}(D)$. 
The crucial condition in \cite[Theorem III.5.3]{Galdi-steady}, {\em i.e.} the solvability of problem \cite[(III.3.65)]{Galdi-steady}, is justified by properties of Bogovski\u{\i}'s operator on $D$, see Proposition \ref{Bog}. Since $f = h +\nabla \pi= -\Delta u +\nabla \pi$, \eqref{var-ineq-nabla_bdd-fg} yields the estimate 
\begin{align*} 
\|\nabla u\|_{L^q(D)} \leq C\|f\|_{\widehat H^{-1,q}(D)} = C\|h\|_{\widehat H^{1,q'}_{\sigma,0}(D)^*} 
& = C\sup_{0\neq v\in H^{1,q'}_{\sigma,0}(D)} \frac{|\langle \nabla u, \nabla v\rangle\big|}{\|\nabla v\|_{L^{q'}(D)}}.
\end{align*} 
By the density of $C^\infty_{\sigma,0}(D)$ in $H^{1,q}_{\sigma,0}(D)$ the last estimate extends from $\D(A)$ to $H^{1,q}_{\sigma,0}(D)$.\BLACK
\end{proof}

\begin{rem}\label{varineq}
{\rm 
(i) The assertions Lemma \ref{lem-var-ineq-nabla-Rn-bdd} (i), (ii) hold for any weight $w\in\mathscr A_q(\IR^n)$. The only difference in the proof is the argument to show that {\em e.g.} $p\in \mathbb L^q_w(\IR^n)$ as a harmonic polynomial vanishes. Indeed, by \cite[Theorem 7.3.3 (b)]{Grafakos-C}, any weight $w\in\mathscr A_q(\IR^n)$ satisfies the following doubling condition: There exist $0<r,K<1$  such that for each cube $Q\subset \IR^n$ there holds $w(rQ)\leq K w(Q)$. Then an iterative argument shows that $w(\frac1r Q) \to 0$ as $r\to 0$, {\em i.e.}, $\int_{\IR^n} w\dx = \infty$. Thus also $\int_{\IR^n} |\pi| w\dx = \infty$ for each polynomial $\pi$, in contradiction to the condition $p=\pi\in \mathbb L^q_w(\IR^n)$. The same idea applies to $\nabla u$ {\em etc.} \BLACK

(ii) The crucial assertions in Lemma \ref{lem-var-ineq-nabla-Rn-bdd} are the variational estimates \eqref{var-ineq-nabla_Rn}  and \eqref{var-ineq-nabla-bdd} in which the supremum is taken over solenoidal vector fields only. 
Although finally only the corresponding estimates for an exterior domain will be used, see Proposition \ref{lem-var-ineq-nabla_Omega0}, Theorem \ref{A1/2-nabla_A1/2} below, the estimates \eqref{var-ineq-nabla_Rn} and \eqref{var-ineq-nabla-bdd} are of independent interest.
In \cite{BoMi90} \eqref{var-ineq-nabla-bdd} is proved by the estimate $\|\nabla u\|_{L^q(D)} \leq C\|A_q^{1/2}u\|_{L^q(D)}$ based on the identity $\D(A_q^{1/2}) = H^{1,q}_0(D) \cap L^q_\sigma(D)$. 
These latter deep results are due to Giga in \cite{Giga85}. Our proof is based on more standard methods; the crucial step is the use of Bogovski\u{\i}'s operator for bounded domains.
}
\end{rem}

\vspace*{1ex}

The exterior domain case is studied in the following Proposition \ref{lem-var-ineq-nabla_Omega0}. Since for our purposes only the {\em a priori} estimate \eqref{var-ineq-nabla} is necessary, we omitted results on solvability of the weak $L^q_\ell(\Omega)$ Stokes problem. A second variational inequality, \eqref{var-ineq-nabla-sigma}, excludes the pressure, whereas the supremum is taken over the smaller space $\widehat H^{1,q'}_{\sigma,-\ell,0}(\Omega)$. \BLACK

\begin{prop}\label{lem-var-ineq-nabla_Omega0}
Let $1<q<\infty$, and let $\Omega\subset\IR^n$ be an exterior domain of class $C^2$.
 
(i) Additionally assume that $1-\frac{n}{q} <\ell< \frac{n}{q'}$. %
Then there exists a constant $C$ such that for any  $u\in \D(A_{q,\ell})$  and $p\in L^q_\ell(\Omega)$ 
there holds the variational inequality
 \begin{align}\label{var-ineq-nabla}
  \|\nabla u\|_{L^q_\ell(\Omega)} + \|p\|_{L^q_\ell(\Omega)} \leq C\sup_{0\neq v\in \widehat H^{1,q'}_{-\ell,0}(\Omega)} \frac{\big|\langle \nabla u,\nabla v\rangle - \langle p,\div v\rangle\big|}{\|\nabla v\|_{L^{q'}_{-\ell}}} .
\end{align}

(ii) If $\nabla p\in \widehat H^{-1,q}_\ell(\Omega)$, then there exists $\pi\in L^q_\ell(\Omega)$ such that $\nabla p=\nabla \pi$.
\end{prop}

\begin{proof}
(i) Let $\eta_1\in C^\infty_0(\IR^n)$ denote a cut off function such that $0\leq \eta_1\leq 1$, $\eta_1(x)=1$ in $\Omega_R=\Omega\cap B_R$ and $\eta_1(x)=0$ in $B_{2R}^c$. 
Further let $\eta_2=1-\eta_1$. Given $u\in \D(A_{q,\ell})$, $p\in L^q_\ell(\Omega)$ and $f := -\Delta u +\nabla p$, we consider $u,p$ as solution of the weak Stokes problem 
%with an associated pressure $p$ as a (weak) solution of the Stokes problem 
\begin{equation}\label{eq-u-p-f}
-\Delta u + \nabla p = f, \;\; \div u=0\;\textrm{ in }\;\;\Omega,\;\;u=0\;\textrm{ on }\;\partial\Omega, 
\end{equation}
where $f \in \widehat H^{-1,q}_{\ell}(\Omega)$. 
Then $(\eta_j u, \eta_j p)$, $j=1,2$, solve weak Stokes problems on $\Omega_1=\Omega\cap B_{2R}$ and $\Omega_2=\IR^n$, respectively, namely
\begin{align}\label{eta_j-u-p}
\begin{aligned}
    -\Delta(\eta_j u) +\nabla(\eta_j p) =f_j, & \quad \div(\eta_j u)=g_j \;\textrm{ on }\Omega_j,\\[1ex]
    f_j = \eta_j f - 2\nabla\eta_j\cdot \nabla u -(\Delta \eta_j)u + (\nabla\eta_j)p, 
    & \quad g_j= u\cdot \nabla\eta_j,
\end{aligned}\end{align}
together with the boundary condition $\eta_1u\big|_{\partial\Omega_1}=0$. \BLACK To apply \eqref{var-ineq-nabla_bdd-fg} on $\Omega_1$ and 
\eqref{var-ineq-nabla_Rn-fg} on $\Omega_2$ we derive suitable estimates of $f_j,g_j$.

For $j=1$  note that by Poincar\'e's inequality on the bounded domain $\Omega_1$ satisfying $\overline{\Omega_1}\supset \supp\nabla\eta_1$, there holds $H^{1,q}_0(\Omega_1) = \widehat H^{1,q}_0(\Omega_1)$ and $H^{-1,q}(\Omega_1) = \widehat H^{-1,q}(\Omega_1)$. 
Hence  
\begin{align}\label{est-eta_1u}
\begin{aligned}
 \|f_1\|_{H^{-1,q}(\Omega_1)} & \leq C \big(\|f\|_{\widehat H^{-1,q}_{\ell}(\Omega)} + \|u\|_{L^q(\Omega_1)} + \|p\|_{H^{-1,q}(\Omega_1)}\big), \\[1ex]
 \|g_1\|_{L^q(\Omega_1)} & \leq C\|u\|_{L^q(\Omega_1)}.
\end{aligned}\end{align}
Then \eqref{var-ineq-nabla_bdd-fg} implies that 
$$ \|\nabla(\eta_1 u)\|_{L^q(\Omega_1)} + \| \eta_1 p\|_{L^q(\Omega_1)} \leq C\Big(\|f\|_{\widehat H^{-1,q}_\ell(\Omega)} + \|u\|_{L^q(\Omega_1)} + \|p\|_{H^{-1,q}(\Omega_1)} + \Big|\int_{\Omega_1} \eta_1 p\dx\Big| \Big).$$
The integral $\int_{\Omega_1} \eta_1 p\dx$ on the right hand side is needed to satisfy the vanishing integral  mean condition on the pressure in Lemma \ref{lem-var-ineq-nabla-Rn-bdd} (iii).  
\vspace*{1ex}

For $j=2$ we have to consider two cases.
\vspace*{1ex}

{\bf Case 1:}  $\ell<-1+\frac{n}{q'}$ or, equivalently, $-\ell>1-\frac{n}{q'}$. Note that $\supp \nabla\eta_2 \subset \overline{\Omega_1}$. 
Moreover, for any $r'\geq q'$, $\ell'\geq \ell$ satisfying $-\ell+\ell' = 1+ \frac{n}{r'} - \frac{n}{q'}$, and any test function $\phi\in C^\infty_0(\IR^n)$ by H\"older's inequality and the Sobolev embedding, see Proposition \ref{embed-ext-weight}  applied to $\phi$ with $\mathcal C=0$  in \eqref{E3.12'}, \BLACK we get that 
\begin{equation}\label{embed-phi}
\|\phi\|_{L^{q'}(B_{2R})} \leq C\|\phi\|_{L^{r'}_{-\ell'}(B_{2R})} \leq C\|\phi\|_{L^{r'}_{-\ell'}(\IR^n)} \leq C\|\nabla\phi\|_{L^{q'}_{-\ell}(\IR^n)}.
\end{equation}
For simplicity, assume that even $\Omega^c\subset B_{R/2}$ and fix a cut off function $\eta'\in C^\infty(\IR^n)$ such that $\supp \nabla\eta' \subset B_R\setminus B_{R/2}$, $\eta'(x)=0$ for $|x|\leq \frac{R}{2}$, and $\eta'(x)=1$ for $|x|\geq R$. Since $\supp f_2\subset B_R^c$ we may write
$\langle f_2,\phi\rangle = \langle f_2,\eta' \phi\rangle$.
Hence 
\begin{equation}\label{phi-eta}
\|\nabla(\eta'\eta_2\phi)\|_{L^{q'}_{-\ell}(\Omega)} \leq C\big(\|\nabla\phi\|_{L^{q'}_{-\ell}(\Omega)} + \|\phi\|_{L^{q'}(B_{2R})}\big) \leq C\|\nabla\phi\|_{L^{q'}_{-\ell}(\IR^n)},
\end{equation} 
and consequently $\|f_2\|_{\widehat H^{-1,q}_\ell(\IR^n)} \leq C\|f_2\|_{\widehat H^{-1,q}_\ell(\Omega)} $. 
Then as in the estimate of $f_1$ we obtain that 
\begin{align}\label{est-eta_2u}
\begin{aligned}
\|f_2\|_{\widehat H^{-1,q}_\ell(\IR^n)} & \leq C \|f_2\|_{\widehat H^{-1,q}_\ell(\Omega)} \leq C \big(\|f\|_{\widehat H^{-1,q}_\ell(\Omega)} + \|u\|_{L^q(\Omega_1)} + \|p\|_{H^{-1,q}(\Omega_1)} \big)\\
\|g_2\|_{L^q_\ell(\IR^n)} & \leq C\|u\|_{L^q(\Omega_1)}.
\end{aligned} \end{align} 
Finally, \eqref{var-ineq-nabla_Rn-fg} together with \eqref{est-eta_2u} implies that
$$ \|\nabla(\eta_2 u)\|_{L^q(\Omega_2)} + \| \eta_2 p\|_{L^q(\Omega_2)} \leq C\big(\|f\|_{\widehat H^{-1,q}_\ell(\Omega)} + \|u\|_{L^q(\Omega_1)} + \|p\|_{H^{-1,q}(\Omega_1)}\big).$$

Adding the estimates for $\nabla(\eta_1 u)$, $\eta_jp$, $j=1,2$,  we get from \eqref{est-eta_1u}, \eqref{est-eta_2u} the preliminary result
\begin{equation}\label{est-u-prel}
\|\nabla u\|_{L^q_\ell(\Omega)} + \|p\|_{L^q_\ell(\Omega)}
\leq C\Big(\|f\|_{\widehat H^{-1,q}_\ell(\Omega)} + \|u\|_{L^q(\Omega_1)} + \|p\|_{H^{-1,q}(\Omega_1)} + \Big|\int_{\Omega_1} \eta_1 p\dx\Big| \Big) .
\end{equation}

{\bf Case 2:}  $\ell\geq-1+\frac{n}{q'}$ or, equivalently, $-\ell\leq 1-\frac{n}{q'}$. The critical terms of $f_2$ in $\eqref{eta_j-u-p}_2$ are  those where \eqref{embed-phi} had been used in Case 1; however, \eqref{embed-phi} is not available in Case 2. \BLACK On the other hand, due to $\eqref{eta_j-u-p}_1$ in its weak formulation, we have 
$$\langle f_2,\phi\rangle = \langle f_2,\phi-c\rangle = \langle f,\eta_2(\phi-c)\rangle + \ldots$$
for any  constant $c\in\IR$;  this step is made rigorous by approximating $c$ by $(c\chi_k)$ as in the proof of Lemma \ref{D^1}. In the present case, by Lemma \ref{D^1}, $\phi-c_\phi\in H^{1,q'}_{-\ell}(\IR^n)$ and $\eta_2(\phi-c_\phi)\in \widehat H^{1,q'}_{-\ell,0}(\Omega)$ where
$$ c_\phi=\frac{1}{|B_{2R}|}\int_{B_{2R}} \phi\dx, $$
Friedrichs' inequality yields for the test function $\phi-c_\phi$ the estimate
\begin{equation}\label{Fried-phi}
\|\phi-c_\phi\|_{L^{q'}(B_{2R})} \leq C\|\nabla\phi \|_{L^{q'}(B_{2R})} \leq C\|\nabla\phi \|_{L^{q'}_{-\ell}(\IR^n)} . 
\end{equation}
 Replacing $\phi$ in \eqref{phi-eta} by $\phi-c_\phi$ and \BLACK summarizing these ideas we arrive at \eqref{est-eta_2u} and finally at \eqref{est-u-prel} as in Case 1.

To get rid of the perturbation terms involving $u,p$ on the right hand side of \eqref{est-u-prel} assume that there exist sequences $(u_k)\subset \widehat H^{1,q}_{\sigma,\ell,0}(\Omega)$, $(p_k)\subset L^q_\ell(\Omega)$ with $f_k = -\Delta u_k + \nabla p_k \in H^{-1,q}_\ell(\Omega)$ such that  
\begin{equation}\label{u_kp_kf_k}
 \|u_k\|_{\widehat H^{1,q}_\ell(\Omega)} + \|p_k\|_{L^q_\ell(\Omega)}=1,\quad\textrm{ but }\; 
\|f_k\|_{\widehat H^{-1,q}_\ell(\Omega)} \to 0 \end{equation} 
as $k\to\infty$.
Then a compactness argument implies that there exists a subsequence, labeled again by $k\in\IN$, such that $(u_k)$ converges strongly in $L^q(\Omega_1)$ and $(p_k)$ converges strongly in $H^{-1,q}(\Omega_1)$. 
Moreover, the real sequence $\big(\int_{\Omega_1} \eta_1 p_k\dx\big)$ converges in $\IR$. 
Thus by \eqref{est-u-prel} with $u$ replaced to $u_k-u_m$ we see that $(u_k)$ is a Cauchy sequence in $\widehat H^{1,q}_{\sigma,\ell,0}(\Omega)$ and hence converges strongly to a velocity field $u$ in $\widehat H^{1,q}_{\sigma,\ell,0}(\Omega)$. 
By analogy, $(p_k)$ converges to some $p \in L^q_\ell(\Omega)$.
Finally, since $(f_k)$ converges to $0$, we obtain that $-\Delta u + \nabla p =0$, $\div u =0$ in $\Omega$ and $u=0$ on $\partial\Omega$.
For a moment take injectivity of the Stokes problem for granted. Then we get that $u\equiv 0$ as well as $p\equiv \textrm{const}$; actually, $p\equiv 0$ since $p\in L^q_\ell(\Omega)$. 
However, this contradicts \eqref{u_kp_kf_k}
%the property $\|u\|_{\widehat H^{1,q}_\ell(\Omega)} =1$ which follows from 
and the strong convergence of $(u_k), (p_k)$.
%in $\widehat H^{1,q}_\ell(\Omega)$. 
Hence the above assumptions are impossible. %\BLUE (Could you teach me why this contradiction removes the reminder perturbation terms?) %\RED (Since $\|u_k\|_{\widehat H^{1,q}_\ell(\Omega)} + \|p_k\|_{L^q_\ell(\Omega)}=1$, we get in the limit  $0 = \|u=0\|_{\widehat H^{1,q}_\ell(\Omega)} + \|p=0\|_{L^q_\ell(\Omega)}=1$, a contradiction) \BLACK 

Summarizing,  we proved that \begin{align}\label{var-ineq-nabla-Om}
  \|\nabla u\|_{L^q_\ell(\Omega)} + \|p\|_{L^q_\ell(\Omega)} \leq C\|f\|_{\widehat H^{-1,q}_\ell(\Omega)} = C\sup_{v\neq 0} \frac{\big| \langle \nabla u,\nabla v\rangle - \langle p,\div v\rangle \big|}{\|\nabla v\|_{L^{q'}_{-\ell}}}  
\end{align} 
where the supremum is taken over all $0\neq v\in \widehat H^{1,q'}_{-\ell,0}(\Omega)$. \\[-1ex]

It remains to prove uniqueness, {\em i.e.}, a weak solution $u\in \widehat H^{1,q}_{\ell,0}(\Omega),\, p\in L^q_\ell(\Omega)$ of the Stokes system $-\Delta u + \nabla p =0$, $\div u =0$ in $\Omega$ and $u\big|_{\partial\Omega}=0$ is trivial. The idea is to prove that $u,p$ satisfies $u\in \widehat H^{1,2}_0(\Omega)$. 
Then testing the weak Stokes system with the solution $u$ itself, we get that $\| \nabla u\|_{L^2(\Omega)} = 0$, hence $u\equiv 0$ since $u\big|_{\partial\Omega} =0$, and $p=0$.
Note that in view of $\eqref{eta_j-u-p}_2$ $f_j,\, g_j$ have compact support in $\Omega_1$ so that the weight $\langle x\rangle^{\ell p}$ will not play any role below. We follow arguments as in \cite[Proof of Theorems 3.1 and 3.3]{KoSo92}.

In a first step assume $1<q<\infty$ and $r>n'$ such that
\begin{equation}\label{qrn}
\frac{1}{q}- \frac{1}{n}\leq \frac{1}{r} < \frac{1}{n'},
\end{equation} 
and let $\frac{1}{s} = \frac{1}{r} + \frac{1}{n}$ so that $s\leq q$, $r=s^*$, $s'=(r')^*$; moreover, $r\leq q^*$, if $q<n$, and $r<\infty$, if $q\geq n$. 
Then by Poincar\'e's inequality ($j=1$) and Sobolev embeddings ($j=2$) there holds 
$$ \|\phi\|_{L^{q'}(\Omega_1)} \leq c\|\phi\|_{L^{s'}(\Omega_1)} \leq C 
\|\nabla\phi\|_{L^{r'}(\Omega_j)}, \quad \phi\in C_0^\infty(\Omega_j),\; j=1,2,$$
respectively. Therefore, by $\eqref{eta_j-u-p}_2$ we get for $f_2$ the estimate
$$ |\langle f_2,\phi\rangle| \leq C\big(\|u\|_{L^{q}(\Omega_1)} + \|\nabla u\|_{L^{q}(\Omega_1)} + \|p\|_{L^{q}(\Omega_1)}\big) \|\nabla\phi\|_{L^{r'}(\IR^n)} $$
which shows that $f_2\in \widehat H^{-1,r}(\IR^n)$; by analogy,  $\|f_1\|_{H^{-1,r}(\Omega_1)} \leq C\|\{u,\nabla u,p\}\|_{L^{q}(\Omega_1)}$. Moreover, $\int_{\Omega_1} g_1\dx=0$, and there holds the estimate 
$$ \|g_j\|_{L^r(\Omega_j)} \leq c\|u\|_{L^r(\Omega_1)} \leq c\|u\|_{H^{1,q}(\Omega_j)},\quad j=1,2.$$ 
We obtain with the regularity and uniqueness assertions in Lemma \ref{lem-var-ineq-nabla-Rn-bdd} (i) and (iii) that  for all $r>n'$ satisfying \eqref{qrn}
\begin{equation} \label{local-result-L^r}
\nabla( \eta_j u)\in L^r(\Omega_j), \;\;\eta_j p\in L^r(\Omega_j)\;\; (j=1,2),\quad \nabla u,p\in L^r(\Omega).
\end{equation} 

To achieve \eqref{local-result-L^r} for smaller exponents $q$ avoiding the restriction $\frac{1}{q}- \frac{1}{n}\leq \frac{1}{r}$ in \eqref{qrn} consider the case that
$$\frac{1}{q} - \frac{2}{n} \leq \frac{1}{r} < \frac{1}{q} -\frac{1}{n},$$ 
and define $q_1$ by $\frac{1}{q_1} = \frac{1}{q} - \frac{1}{n}$. 
Then $r=q_1$ is an admissible exponent in \eqref{qrn} so that \eqref{local-result-L^r} holds for $r=q_1$. Hence \eqref{qrn} holds with $q$ replaced by $q_1$, and the above argument proves  \eqref{local-result-L^r} for all $\frac{1}{q} - \frac{2}{n}\leq  \frac{1}{r}  < \frac1{n'}$. By a bootstrap argument we conclude that $\nabla u,\, p\in L^r(\Omega)$ for all $1<q<\infty$ and $r>n'$. 

Consequently, for all $1<q<\infty$ and $r=2$ we obtain that a weak solution $u,p$ with $\nabla u \in L^q(\Omega),\,\, p\in L^q(\Omega)$ of the homogeneous Stokes system ($f=0, g=0$) satisfies 
\begin{equation} \label{local-result-L^2}
\nabla( \eta_j u)\in L^2(\Omega_j), \;\;\eta_j p\in L^2(\Omega_j)\;\; (j=1,2),\quad \nabla u,p\in L^2(\Omega).
\end{equation} 

It still remains to show that $u\in\widehat H^{1,2}_0(\Omega)$. Since $f_2\in \widehat H^{-1,2}(\IR^n)$ and $g_2\in L^2(\IR^n)$, Lemma \ref{lem-var-ineq-nabla-Rn-bdd} implies that the weak Stokes system on 
$\IR^n$ with right hand side $f_2,g_2$ has a unique solution $(u',p') \in \widehat H^{1,2}(\IR^n) \times L^2(\IR^n)$. 
Then $(v,\pi)= (u'-\eta_2 u, p'-\eta_2 p)$ is a weak solution of the homogeneous Stokes system. In particular, in view of \eqref{local-result-L^r} with $r=2$, $\pi\in L^2(\IR^n)$ and $\pi$ is harmonic so that $\pi=0$. 
Moreover, due to the embeddings $\widehat H^{1,2}_0(\IR^n) \hookrightarrow L^{2^*}(\IR^n)$ and $\widehat H^{1,q}_{\ell,0}(\IR^n)\hookrightarrow L^{q^*}_\ell(\IR^n)$ we obtain that $u'\in L^{2^*}(\IR^n)$ and $\eta_2 u\in L^{q^*}_\ell(\IR^n)$; note that for the second embedding we used the assumption  $1-\frac{n}{q}<\ell$.
Hence $v=u'-\eta_2 u\in L^{2^*} + L^{q^*}_\ell$ is harmonic and satisfies the integral mean formula 
$v(x)=\frac{1}{|B_\rho(x)|} \int_{B_\rho(x)} (u'-\eta_2 u)\dy$ for any $x\in\IR^n$ and $\rho>0$.
Here the second part admits the estimate
\begin{align*} 
\frac{1}{|B_\rho(x)|} \int_{B_\rho(x)} |\eta_2 u(y)|\,\langle y\rangle^\ell \langle y\rangle^{-\ell} \dy & \leq \frac{c}{\rho^n} \Big(\int_{B_\rho(x)} (|\eta_2 u|\langle y\rangle^{\ell})^{q^*} \dy \Big)^{\frac1{q^*}} \Big(\int_{B_{\rho+|x|}(0)} \langle y\rangle^{-\ell(q^*)'}\dy\Big)^{\frac1{(q^*)'}}\\
& \leq \frac{c}{\rho^n} \|\eta_2 u\|_{L^{q^*}_\ell} (\rho+|x|)^{-\ell+n-\frac{n}{q^*}} \\
& \leq C \rho^{-\ell+(1-n/q)} \to 0 
\end{align*} 
as $\rho\to\infty$ since $\ell>1-\frac{n}{q}$. A similar estimate with $\ell=0$ and  $2^*$ for $q^*$ holds for $u'(x)$. Consequently, $v=u'-\eta_2 u \equiv 0$, $\eta_2 u=u' \in L^{2^*}(\IR^n)$ and $\nabla (\eta_2 u) \in L^2(\IR^n)$. Then \cite[Theorem II.7.5]{Galdi-steady} implies that $u\in \widehat H^{1,2}_0(\Omega)$; hence it is an admissible test function for the weak homogeneous Stokes system on $L^2(\Omega)$ and yields $u=0$. \BLACK

Now the proof of uniqueness is complete.\\

%\footnote{ 
%We refer to \cite{KoSo92} for the exterior weak Stokes problem in $L^q$ for which \RED $q<n$ \BLACK  is needed to get injectivity, and to \cite{FS} for the (strong) resolvent problem of the weak Stokes problem in weighted $L^q_\ell$ spaces where injectivity holds for all resolvent parameters $\lambda\neq 0$, $1<q<\infty$ and $-\frac{n}{q}<\ell<\frac{n}{q'}$.}.

(ii)  With $\eta', \eta_2\in C^\infty(\IR^n)$ as in (i) we consider $\eta_2 p = \eta'\eta_2 p$ as a distribution on $\IR^n$ and estimate as follows: 
\begin{align}\label{1-eta1-p}
\begin{aligned} 
\|\nabla(\eta'\eta_2 p)\|_{H^{-1,q}_\ell(\IR^n)} 
& \leq C\big( \|\nabla p\|_{H^{-1,q}_\ell(\Omega)} + \|\nabla(\eta'\eta_2 p)\|_{H^{-1,q}_\ell(\Omega_R)} \big)\\
& \leq C\big( \|\nabla p\|_{H^{-1,q}_\ell(\Omega)} + \| p\|_{L^q_\ell(\Omega_R)} \big).
\end{aligned}
\end{align}
In the first step we differ between $\ell<-1+\frac{n}{q'}$, see Case I above, and $\ell\geq -1+\frac{n}{q'}$, see Case 2, with $f_2$ replaced by $\nabla(\eta'\eta_2 p)$, and apply \eqref{embed-phi}, \eqref{phi-eta}
and \eqref{Fried-phi}, respectively. The second step is trivial.
\BLACK
%\RED (This is the proof by Borchers-Miyakawa of their Thm. 3.5 (iii) on p. 207. The estimate of $(\eta_1 \phi)$ is based on their proof of Thm. 3.5 (ii), that is our Case 1 and Case 2. Therefore, the estimate of $\phi$ again needs two cases, \eqref{embed-phi} and \eqref{Fried-phi}, since we do not exploit $\ell>1-n/q$ here)) 
%\BLUE (I have several questions here. 
 %(3)  How to estimate $\|p\|_{L^q(\Omega_R)}$? 
%Here we use $\|\nabla(\eta_1\phi)\|_{L^{q'}_{-\ell}(\Omega)} \leq C\big(\|\nabla\phi\|_{L^{q'}_{-\ell}(\Omega)} + \|\phi\|_{L^{q'}(B_{2R})}\big) \leq C\|\nabla\phi\|_{L^{q'}_{-\ell}(\IR^n)}$ in \eqref{1-eta1-p} as you recently wrote, not Poincar\'e's inequality. 
%)
Hence Lemma \ref{lem-var-ineq-nabla-Rn-bdd} (ii)  
%together with \cite[Proposition 3.8 (i)]{BoMi90}, a result for bounded domains analogous to Lemma \ref{lem-var-ineq-nabla-Rn-bdd} (ii), 
yields a function $\tilde\pi\in L^q_\ell(\IR^n)$ such that $\nabla\tilde\pi = \nabla((1-\eta_1)p)$ on $\IR^n$. Thus
$$ \nabla p = \nabla((1-\eta_1)p) + \nabla(\eta_1 p) = \nabla (\tilde\pi+\eta_1 p)\quad\textrm{in }\Omega.$$
Obviously, $\pi = \tilde\pi+\eta_1 p\in L^q_\ell(\Omega)$ solves (ii). 
\end{proof}

\vspace*{1ex}

Our main aim is to improve \eqref{var-ineq-nabla}
%\eqref{var-ineq-nabla-Om} 
to \eqref{var-ineq-nabla-sigma} below where the test functions run through the set of all {\em solenoidal} $v\in \widehat H^{1,q'}_{\sigma,-\ell,0}(\Omega)$.

\begin{thm}\label{var-ineq-sol}
Assume $1<q<\infty$, and $1-\frac{n}{q} <\ell< \frac{n}{q'}$. Let $\Omega\subset\IR^n$ be an exterior domain of class $C^2$.
Then there exists a constant $C$ such that for any  $u\in \widehat \D(A_{q,\ell})$ 
there holds the variational inequality 
\begin{align}\label{var-ineq-nabla-sigma}
  \|\nabla u\|_{L^q_\ell}\; \leq\; C\sup_{0\neq v\in \widehat H^{1,q'}_{\sigma,-\ell,0}(\Omega)} \frac{\big|\langle \nabla u,\nabla v\rangle\big|}{\|\nabla v\|_{L^{q'}_{-\ell}}} \;\leq\; C\|\nabla u\|_{L^q_\ell}.
\end{align}
\end{thm}

\begin{proof}
Given $u\in \D(A_{q,\ell})$ we set $h = -\Delta u \in L^q_\ell(\Omega)$ which will be considered as a functional in $\widehat H^{1,q'}_{\sigma,-\ell,,0}(\Omega)^*$.
By Hahn-Banach's theorem there exists $f\in \widehat H^{-1,q}_\ell(\Omega)$ such that $h=f\big|_{\widehat H^{1,q'}_{\sigma,-\ell,0}(\Omega)}$ and $\|f\|_{\widehat H^{-1,q}_\ell(\Omega)} = \|h\|_{(\widehat H^{1,q'}_{\sigma,-\ell,0}(\Omega))^*}$. 
Since $f-h\big|_{\widehat H^{1,q'}_{\sigma,-\ell,0}(\Omega)} = 0$, de Rham's Theorem  \cite[Lemma II.2.2.1]{Sohr} yields a distribution $p\in L^q_{\rm loc}(\Omega)$ such that $f-h=\nabla p$ on $\Omega$. Moreover, $\nabla p\in \widehat H^{-1,q}_\ell(\Omega)$. Then Proposition \ref{lem-var-ineq-nabla_Omega0} (ii) %Proposition \ref{Bog-ext} (ii) 
yields $\pi\in L^{q}_{\ell}(\Omega)$ such that $f-h = \nabla p = \nabla \pi$ on $\widehat H^{1,q'}_{\sigma,-\ell,0}(\Omega)$. Since $f = h +\nabla \pi= -\Delta u +\nabla \pi$, \eqref{var-ineq-nabla-Om} applied to \eqref{eq-u-p-f}
%\eqref{var-ineq-nabla_bdd-fg} and the fact that $\langle \nabla:\nabla \pi, \nabla v\rangle = 0$ for $v\in C^\infty_{0,\sigma}(\Omega)$ 
yields the estimate 
\begin{align*} 
\|\nabla u\|_{L^q(\Omega)}  \leq C\|f\|_{\widehat H^{-1,q}_\ell(\Omega)} = C\|h\|_{(\widehat H^{1,q'}_{\sigma,-\ell,0}(\Omega))^*}
& = C\sup_{v\neq 0} \frac{|\langle \nabla u, \nabla v\rangle\big|}{\|\nabla v\|_{L^{q'}_{-\ell}}}
\end{align*}  
where $v\neq 0$ runs through the set $H^{1,q'}_{\sigma,-\ell,0}(\Omega)$.
%Finally, a density argument extends the result to $u\in \widehat \D(A_{q,\ell}^{1/2})$.
\end{proof}  
 §3.2 is also completed. \BLACK 

%\begin{cor}\label{domains:A1/2-nabla_A1/2}
%Let $\Omega_0$ be a $C^2$ exterior domain in $\IR^n$ and assume $1<q<\infty$ and $1-\frac{n}{q} <\ell< \frac{n}{q'}$. 

%(i) There holds $\widehat \D(A_{q,\ell}^{1/2}) = \widehat H^{1,q}_{0,\sigma,\ell}(\Omega_0)$.

%(ii) There holds $\widehat \D((-\Delta_{q,\ell})^{1/2}) = \widehat H^{1,q}_{0,\ell}(\Omega_0)$. In particular,  
%$$ \|(-\Delta_{q,\ell})^{1/2} u\|_{L^q_\ell} \sim \|\nabla u\|_{L^q_\ell} $$
%with embedding constants independent of $u\in \widehat H^{1,q}_{0,\ell}(\Omega_0)$.
%\end{cor}

%\begin{proof} 
%(i) is an immediate consequence of Theorem \ref{A1/2-nabla_A1/2}. The proof of (ii) follows the lines of (i), but is shorter, since solenoidality does not enter the problem. 
%\end{proof}

\subsection{Construction of the $\mathscr H^\infty$-calculus in Exterior Domains}\label{S3.2}

\begin{thm}\label{theorem-H-infty}
Let $n\geq3 $, $1<q< \infty$, and $\Omega\subset\IR^n$ be an exterior domain with boundary of class $C^3$. Moreover, let $A_{q,\ell} = -\mathbb P \Delta$ be the Stokes operator on $L^q_{\sigma, \ell }(\Omega)$ for $2-\frac{n}{q} < \ell < \frac{n}{q'}$  satisfying $0 \leq \ell$.
Then there exists $C>0$ such that for any $\phi\in (0,\frac{\pi}{2})$ and $h\in \HH_0(\Sigma_\phi)$
%\BLUE 
%\begin{align}
%\begin{aligned}
%& {\rm (i)} \quad \Big\|\frac{1}{2\pi i}\int_{\Gamma_{\omega,\infty}} h(\lambda)
%(\lambda +A_q)^{-1} f \,{\rm d}\lambda \Big\|_{L^q_\ell(\Omega_0)}
%\leq C |h|_{\infty,\phi\BLACK} \|f\|_{L^q_\ell(\Omega_0)},  \\ 
%& {\rm (ii)} \quad \Big\|\frac{1}{2\pi i}\int_{\Gamma_{\omega,1}} h(\lambda)(\lambda +A_q)^{-1}
%f \,{\rm d}\lambda \Big\|_{L^q_\ell(\Omega_0)}
%\leq C |h|_{\infty,\phi\BLACK} \|f\|_{L^q_\ell(\Omega_0)},  
%\end{aligned} 
%\end{align}
%\BLACK 
%where 
%$\Gamma_{\omega,1 }: \lambda=\rho e^{\pm i\omega}, \, 0 \leq \rho <1$ and $\Gamma_{\omega,\infty }: \lambda=\rho e^{\pm i\omega}, \, 1 \leq \rho < \infty$. 
%\BLACK 
%
\begin{align}\label{H-infty}
\|h(A_{q,\ell})\|_{\mathcal{L}(L^q_{\sigma,\ell}(\Omega))} \leq C |h|_{\infty,\phi\BLACK}.    
\end{align}
In particular, $A_{q,\ell}$ possesses a bounded $\mathscr H^\infty$-calculus on $L^q_{\sigma.\ell}(\Omega)$ with $\mathscr H^\infty$-angle $\Phi_ {A}^\infty=0$.

%(\RED To my opinion, using the notation introduced below, with $u=u(f,\lambda,\Omega_0) = u\eta_1 + u\eta_2 = (v_1+w_1) + (v_2+w_2)$ we may write 
%\begin{align*}
%\int_{\Gamma_\omega} h(\lambda) & (\lambda-A_{\Omega_0})^{-1}f \dlambda = \int_{\Gamma_\omega} h(\lambda) (v_1+v_2) \dlambda + \dots\\
%& = \int_{\Gamma_\omega} h(\lambda) (\lambda-A_{\Omega_1})^{-1} \mathbb P_{\Omega_1}(\eta_1 f) \dlambda + \ldots\\
%& \quad  + \int_{\Gamma_\omega} h(\lambda) (\lambda-A_{\Omega_2})^{-1} \mathbb P_{\Omega_2}(\eta_2 f) \dlambda + \ldots.
%\end{align*}
%where the known $\mathcal H^\infty$ calculi for $A_{\Omega_j}$ maybe applied to the two contour integrals involving $\mathbb P_{\Omega_1}(\eta_1 f)$ and $\mathbb P_{\Omega_2}(\eta_2 f)$.
%All other integrands include $\lambda$ directly or indirectly (e.g. in $u,g_j$, hence in $w_j$) and therefore need an estimate by $|\lambda|^{+\alpha}$ for small $\lambda$ and   $|\lambda|^{-\alpha}$ for large $\lambda$, with $\alpha>0$. There the splitting of $\Gamma_\omega$ into two parts is natural.
\BLACK
\end{thm}

\begin{proof}
To prove Theorem \ref{theorem-H-infty}, we apply the cut off procedure. Let $(u,p)$ be a solution to the resolvent problem \begin{align}\label{res-eq}
\lambda u  -\Delta  u + \nabla p =f,  \ \ \div u = 0\mbox{ in } \Omega, \quad u=0 \mbox{ on } \partial\Omega.
\end{align}  
Recall the domains $\mathbb{R}^n\setminus \Omega \subset B_R(0)$, 
$\Omega_1:=\Omega \cap B_{2R}(0)$  
and $\Omega_2:=\mathbb{R}^n$, and the cut off functions $\eta_1$ and $\eta_2=1-\eta_1$ from the proof of Proposition \ref{lem-var-ineq-nabla_Omega0}. 
%$\eta_1 \in C^\infty_0(\mathbb{R}^n)$, $0 \leq \eta_1 \leq 1$, $\eta_1=1$ on $B_R(0)$, $\supp \eta_1 \subset B_{2R}(0)$ and $\eta_2 = 1- \eta_1$. 
Then  $(\eta_j u, \eta_j p) $ for $j=1,2$ solve the resolvent equation on $\Omega_j$:
\begin{align}\label{eq-eta-u}
\lambda\eta_j u  -\Delta (\eta_j u) + \nabla (\eta_j p) =f_j,  \ \ \div (\eta_j u) =g_j \mbox{ in } \Omega_j, 
\end{align}
where 
\begin{equation}\label{eq-eta-u-f-eta-j}
f_j = \eta_j f -2 \nabla u \cdot \nabla \eta_j -u \Delta \eta_j + p\nabla \eta_j,  \ \ 
g_j= u\cdot \nabla \eta_j. 
\end{equation}

We exploit that the $\lambda$-independent terms $f\eta_j$ can be estimated directly by the $\mathscr H^\infty$-calculus results \cite[Theorem 16]{NollSaal} on $\Omega_1$ and Proposition \ref{HinftyRn} on $\Omega_2$. Indeed, for any $h\in\mathcal H^\infty_0(\Sigma_\omega)$ and any $\phi\in (0,\frac{\pi}{2})$ there holds
\begin{align}\label{Hinfty-f eta j}
\begin{aligned}
\sum_{j=1}^2 & \Big\|\frac{1}{2\pi i} \int_{\Gamma_{\omega}} h(\lambda)
(\lambda -  \BLACK A_{\Omega_j})^{-1} \mathbb P_{\Omega_j}(f\eta_j) \dlambda \Big\|_{L^q_\ell(\Omega_j)}\\
%& \RED = 
%\sum_{j=1}^2 \Big\|\frac{-1}{2\pi i} \int_{\Gamma_{\omega}} h(\lambda)
%((-\lambda) + A_{\Omega_j})^{-1} \mathbb P_{\Omega_j}(f\eta_j) \dlambda \Big\|_{L^q_\ell(\Omega_j)}\BLACK \\
& \leq  C\sum_{j=1}^2|h|_{\infty,\phi} \|f\eta_j\|_{L^q_\ell( \Omega_j) \BLACK}\\
& \leq C|h|_{\infty,\phi} \|f\|_{L^q_\ell(\Omega_0)};
\end{aligned}
\end{align}
Now we subtract in \eqref{eq-eta-u} the function $u_j$ defined as the solution of the resolvent problem 
$$
\lambda u_j -\Delta u_j + \nabla  p_j = \mathbb P(f\eta_j), \quad \div u_j=0 \;\textrm{in } \Omega_j, \;u_1=0 \textrm{ on }\partial\Omega_1.
$$
Then $ u_j':= u\eta_j-u_j$, $j=1,2$, satisfy the generalized Stokes resolvent problems
\begin{align}\label{eq-eta-uu}
\lambda u_j' -\Delta u_j' + \nabla p_j' = f_j',  \ \ \div u_j' =g_j \mbox{ in } \Omega_j, 
\end{align}
with $u_1' = \eta_1 u-u_1=0$ on $\partial\Omega_1$, and
\begin{align}\label{eq-f-eta-u} 
f_j' & = (I-\mathbb P_{\Omega_j})(f\eta_j)-2 \nabla u \cdot \nabla \eta_j -u \Delta \eta_j + p\nabla \eta_j,  \ \ 
g_j= u\cdot \nabla \eta_j.
\end{align}

Next we write $u_j' =: v_j+w_j$ where $w_j$ is the solution to a generalized resolvent problem on $\Omega_j$ with resolvent parameter $\lambda=1$, $j=1,2$:
\begin{align}\label{w1-eq}
 w_1 -\Delta w_1 + \nabla  p^{w_1} =0, & \ \ \div w_1  =g_1 \mbox{ in } \Omega_1,  \ \ w_1=0  \mbox{ on } \del\Omega_1, \\
  \label{w2-eq}
 w_2  -\Delta w_2  + \nabla  p^{w_2} =0, & \ \ \div w_2  =g_2 \mbox{ in } \Omega_2. 
\end{align}
\BLACK
Moreover,  $v_j$ $(j=1,2)$ solves the Stokes resolvent equations
\begin{align}\label{v_j-res}
   (\lambda - \Delta ) v_j   + \nabla ( p_j'\BLACK- p^{w_j}) = f_j' +\BLACK (1-\lambda ) w_j ,  \ \ \div v_j  =0  \mbox{ in } \Omega_j. 
\end{align}
In addition, $v_1= (\eta_1 u-u_1) -w_1=0$ on $\del \Omega_1$. Hence  we get that
$$
v_j = (\lambda + A_{\Omega_j,q,\ell})^{-1} 
\mathbb{P}_{\Omega_j} (f_j' +\BLACK (1-\lambda)w_j)
$$
where in $\mathbb{P}_{\Omega_j}f_j'$, see \eqref{eq-f-eta-u}, the term $\mathbb P_{\Omega_j}(I-\mathbb P_{\Omega_j})(f\eta_j)=0$.
It follows from the identity 
$$
\lambda (\lambda + A_{\Omega_j,q,\ell})^{-1}
\mathbb{P}_{\Omega_j}w_j= \mathbb{P}_{\Omega_j}w_j -A_{\Omega_j,q,\ell}(\lambda + A_{\Omega_j,q,\ell})^{-1}
\mathbb{P}_{\Omega_j}w_j
$$
that 
$$
 v_j +  w_j = \sum_{k=1}^4 Q_{j,k}(u,p),  
$$
with 
\begin{align*}
  Q_{j,1}(u,p) & = (\lambda + A_{\Omega_j,q,\ell})^{-1}\mathbb{P}_{\Omega_j}  f_j'\BLACK, \\
  Q_{j,2}(u,p) & = (\lambda + A_{\Omega_j,q,\ell})^{-1}
\mathbb{P}_{\Omega_j}w_j\\
Q_{j,3}(u,p) & = A_{\Omega_j,q,\ell}(\lambda + A_{\Omega_j,q,\ell})^{-1}
\mathbb{P}_{\Omega_j}w_j\\
Q_{j,4}(u,p) & =  (I-\mathbb{P}_{\Omega_j}) w_j.  
\end{align*}
%\BLUE 
%The term $Q_{2,4}(u,p)$ is by \cite[Corollary 4.4]{FS} a gradient field  which in \eqref{v_j-res}  will be considered as part of the pressure and hence can be ignored in estimates of $\eta_2 u$; actually, $(I-\mathbb P_{\IR^n})w_2 = (I-\mathbb P_{\IR^n})u_2' = (I-\mathbb P_{\IR^n})(\eta_2 u)$.   \RED (But this statement does not help. These lines can be omitted, isn't it?)

The terms $Q_{j,1}, Q_{j,2},Q_{j,3}$  which depend on $\lambda$ via the resolvents $(\lambda+A)^{-1}$ will be estimated pointwise in $\lambda$ for small and large $|\lambda|$ whereas we will apply the $\mathscr H^\infty$-calculus of  $A_{\Omega_j,q}$ without weight to $Q_{j,4}$ and $w_j$.

First we consider the case of %estimate $Q_{2,1}, Q_{2,2},Q_{2,3}$ for 
large  $|\lambda|$, $\arg\lambda = \phi\in (0,\frac{\pi}{2})$.  
%On the weighted spaces $L^q_\ell$, we focus on $\Omega_2 = \mathbb{R}^n$ because any weight function $\langle x\rangle^\ell$ is bounded in estimates on the bounded domain $\Omega_1$  so that $L^q_\ell(\Omega_1) = L^q(\Omega_1)$.  
By \cite[Theorem 1.2, (1.3)]{FS94} when $j=1$, and   Proposition \ref{FS-weighted}, 
%\ref{generalized-stokes-problem}, 
$\eqref{St-res-wRn}_1$ when $j=2$ we know that 
\begin{align} \label{w-est0}
\|w_j\|_{ L^{q}_{\ell}  (\Omega_j)} \leq C\|g_j\|_{\widehat H^{-1,q}_{\ell,0\; \BLACK}(\Omega_j)}  = C\sup_{\varphi\neq 0} \frac{\int_{\Omega_j}  u\cdot \nabla \eta_j \,\varphi\dx}{\|\varphi\|_{\widehat H^{1,q'}_{-\ell}(\Omega_j)} }, 
%\leq C\|\nabla \eta_j u\|_{L^q_\ell(\{R \leq x \leq 2R\})} 
%\leq C \|u\|_{L^q(\Omega_0)}. 
\end{align}
 where $\widehat H^{-1,q}_{\ell,0}(\Omega_1) = \big(\widehat H^{1,q'}_{-\ell}(\Omega_1)\big)^*$. \BLACK
In the case $j=1$ note that $\int_{\Omega_1} u\cdot\nabla\eta_1 \dx =0$ so that we may assume that $\int_{\Omega_1}\varphi\dx =0$ and Friedrichs' inequality can be applied to $\varphi$ on $\Omega_1$.  Since the weight $\langle x\rangle^{-\ell}$ has no influence on the bounded domain $\Omega_1$, there holds
\begin{equation} \label{u-nabla-eta-j_phi}
\Big|\int_{\Omega_1} u\cdot\nabla\eta_1\, \varphi\dx\Big|
\leq C\|u\|_{L^q_\ell(\Omega_1)} \|\nabla\varphi\|_{L^{q'}_{-\ell}(\Omega_1)} .
\end{equation}
Similarly, since 
$$\int_{\Omega_2} u\cdot\nabla\eta_2 \dx = - \int_{\Omega_1} u\cdot\nabla\eta_1 \dx = 0, $$
we assume again that $\int_{\Omega_1}\varphi\dx =0$. Hence 
\eqref{u-nabla-eta-j_phi} holds for $j=2$ as well.

%((\RED In the argument below we need that if $j=2$ for Hardy, that $1-n/q' <-\ell <n/q'-1$, or equivalently, $\ell < -1+n/q'$. But we only know that $-n/q<\ell<n/q'$. Thus for $-\ell\leq 1-n/q'$ we need a further argument such as Lemma 2.16.)) ... we apply Hardy's inequality \eqref{Hardy-ineq} to $\varphi$ and see that due to the compact support of $\nabla\eta_2$ the integral in \eqref{w-est0} is bounded by 
%\begin{equation}\label{hatW-1toLq}
%\Big|\int_{\Omega_2} u\cdot\nabla\eta_2\, \varphi\dx\Big|
%= \Big|\int_{\Omega_1} u\cdot\nabla(\eta_2-1) \varphi\dx \Big| \leq C\|u\|_{L^q_\ell(\Omega_1)} \|\nabla\varphi\|_{L^{q'}_{-\ell}(\Omega_1)} .\BLACK ))\end{equation}
%\RED (In the second integral I removed power 2 for $(1+|x|)^\pm$) 
%
Summarizing \eqref{w-est0} together with \eqref{u-nabla-eta-j_phi},  %\eqref{hatW-1toLq}, 
we obtain in both cases \BLACK that
\begin{equation}\label{w-est}
 \|w_j\|_{ L^{q}_{\ell}  (\Omega_j)} \leq C\|u \cdot\nabla \eta_j \|_{L^q_\ell(\Omega_j)} 
\leq C \|u\|_{L^q_\ell(\Omega)}. 
\end{equation}
Modifying the estimate of the second term in \eqref{w-est}, we also get that  
\begin{equation}\label{w-est3}
 \|w_j\|_{ L^{q}_{\ell}  (\Omega_j)} 
\leq C \|\nabla u\|_{L^q_\ell(\Omega)}. 
\end{equation}
% ((\RED no longer needed: Indeed, if $j=2$, we write the integrand as $(u\cdot\nabla\eta_2) \varphi =  
%\frac{u}{1+|x|} \,\nabla\eta_2\,(1+|x|)^2 \, 
%\frac{\varphi}{1+|x|}$ and apply Hardy's inequality to both $u$ and $\varphi$; %(\RED I removed powers 2 and changed power 4 to 2) 
%here, $u$ is extended by $0$ to $\IR^n$ to yield a vector field $u\in H^{1,q}(\IR^n)$. If $j=1$, then by Poincar\'e's inequality on $\Omega_1$ we get that
%$\|u\cdot \nabla \eta_1\|_{L^q_\ell(\Omega_1)} 
%\leq C \|u\|_{L^q(\Omega_1)} \leq C \|\nabla u\|_{L^q(\Omega_1)}.$ \BLACK ))

Thus $Q_{j,2}(u,p)$, $j=1,2$, is estimated by Proposition \ref{FS-weighted} and \eqref{w-est} as 
\begin{align}\label{Qj2}
\begin{aligned}
\|Q_{j,2}(u,p)\|_{L^q_{\ell}(\Omega_j)} & \leq 
C\dfrac{1}{|\lambda|} \|\mathbb{P}_{\Omega_j} w_j \|_{L^q_{\ell}(\Omega_j)}  
\leq C\dfrac{1}{|\lambda|} \|w_j\|_{L^q_{\ell}(\Omega_j)} \\
& \leq  C\dfrac{1}{|\lambda|} \|u\|_{L^q_\ell(\Omega)}\\
& \leq C\dfrac{1}{|\lambda|^2} \|f\|_{ L^q_\ell(\Omega)} .
\end{aligned}
\end{align}
for large $|\lambda|$.

For later use, we prepare further estimates for $w_j$. The generalized resolvent estimates \cite[Theorem 1.2]{FS94} for $\Omega_1$,  together with Poincar\'e's and Friedrichs' inequality, \BLACK and Proposition \ref{FS-weighted}, $\eqref{St-res-wRn}_3$ for $\Omega_2=\IR^n$ %\RED (Delete: even for all $q,\ell$)  
\BLACK imply that 
\begin{align}\label{Delta-w-est}
\|\nabla^2 w_j\|_{ L^{q}_{\ell}  (\Omega_j)} \leq C\|(-\Delta) w_j\|_{L^q_\ell(\Omega_j)} \leq C\|\nabla g_j\|_{L^q_\ell(\Omega_j)}. 
\end{align}
%Here we apply Remark \ref{rem-whole-space- case} for $\lambda=1$ and $f=0$. 
%We note that 
%$$
%\|\nabla^2 w_j\|_{L^q_\ell} \leq C\|(-\Delta) w_j\|_{L^q_\ell} $$
%due to the weight condition $2-n/q < \ell < n/q'$. 
%\RED (I do not really understand this argument: there is no $\lambda$ involved in \eqref{w1-eq}, \eqref{w2-eq} for $w_j$ which can go to $0$. But the result is hidden in the proof of \cite[Theorem 4.5]{Farwig-Sohr} in the form $\|u_g\|_{L^q_w} \leq C \|g\|_{\hat H^{-1,q}_w}$ where $u_g=-(-\Delta)^{-1}\nabla g$. Indeed, the unique solution of the generalized resolvent problem with r.h.s. $f,g$ on $\IR^n$ in weighted spaces $L^q_w$ is
%$$u=(\lambda-\Delta)^{-1} \mathbb Pf-(-\Delta)^{-1}\nabla g.$$
%In our case $\lambda=1$ and $w_2=u=u_g$. I propose to write down in Subsect. 2.2 a general estimate for the generalized resolvent problem on $\IR^n$  including $u, \nabla u$ and $\nabla^2 u$ which was missing explicitly in \cite{Farwig-Sohr}.) 
Moreover, for $j=2$, $\eqref{St-res-wRn}_2$ and the Poincar\'e inequality applied to $u\cdot\nabla\eta_2$ imply that 
\begin{align}\label{w-est2}
\begin{aligned}
\|\nabla w_2\|_{ L^q_{\ell} (\Omega_2)} & \leq C \|g_2\|_{L^q_\ell(\Omega_2)}
= C\|u\cdot\nabla \eta_2\|_{L^q_\ell(\Omega_2)}
\leq C \|u\|_{L^q(\Omega_1)}. %\leq C\|\nabla u\|_{L^q(\Omega)}.\\
\end{aligned}
\end{align}
Next, \eqref{w-est0} and Friedrichs' inequality for $\varphi$ as in \eqref{u-nabla-eta-j_phi}  
%and Friedrichs' inequality for $\varphi$, which may be taken with vanishing integral mean on $\Omega_1$ since $\int_{\Omega_1} u\cdot\nabla\eta_1 \dx=0$,\RED)) \BLACK 
imply that
$$
\|w_1 \|_{L^q(\Omega_1)} \leq 
\|g_1\|_{\widehat{H}^{-1,q}_{ 0\BLACK}(\Omega_1)}
\leq C\|u\|_{H^{-1,q}_0(\Omega_1)} \,. 
$$
In addition, we see from \eqref{w-est3} and \eqref{Delta-w-est} that 
$$
\|w_1 \|_{H^{2,q}(\Omega_1)} \leq C\|u\|_{H^{1,q}(\Omega_1)}.
$$
Hence complex interpolation, {\em cf.} \cite[Proof of Lemma 14]{NollSaal}, derives that 
\begin{equation}\label{w_1H^1-u}
    \|w_1\|_{H^{1,q}(\Omega_1)} \leq C\|u\|_{L^q(\Omega_1)}. 
\end{equation}
 
% In addition, the Hardy inequality \eqref{Hardy-ineq} implies that for $1-\frac{n}{q} <\ell< \frac{n}{q'}$
%\begin{align}\label{w-est3}
%\|w_j\|_{ W^{1,q}_{\ell} (\Omega_j)} \leq C \|\nabla u\|_{L^q_{\ell}(\Omega_0)}. 
%\end{align}
%(We will apply this estimate for small $|\lambda|$.)
%
%(Here the second estimate in [Noll-Saal, Theorem 1], is used where $\nabla g$ is not necessary; moreover, $\ell\geq 0$) \BLACK
%(Here we used the duality argument for $\widehat{H}^{-1,q}$ norm. )
%(The same argument as Noll and Saal, pp.24.) 

As for $Q_{2,3}(u,p)$ we recall the {\em BIP} property of the Stokes operator in $L^{q}_{\ell}(\mathbb{R}^n)$, 
 see Proposition \ref{HinftyRn} or \cite[Theorem 1.4 (iii)]{FS} and Corollary \ref{cor:Rn-domains-char}.
Therefore as in \eqref{Qj2} together with   
\eqref{w-est} and  \eqref{w-est2}, \BLACK
 we obtain for $\alpha \in (0,1)$ that
\begin{align} 
\begin{aligned}\label{Q23}
\|Q_{2,3}(u,p)\|_{L^q_{\ell}(\Omega_2)} & = \|A_{q,\ell}^{1-\alpha }(\lambda + A_{q,\ell})^{-1} A_{q,\ell}^\alpha 
\mathbb{P}w_2\|_{L^q_{\ell}(\Omega_2)}\\
& \leq C\frac{1}{|\lambda|^\alpha }\|\mathbb{P}w_2\|_{\widehat{H}^{2\alpha }_{q,\ell}(\mathbb{R}^n)}\\
&\leq C\frac{1}{|\lambda|^\alpha }\|\mathbb{P}w_2\|_{{H}^{1 }_{q,\ell}(\mathbb{R}^n)}\\
&\leq C\frac{1}{|\lambda|^\alpha }\|u\|_{L^q(\Omega)}\\
&\leq C\frac{1}{|\lambda|^{\alpha+1}} \|f\|_{ L^q_\ell (\Omega)}.
\end{aligned}
\end{align} 
For $j=1 $ we argue by analogy, but exploit \eqref{w_1H^1-u} rather than \eqref{w-est2} and obtain that
\begin{equation} \label{Q13}
\|Q_{1,3}(u,p)\|_{L^q_{\ell}(\Omega_1)} \leq C\frac{1}{|\lambda|^{\alpha+1}} \|f\|_{ L^q_\ell (\Omega)}.
\end{equation}

Since $Q_{2,1}(u,p)$ contains the pressure term $p\nabla\eta_2$, we recall the following lemma about local decay estimates of the pressure $p$ of the Stokes resolvent problem \eqref{res-eq}.  

\vspace{2ex}

\begin{lem} {\rm \cite[Lemma 13]{NollSaal}} \label{p-local-decay}
Let $\alpha \in (0, \frac{1}{2q})$. Then it holds for any bounded domain $D\subset \Omega$ of class $C^{1,1}$ that
$$
\|p\|_{L^q(D)} \leq C|\lambda|^{-\alpha}\|f\|_{L^q(\Omega)}
$$
for $\lambda \in \Sigma_\theta $ with $|\lambda| \geq 1$. 
\end{lem}

Hence the term $Q_{j,1}(u,p)$ admits by Proposition \ref{FS-weighted} for $\Omega_2$, classical Stokes resolvent estimates for $\Omega_1$, and Lemma \ref{p-local-decay} for $|\lambda|\geq 1$ the estimate
\begin{align}
\begin{aligned}
\label{Qj1}
\|Q_{j,1}(u,p) \|_{L^q_{\ell}(\Omega_2)} & \leq \frac{C}{|\lambda|} \|f\|_{L^q_\ell(\Omega_0)} \Big(\frac{1}{|\lambda|^{1/2}} + \frac{1}{|\lambda|} + \frac{1}{|\lambda|^\alpha}\Big) \\
& \leq \frac{C}{|\lambda|^{1+\alpha}} \|f\|_{L^q_\ell(\Omega)}.
\end{aligned}
\end{align}

Summarizing \eqref{Qj2}, \eqref{Q23}, \eqref{Q13}, \eqref{Qj1} for $|\lambda|>1$ we can estimate contour integrals of the type $\int_\Gamma h(\lambda)Q_{j,k}(u,p)\dlambda$. 
However, since $Q_{j,k}(u,p) = Q_{j,k}(u,p)(\lambda)$ depends on $(\lambda+A)^{-1} = -((-\lambda)-A)^{-1}$ we use the change of variables $\lambda\mapsto -\lambda$ to get for the contour $\Gamma_{\omega} = \Gamma_{\omega +}\cup \Gamma_{\omega -}$ in the right half plane of $\IC$ the contour $-\Gamma_{\omega} = -\Gamma_{\omega +}\cup -\Gamma_{\omega -}$ in the left half plane. Hence
the term to be estimated in the $\mathcal H^\infty$-calculus is the left hand side of the identity
$$ \int_{\Gamma_\omega} h(\lambda) Q_{jk}(-\lambda) \dlambda = \int_{-\Gamma_\omega} h(-\lambda) Q_{jk}(\lambda) \dlambda, $$
whereas we estimate its right hand side. To that aim we consider the unbounded subset $-\Gamma_{\omega,\infty} = -\Gamma_\omega \setminus B_1(0)$ in \eqref{Atslambda2} of the contour $-\Gamma_\omega $ and obtain that 
\begin{equation}\label{Gomegainftyj=123}
\Big\|\int_{-\Gamma_{\omega,\infty}} h(-\lambda) (Q_{j,1}+Q_{j,2} + Q_{j,3})(u,p)(\lambda) \dlambda\Big\|_{L^q_\ell(\Omega_j)}   \leq C|h|_{\infty,\phi} \|f\|_{L^q_\ell (\Omega)},\quad j=1,2. 
\end{equation}\BLACK

Next we consider $\lambda\in\Sigma_\omega$ with $0<|\lambda|<1$.
For $j=1$ we see from the bounded domain case that 
$$ 
\|(\lambda + A_{\Omega_1, q,\ell})^{-1}\|_{\mathcal{L}(L^q_\ell(\Omega_1))}  \leq 
C\|(\lambda + A_{\Omega_1, q,\ell})^{-1}\|_{\mathcal{L}(L^q(\Omega_1)}
 \leq C\frac{1}{1+|\lambda|}.
$$ 
In addition, it follows from the definition of $f_1'$, Poincar\'e's inequality applied to $u$ and Friedrichs' inequality to $p$ under the admissible assumption $\int_{\Omega_1} p\dx=0$ that
\begin{align}\label{Q11}
\begin{aligned}
\|Q_{1,1}(u,p)\|_{L^q_{\ell}( \Omega_1\BLACK)} & \leq C\|\mathbb Pf_1'\|_{L^q_\ell(\Omega_1)} \leq 
C(\|\nabla u\|_{L^q(\Omega)}+
\|p\|_{L^q(\Omega_1)} )\\
& \leq C|\lambda|^{-1/2} \|f\|_{L^q_{\ell}(\Omega)}. 
\end{aligned}
\end{align}
By analogy, using \eqref{w-est3} for $w_1$ in $Q_{1,2}$, $Q_{1,3}$, there holds 
\begin{align}\label{Q1,23}
\|Q_{1,2}(u,p)\|_{L^q_{\ell}(\Omega_1)} + \|Q_{1,3}(u,p)\|_{L^q_{\ell}(\Omega_1)} \leq C(1+|\lambda|^{-1/2}) \|f\|_{L^q_{\ell}(\Omega)}. 
\end{align}
Summarizing \eqref{Q11}, \eqref{Q1,23} and integrating along the bounded contour  $-\Gamma_{\omega,1} = -\Gamma_\omega \cap B_1(0)$ together with $h(-\lambda)$ we are led to the estimate
\begin{equation}\label{Gomega1j=123}
\Big\| \int_{-\Gamma_{\omega,1}} h(-\lambda) \BLACK (Q_{1,1}+Q_{1,2} + Q_{1,3})(u,p) \dlambda\Big\|_{L^q_\ell(\Omega_1)} \leq C|h|_{\infty,\phi} \|f\|_{L^q_\ell (\Omega)}.
\end{equation}

%These estimates together with \eqref{w-est3}, where we may replace $\|\nabla u\|_{L^q(\Omega_0)}$ by $\|\nabla u\|_{L^q(\Omega_1)}$, imply that 
%\begin{align}
%\begin{aligned}\label{small-lambda-est-eta1}
%\Big\|\frac{1}{2\pi i}\int_{\Gamma_{\omega,1}} h(\lambda)\,u_1' \dlambda \Big\|_{L^q_\ell(\Omega_1)} & = \Big\|\frac{1}{2\pi i}\int_{\Gamma_\omega,1} h(\lambda)(v_1 +w_1) \,{\rm d}\lambda \Big\|_{L^q(\Omega1)}\\
%& \leq C|h|_{\infty,\phi} \int_0^1 \frac{1}{|1+s e^{i \phi}|}\Big(1+\frac{1}{\sqrt{s}%\RED\;Delete:+s^\alpha\BLACK}\Big)\ds \, \|f\|_{L^q(\Omega_0)}  \\
%& \leq C |h|_{\infty,\phi} \|f\|_{L^q_\ell (\Omega_0)}. 
%\end{aligned} 
%\end{align}

For small $\lambda \in\Sigma_\omega$ and $j=2$ we exploit the condition $2-\frac{n}{q}<\ell<\frac{n}{q'}$ and apply Hardy's \eqref{Hardy-ineq} and Rellich's inequality \eqref{Rellich-ineq}. Then we obtain from the cut off property of $\nabla\eta_2$  so that \eqref{w-est} can be modified to the estimate $\|w_2\|_{L^{\tilde{q}}_\ell(\Omega_2)} \leq c\|\nabla^2 u\|_{L^q_\ell(\Omega_2)}$ with $\tilde{q}<q$ close to $q$ the estimates
\begin{align}
\begin{aligned}\label{small-lambda-est1}
\|\nabla u \cdot \nabla \eta_2\|_{L^{\tilde{q}}_\ell(\Omega_2) } +\| u \Delta \eta_2 \|_{L^{\tilde{q}}_\ell (\Omega_2) }  + \| w_2\|_{L^{\tilde{q}}_\ell (\Omega_2)}   \leq C\|\nabla^2 u\|_{L^q_\ell (\Omega)}
\end{aligned} 
\end{align}
and 
\begin{align}
\begin{aligned}\label{small-lambda-est2}
\|p \nabla \eta_2 \|_{L^{\tilde{q}}_\ell(\Omega_2) } \leq 
c\|\nabla p\|_{L^q_\ell (\Omega)}. 
\end{aligned} 
\end{align}
Furthermore, 
we see from Proposition \ref{FS-weighted} and the Sobolev inequality on weighted spaces \eqref{fractSob}, \eqref{pi_k} %\eqref{Hardy-Rellich} 
that for  $0<\alpha  = \frac{n}{2}\big(\frac{1}{\tilde q}-\frac{1}{q}\big) \BLACK<1$ 
and any $\psi \in L^{\tilde{q}}_{\sigma, \ell}(\mathbb{R}^n)$ with $\tilde{q}<q $ chosen close to $q$ such that $2-\frac{n}{\tilde{q}} < \ell < \frac{n}{\tilde{q}'}$ that  
\begin{align}
\begin{aligned}\label{small-lambda-est3}
\|(\lambda + A_{\Omega_2,q})^{-1} \psi\|_{L^q_\ell(\IR^n)}  & \leq \|(-\Delta)^\alpha (\lambda + A_{\Omega_2, \tilde{q}})^{-1}\psi \|_{ L^{\tilde{q}}_\ell(\IR^n )}\BLACK\\
& \leq 
C\|\nabla^2 (\lambda + A_{\Omega_2, \tilde{q}})^{-1}\psi \|_{ L^{\tilde{q}}_\ell(\IR^n )}^{\alpha} \|(\lambda + A_{\Omega_2, \tilde{q}})^{-1}\psi \|_{ L^{\tilde{q}}_\ell(\IR^n )}^{1-\alpha} \\
& \leq C|\lambda|^{\alpha-1} \|\psi\ \|_{L^{\tilde{q}}_{\ell}(\IR^n)} .
\end{aligned} 
\end{align}
Note that in the second step we also used the moment inequality. 
Substituting \eqref{small-lambda-est1}, \eqref{small-lambda-est2} and \eqref{small-lambda-est3}  %to the estimate 
%$$
%\Big\|\frac{1}{2\pi i}\int_{\Gamma_{\omega,1}} h(\lambda)\RED (\eta_2 u-u_2)\BLACK \dlambda \Big\|_{L^q_\ell(\Omega_2)} = \Big\|\frac{1}{2\pi i}\int_{\Gamma_{\omega,1}} h(\lambda)(v_2+w_2) \,{\rm d}\lambda \Big\|_{L^q_\ell(\Omega_2)}
%$$
and applying Theorem \ref{res-weighted} to $u$ and $\nabla p$ we obtain that %\RED ($+1$ from $\nabla p$) 
\BLACK
\begin{align}
\begin{aligned}\label{small-lambda-est-eta2}
\Big\| \int_{-\Gamma_{\omega,1}} h(-\lambda)\BLACK  (Q_{2,1}+Q_{2,2}+Q_{2,3}) \dlambda \Big\|_{L^q_\ell(\Omega_2)} 
& \leq C |h|_{\infty,\phi} \int_0^1  \Big(1+\BLACK \frac{1}{s^{1-\alpha }}\Big) \ds \, \|f\|_{L^q(\Omega)}  \\
& \leq C |h|_{\infty,\phi} \|f\|_{L^q_\ell (\Omega)}.
\end{aligned} 
\end{align}
%
%Therefore we get that 
%\begin{align}
%\begin{aligned}\label{small-lambda-est-result}
%sum_{j=1}^2 \Big\|\frac{1}{2\pi i}\int_{\Gamma_{\omega,1}} h(\lambda)(v_j + w_j) \dlambda \Big\|_{L^q_\ell(\Omega_0)} 
%&\leq C \sum_{j=1}^2 \|\frac{1}{2\pi i}\int_{\Gamma_{\omega,1}} h(\lambda)
%(\lambda+A_{\Omega_j,q})^{-1} f \,{\rm d}\lambda \|_{L^q_\ell(\Omega_j)} \leq C|h|_{\infty,\phi} \|f\|_{L^q_\ell(\Omega_0)}.
% \end{aligned} 
%\end{align}

Finally, for $Q_{2,4}(u,p)$,
recall that $w_2, g_2$ depend on $\lambda$ via \eqref{w2-eq} and 
$$ g_2=u\cdot \nabla\eta_2 = \big(( \lambda+A_{\Omega})^{-1}f\big) \cdot\nabla\eta_2  = -\big(( (-\lambda)-A_{\Omega}\BLACK)^{-1}f\big) \cdot\nabla\eta_2.$$ 
Hence the decay in $\lambda$ is not faster than $|\lambda|^{-1}$ and fails to yield integrability along $\Gamma_{\omega,\infty}$. Therefore, we use directly the $\mathscr H^\infty$-calculus of $A_{\Omega_2,q}$.  
We write  $L_2 (g_2):= w_2$, {\em i.e.}, $L_2 = -(-\Delta)^{-1}\nabla$ which maps $\widehat H^{-1,q}_\ell(\Omega_2)$ to $L^q_\ell(\Omega_2)$. Moreover, we recall  \eqref{u-nabla-eta-j_phi} \BLACK with $\psi\in L^q(\Omega)$ replacing $u$ \BLACK which % Do we use this: 
by \eqref{Hardy-ineq} \BLACK yields the estimate
$$ \|\psi \nabla\eta_2\|_{\hat H^{-1,q}_\ell(\Omega_2)} \leq C \|\psi\|_{L^q(\Omega)}, \quad \psi \in L^q(\Omega). $$
Then due to the $\mathscr H^\infty$-calculus of  $A_{\Omega,q}$ without weight, see \cite[Theorem 17]{NollSaal},
\begin{align}
%\begin{aligned}
\label{estimate-forth-term}
\Big\|\frac{1}{2\pi i} \int_{-\Gamma_\omega} \!\!h(-\lambda)w_2 \dlambda \Big\|_{L^q_\ell(\Omega_2)}
& = \Big\|L_2 \Big(\frac{1}{2\pi i} \int_{-\Gamma_\omega} \!\!h(-\lambda) (u\cdot\nabla\eta_2) \dlambda \Big)\Big\|_{L^q_\ell(\Omega_2)}\nonumber\\[1ex]
& \leq C \Big\|\frac{1}{2\pi i}  \int_{-\Gamma_\omega} \!\!h(-\lambda) (\lambda-A_{\Omega,q})^{-1} \BLACK f \dlambda \cdot\nabla\eta_2 \Big\|_{\hat H^{-1,q}_\ell(\Omega_2)} \nonumber\\[1ex]
& \leq C
\| h(A_{\Omega,q}) f \|_{L^q(\Omega)}\nonumber\\
& \leq C |h|_{\infty,\phi} \|f\|_{L^q_\ell (\Omega)}.
%\\[1ex]
%\RED Delete:&\;\leq C \|\nabla \eta_2 \cdot h(A_{\Omega_2}) f\|_{\hat H^{-1,q}_{\ell}(\Omega_2)}\\
%\RED Delete:&\;\leq C \|\nabla \eta_2 \cdot h(A_{\Omega_2}) f\|_{\hat H^{-1,q}(\Omega_2)}
%\\
%& \leq C |h|_{\infty,\phi} \|f\|_{L^q(\Omega_0)}\\
%& \leq C |h|_{\infty,\phi} \|f\|_{L^q_\ell(\Omega_0)}.
%\end{aligned} 
\end{align}
 Note that in the last step of \eqref{estimate-forth-term}
the assumption $\ell\geq 0$ is crucial. \BLACK
%Here in the third line we used the cut off property of $\supp \nabla\eta_2 \subset \{R \leq |x| \leq 2R\}$ and in the last line we used the same argument as that in \cite[pp.678-679]{NollSaal} based on the duality argument without weight functions. 
%\RED (Is this still needed here?: We decompose the contour %$\Gamma_\omega$ to 
%$\Gamma_{\omega}=\Gamma_{\omega,1} + \Gamma_{\omega,\infty}$, where 
%$\Gamma_{\omega,1 }: \lambda=\rho e^{\pm i\omega}, \, 0 \leq \rho <1$ and $\Gamma_{\omega,\infty }: \lambda=\rho e^{\pm i\omega}, \, 1 \leq \rho < \infty$. Then \eqref{w-est2} and \eqref{estimate-forth-term} derive that 
%\begin{align}
%\begin{aligned}\label{Q24}
%\|\frac{1}{2\pi i}\int_{\Gamma_{\omega,\infty}} h(\lambda)
%Q_{2,4}(u,p) \,{\rm d}\lambda \|_{L^q_\ell(\Omega_2)} & \leq \|\frac{1}{2\pi i}\int_{\Gamma_\omega} h(\lambda)
%L_2 (g_j) \,{\rm d}\lambda \|_{L^q_\ell(\Omega_2)}\\
%&\quad + \|\frac{1}{2\pi i}\int_{\Gamma_{\omega,1}} h(\lambda)
%L_2 (g_j) \,{\rm d}\lambda \|_{L^q_\ell(\Omega_2)}\\
%& \leq C |h|_{\infty,\phi} \|f\|_{L^q(\Omega_0)} + C\int_{0}^1
%|\lambda|^{-1/2}\,{\rm d}\lambda   |h|_{\infty,\phi}\|f\|_{L^q(\Omega_0)}\\
%& \leq C |h|_{\infty,\phi} \|f\|_{L^q(\Omega_0)}. \;)?
%\end{aligned} 
%\end{align}

%The same estimates are obtained for $j=1$, {\em i.e.} for the bounded domain $\Omega_1 = B_{2R}(0)$. \BLUE 

On the bounded domain $\Omega_1$, the weight $\langle x\rangle^\ell$ has no influence, and there holds
$$\D(A_{q,\Omega_1}^\theta) = [L^q_{\sigma,\Omega_1}, \D(A_{q,\Omega_1})]_{\theta}  = {H}^{2\theta}_{q}(\Omega_1)\cap L^q_{\sigma}(\Omega_1)$$
with nonhomogeneous spaces, see \cite[Theorems 2 and 3]{Giga85}, instead of  \eqref{BIP-homogeneous-domain-Rn}. 
Moreover, due to \cite[Theorem 3]{NollSaal}, the Stokes operator $A_{\Omega_1}$ has a bounded $\mathscr H^\infty$-calculus. 
Hence, by analogy to \eqref{estimate-forth-term}, using a bounded operator $L_1:\hat H^{-1,q}_\ell(\Omega_1)$ to $L^q_\ell(\Omega_1)$, $L_1(g_1):=w_1$, we obtain that
\begin{align}
\begin{aligned}\label{estimate-forth-term1}
\Big\|\int_{-\Gamma_\omega} h(-\lambda) \BLACK w_1 \dlambda \Big\|_{L^q_\ell(\Omega_1)}
%& = \Big\|L_1 \Big(\frac{1}{2\pi i}  \int_{\Gamma_\omega} h(\lambda) (u\cdot\nabla\eta_1) \dlambda \Big)\Big\|_{L^q_\ell(\Omega_1)}\\[1ex]
%& \leq C \Big\|\frac{1}{2\pi i} \int_{\Gamma_\omega} h(\lambda) (\lambda-A_{\Omega_0})^{-1} f \dlambda \cdot\nabla\eta_1 \Big\|_{\hat H^{-1,q}_\ell(1)} \\[1ex]
& \leq C |h|_{\infty,\phi} \|f\|_{L^q_\ell (\Omega)}. 
\end{aligned} 
\end{align}
Replacing $w_j$ by $(I-\mathbb P_{\Omega_j})w_j$, \eqref{estimate-forth-term}, \eqref{estimate-forth-term1} yield the result
\begin{equation}\label{h_Q_{j,4}}
\Big\| \int_{-\Gamma_\omega} h(-\lambda)\BLACK Q_{j,4} (u,p) \dlambda \Big\|_{L^q_\ell(\Omega_j)} \leq C |h|_{\infty,\phi} \|f\|_{L^q_\ell (\Omega)},\quad j=1,2.
\end{equation} 

Now the proof of Theorem \ref{theorem-H-infty} is complete.
\end{proof}

\vspace{2ex}

A duality argument and complex interpolation yield the $\mathscr H^\infty$-calculus of $A_{q,\ell}$ for all admissible pairs $q,\ell$, {\em i..e.} for $1<q< \infty$ and $-\frac{n}{q} < \ell < \frac{n}{q'}$.  

\vspace{1ex}

\begin{cor}\label{theorem-H-infty-2}
Let $n\geq3 $, $1<q< \infty$, and  $-\frac{n}{q} < \ell < \frac{n}{q'}$.
Then the Stokes operator $A=A_{q,\ell}$ possesses on $L^q_{\sigma, \ell }(\Omega)$ a bounded $\mathscr H^\infty$-calculus with $\mathscr H^\infty$-angle $\Phi_ {A}^\infty=0$.
\end{cor}

\vspace{1ex}

\begin{proof}
For $h\in\mathcal H_0(\Sigma_\phi)$, $0<\phi<\frac{\pi}{2}$, $\varphi \in C^{\infty}_{0,\sigma}(\Omega)$ and $f \in L^q_{\ell}(\Omega)$ it holds that  
\begin{align*}
    \begin{aligned}
    \langle h(A)f, \varphi\rangle & =
\Big\langle \frac{1}{2\pi i}\int_{\Gamma_\omega} h(\lambda)
(\lambda  -A\BLACK)^{-1} f  \,{\rm d}\lambda, \varphi\Big\rangle\!\!\!  & = \Big\langle f, \frac{1}{2\pi i}\int_{\Gamma_\omega} h(\lambda)
(\lambda   -A^*\BLACK)^{-1}\varphi\Big\rangle  \\
& = \langle f, h(A^*)\varphi\rangle ,
\end{aligned}
\end{align*}
where $A^*=(A_{q,\ell})^* = A_{q',-\ell}$ on $L^{q'}_{\sigma, -\ell }(\Omega)$ with $2-\frac{n}{q'} < -\ell < \frac{n}{q}$. 
Hence Theorem \ref{theorem-H-infty} implies that 
$$
|\langle h(A)f, \varphi\rangle | \leq C |h|_{\infty, \phi}\|f\|_{L^q_{\ell}(\Omega)}\|\varphi\|_{L^{q'}_{-\ell}(\Omega)}. 
$$ 
  
This estimate together with a duality argument completes the proof for pairs $(q,\ell)$ satisfying $-\frac{n}{q} < \ell < \frac{n}{q'}-2$, $\ell\leq 0$ with $\mathscr H^\infty$-angle $\Phi_ {A}^\infty=0$. 

In a second step we use complex interpolation. 
Consider any $q\in(1,\infty)$ and admissible $\ell$, say $\ell>0$. Set $q_0=q_1=q$ and find suitable
$\ell_0< \min(0,-2+n/q')$  and  $\ell_1> \max(0,2-n/q)$  such that $\ell_0 < \ell < \ell_1.$
Then there exist $\theta\in (0,1)$ such that 
$\ell = (1-\theta)\ell_0 + \theta \ell_1.$
Now complex interpolation in weighted spaces, see \cite[Theorem 1.18.5]{Triebel}, implies that  for any $h\in\mathcal H_0(\Sigma_\phi)$
$$ \|h(A)\|_{\mathcal L(L^q_{\sigma,\ell}(\Omega))} \leq C_\theta |h|_{\infty, \phi}.$$
Thus $A$ possesses also on $L^q_{\sigma,\ell}(\Omega)$ a bounded $\mathscr H^\infty$-calculus with $\Phi_{A}^\infty=0$.
 For more details on complex interpolation applied to operators $h(A)$ in an $\mathscr H^\infty$-calculus we refer to \cite[Proposition 4.9]{KaKuWe06}. 
\end{proof}

\setlength{\unitlength}{0.9mm}
%\makebox{
\hspace*{3mm}
\vspace*{5mm}
\begin{picture}(90,120) 
\vector(0,1){130}
\put(0,60){\vector(1,0){140}}
\put(0,120){\line(2,-1){140}}
\put(0,100){\line(2,-1){140}}
\put(0,80){\line(2,-1){140}}
\put(0,60){\line(2,-1){140}}
\put(-4,59){$0$}
\put(-4,79){$1$}
\put(-4,99){$2$}
\put(-4,119){$3$}
\put(-7,39){$-1$}
\put(-7,19){$-2$}
\put(-7,-1){$-3$}
\put(-4,128){$\ell$}

\put(119,55){$3$}
\put(141,59){$\frac{3}{q}$}
\put(143,49){$\frac{3}{q'}$}
\put(143,29){$2-\frac{3}{q}$}
\put(143,9){$-2+\frac{3}{q'}$}
\put(143,-11){$-\frac{3}{q}$}

\put(-1,120){\line(1,0){2}}
\put(-1,100){\line(1,0){2}}
\put(-1,80){\line(1,0){2}}
\put(-1,60){\line(1,0){2}}
\put(-1,40){\line(1,0){2}}
\put(-1,20){\line(1,0){2}}
\put(-1,0){\line(1,0){2}}
\put(120,59){\line(0,1){2}}

\put(30,90){$\ell\geq 0$}
\put(80,67){$\ell\geq 0$}
\put(30,50){$\ell\leq 0$}
\put(80,27){$\ell\leq 0$}
%\put(10,0){\multiput(0,0)(2.5,5){10}{\line(1,2){2}}}
%\put(30,0){\multiput(0,0)(0,5){10}{\line(0,1){2}}}
% For interpolation
\put(15,56){$\bullet$} 
\put(15,101){$\bullet$} 
\put(16,56){\line(0,1){45}}
\put(15,83){$\ast$} 

%%%%%%%%%%% Dashed set \ell\geq 0
\put(81,61){\multiput(0,0)(4,0){9}{\line(1,0){3}}}
\put(2,102){\multiput(0,0)(0,4){4}{\line(0,1){3}}}
\put(1,101){\multiput(0,0)(6,-3){14}{\line(2,-1){4}}}
\put(1,118){\multiput(0,0)(6,-3){19}{\line(2,-1){4}}}

%%%%%%%%%%% Dashed set \ell\leq 0
\put(3,59){\multiput(0,0)(4,0){9}{\line(1,0){3}}}
\put(2,59){\multiput(0,0)(4,-2){15}{\line(2,-1){2}}}
\put(121,2){\multiput(0,0)(0,4){4}{\line(0,1){3}}}
\put(39,59){\multiput(0,0)(6,-3){14}{\line(2,-1){4}}}
\put(4,59){\multiput(0,0)(6,-3){20}{\line(2,-1){4}}}
 
\end{picture}

\vspace*{7mm}
%\hspace*{10mm} 
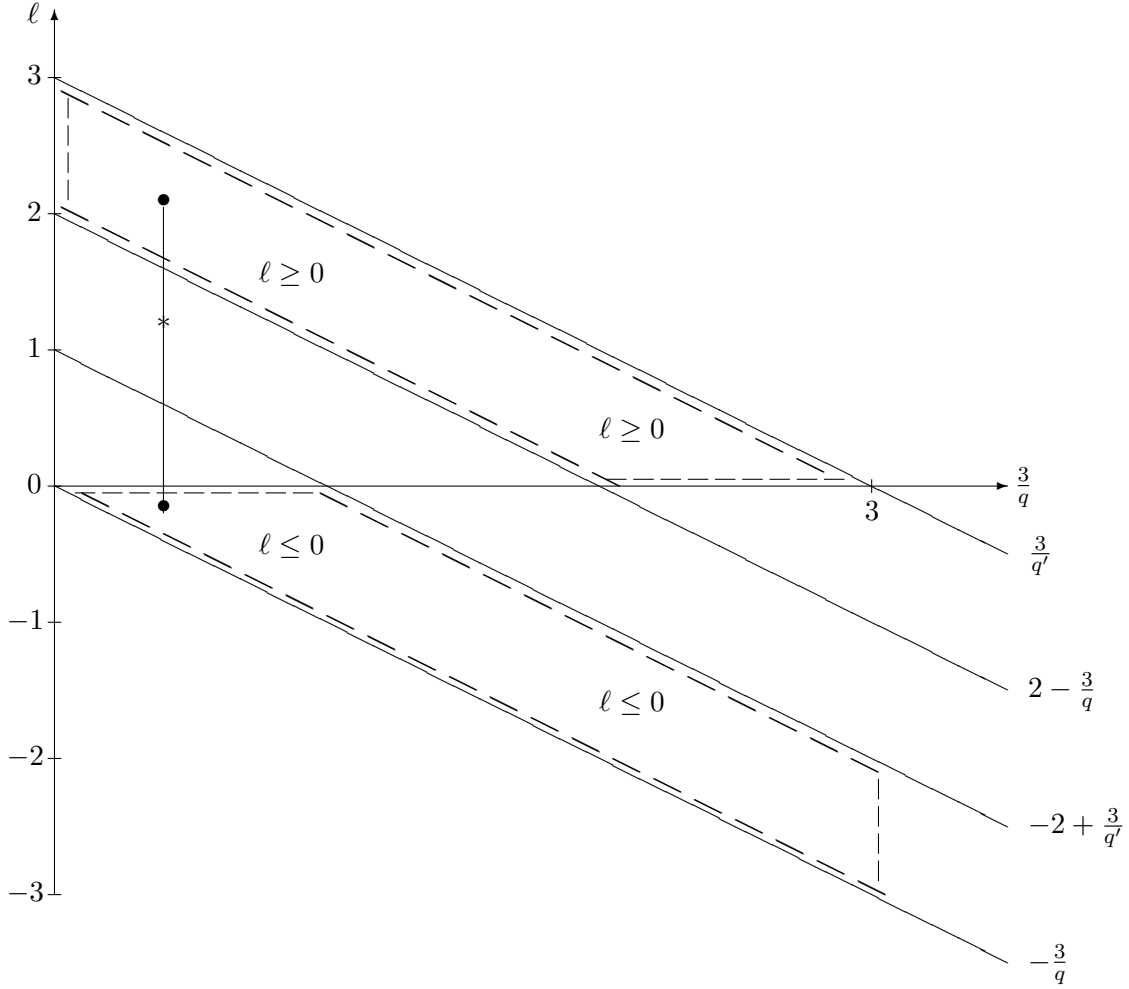
\begin{minipage}[t][25mm][c]{0.9\textwidth}{\bf Figure 1: Case $n=3$.} The areas bounded by the dashed polygonal lines define pairs $\frac{3}{q},\ell$ for which Theorem \ref{theorem-H-infty} and the adjoint setting apply. Interpolation of both results is indicated by the vertical line with bullets as endpoints yielding a result at the asterisk. For $n=4$ the middle strip collapses to the line $2-\frac{4}{q} = -2+\frac{4}{q'}$. For dimensions $n\geq 5$, the two areas with dashed boundaries overlap.
\end{minipage}

\vspace{1ex}

\end{section}

\subsection{The square root of $A$ on exterior domains with weights}\label{S3.3}

In this subsection let $\Omega \subset\IR^n$ be an exterior domain with boundary of class of $C^2$. We start with identities of weighted homogeneous function spaces. The crucial result in Proposition \ref{HhatOmega} is part (ii); however, the range of exponents $q,\ell$ is too restricted for later applications. We follow arguments of Borchers and Miyakwa (\cite{BoMi90, BoMi95}) who obtained similar results for $L^q$ spaces without weights.

\begin{prop}\label{HhatOmega}
(i) For $1<q<\infty $ and  $-\frac{n}{q}<\ell<\frac{n}{q'}$ there holds 
\begin{equation}\label{Hsigma-Hdiv}
 H^{1,q}_{\sigma,\ell,0}(\Omega) = \big\{u\in H^{1,q}_{\ell,0}(\Omega): \div u=0\big\}.
\end{equation}

(ii) If $2-\frac{n}{q}<\ell<\frac{n}{q'}$, then
$$ \big[L^q_{\sigma,\ell}(\Omega), \widehat\D(A_{q,\ell})\big]_{1/2} = \widehat H^{1,q}_{\sigma,\ell,0}(\Omega).$$

(iii) Let $1<q_0<q_1<\infty$, $1-\frac{n}{q_1} < \ell <\frac{n}{q_0'}$ and $\frac{1}{q} = \frac{1-\theta}{q_0} + \frac{\theta}{q_1} $ with $0<\theta<1$. Then
$$ \big[\widehat H^{1,q_0}_{\sigma,\ell,0}(\Omega), \widehat H^{1,q_1}_{\sigma,\ell,0}(\Omega)\big]_\theta = \widehat H^{1,q}_{\sigma,\ell,0}(\Omega).$$
\end{prop}

\begin{proof} (i) First we recall that, by definition, $H^{
1,q}_{\ell,0}(\Omega) = \overline{C^\infty_0(\Omega)}^{\|\nabla \cdot\|_{L^q_\ell(\Omega)}}$ and similarly $X:=H^{1,q}_{\sigma,\ell,0}(\Omega) = \overline{C^\infty_{0,\sigma}(\Omega)}^{\|\nabla \cdot\|_{L^q_\ell(\Omega)}}.$ 
Since $X$ is a closed subspace of $Y:= \big\{v\in \widehat H^{
1,q}_{\ell,0}(\Omega): \div v=0\big\}$, it suffices to prove that 
$X$ is dense in $Y$. 
To this aim, let $f\in H^{-1,q'}_{-\ell}(\Omega) := \big(H^{1,q}_{\ell,0}(\Omega)\big)^*$ satisfy $f=0$ on $X$. In particular, $\langle f,\phi\rangle =0$ for all $\phi\in C^\infty_{0,\sigma}(\Omega)$.
By \cite[Lemma II.2.2.1]{Sohr} there exists $p\in L^{q'}_{\rm loc} (\overline\Omega)$ such that $f=\nabla p$ in the sense of distributions. Actually, the authors only state that $p\in L^{q'}_{\rm loc}(\Omega)$; however, arguments from the proof of \cite[Lemma II.2.2.2]{Sohr} on bounded Lipschitz domains show that indeed $p\in L^{q'}_{\rm loc}(\overline\Omega)$.

By Proposition \ref{lem-var-ineq-nabla_Omega0} (ii) there holds even $p\in L^{q'}_{-\ell}(\Omega)$ - up to an additive constant which is set equal to $0$. Then for any $v\in Y$ there exists a sequence $(v_j)\in C^\infty_0(\Omega)$ such that $\|\nabla (v_j-v)\|_{L^q_\ell} \to 0$ as $j\to\infty$. Moreover, $\div v_j\to 0$ in $L^q_\ell$  since $\div v=0$. Thus
$$ \langle f,v \rangle = \lim_j\, \langle f,v_j\rangle  = \lim_j\, \langle \nabla p,v_j \rangle = - \lim_j\, \langle p,\div v_j \rangle =0,$$
and we conclude that $f=0$ on Y. Consequently, by Hahn-Banach's theorem, $X=Y$.

(ii) Let $E$ and $E_c$ denote extension operators such that for $1<q<\infty$, $2-\frac{n}{q}<\ell<\frac{n}{q'}$, and for $k=0,1,2$,
\begin{align}\label{ext E} 
E: & \,\widehat H^{k,q}_\ell(\Omega)\, \to \widehat H^{k,q}_\ell(\IR^n),\quad
\|Eu\|_{\widehat H^{k,q}_\ell(\IR^n)}\, \leq C\|u\|_{\widehat H^{k,q}_\ell(\Omega)},
\\
\label{ext E_c} 
E_c: & \, H^{k,q}_\ell(\Omega^c) \to H^{k,q}_\ell(\IR^n),\quad
\|E_cu\|_{H^{k,q}_\ell(\IR^n)} \leq C\|u\|_{H^{k,q}_\ell(\Omega^c)},
\end{align}
satisfying $\supp E_cu \subset B=B_R(0)$ and $ E_cu \in  H^{k,q}_{0,\ell}(B)$. \BLACK
Concerning \eqref{ext E} we find  $E$ by Lemma \ref{extension} on the exterior domain $\Omega$. 
%\BLUE (I agree. This together with the resolvent estimate on weighted spaces, we  do not need the condition $q< n/2$ in contrast to Borchers and Miyakawa.)\BLACK  
%
For \eqref{ext E_c} we use a standard extension operator $E_c$ for $\Omega^c$. %((no homogeneous spaces! But see estimate below)). 
%from Lemma \ref{extension} \BLUE (probably not allowed) \RED and a cut off function $\eta\in C^\infty_0(B)$ to define $E_c =\eta E$. Then for $\nabla^2 (E_cu)$ we get {\em e.g.} the perturbation term $(\nabla^2\eta)Eu$ which is estimated by Rellich's inequality \eqref{Rellich-ineq} as follows: 
%\begin{align*}\|(\nabla^2\eta)Eu\|_{L^q_\ell(\IR^n)} & \leq c\|Eu\|_{L^q_\ell(B)} \leq c\|Eu\|_{L^q_{-2+\ell}(B)} \leq c    \|Eu\|_{L^q_{-2+\ell}(\IR^n)}\\    & \leq c\|\nabla^2 Eu\|_{L^q_{\ell}(\IR^n)} \leq c\|\nabla^2 u\|_{L^q_{\ell}(\Omega^c)}\end{align*} 
%$E_cu$ and $\nabla E_cu$ vanish on $\partial B$. Then Poincar\'e's inequality applied once or twice yields the estimate \eqref{ext E_c} for $k=1$ and $k=2$, respectively. 
%((If we use a standard extension operator $E$ - instead of $E$ from Lemma 2.3 - we do not get the last estimate above.))
%\BLUE(In Theorem 1.4 by Chua, maybe we can take $k=2$ only on the whole space extension. 
%\RED DELETE the following sentence since it concerns only the full norm $H^{k,q}_\ell$: 
%"Here we note that the weight $\langle\cdot\rangle^{\ell q}$ does not play any role on $B$. By analogy, further perturbation term are estimated with the help of Hardy's inequality \eqref{Hardy-ineq}".
%\BLUE We use the weighted Rellich inequality for $E_c$ here by the range $2-n/q<\ell < n/q'$ and get the homogeneous estimate. On the first order derivative, we can use the weighted Hardy inequality for $1-n/q<\ell < n/q'$. In this sense, I think that the weighted approach is crucial.)
\BLACK 

Restricting to solenoidal vector fields we recall from \eqref{equ:nabla2-A} that $\|\nabla^2 u\|_{L^{q}_\ell(\Omega)} \leq c\|A u\|_{L^{q}_\ell(\Omega)}$ for $u\in \widehat \D(A_{q,\ell})$. Hence $E$ is bounded from $L^q_\ell(\Omega)$ to $L^q_\ell(\IR^n)$ and from $\widehat \D(A_{q,\ell})$ to $\widehat H^{2,q}_\ell(\IR^n)$. Now complex interpolation and \eqref{BIP-homogeneous-domain-Rn}  imply that 
$$ E: \D(A_{q,\ell}^{1/2}) \to \widehat H^{1,q}_\ell(\IR^n) $$
is bounded so that 
\begin{equation}\label{nabla by A12} 
\|\nabla u\|_{L^q_\ell(\Omega)} \leq \|\nabla Eu\|_{L^q_\ell(\IR^n)} \leq C \|A_{q,\ell}^{1/2} u\|_{L^q_\ell(\Omega)}. 
\end{equation}

To prove the converse to \eqref{nabla by A12} we will construct a restriction operator $\mathscr R$ preserving solenoidality such that 
\begin{align}\label{A-mathscr R} 
\begin{aligned}
\mathscr R &: L^q_{\sigma,\ell}(\IR^n)\, \to L^q_{\sigma,\ell}(\Omega),\\
\mathscr R  &: \widehat H^{2,q}_{\sigma,\ell}(\IR^n) \to \widehat \D(A_{q,\ell})
\end{aligned}
\end{align}
are bounded.  

For the construction of $\mathscr R$ we use for any domain $D\subset \IR^n$ the trivial restriction, $r_D$, of functions on $\IR^n$ to $D$, and let $\mathbb B$ denote a Bogovski\u{\i} operator on the domain $\Omega\cap B$. 
Then let
\begin{equation}\label{def-mathscr R}  
\mathscr R u =  r_{\Omega} u + (\mathbb B\,\div -I)r_{\Omega\cap B}\, E_c r_{\Omega^c}u, \quad u\in C^\infty_{0,\sigma}(\IR^n).
\end{equation}
Note that the necessary condition to apply $\mathbb B$, namely
$$ \int_{\Omega\cap B} \div( r_{\Omega\cap B} \,E_c r_{\Omega^c}u) \dx = \int_{\partial(\Omega\cap B)} (E_c r_{\Omega^c}u)\cdot \textsl{n} \dsigma =0 $$ 
is satisfied since $\div u=0$ on $\Omega^c$. 
 Hence $ \div \mathscr Ru \big|_{\Omega} = \div u\big|_{\Omega}= 0$.
Moreover, on $\partial(\Omega\cap B)$, we see that $\div( r_{\Omega\cap B}\, E_c r_{\Omega^c}u) = \div( -u+ r_{\Omega\cap B}\, E_c r_{\Omega^c}u)$ vanishes  since $r_{\Omega\cap B}\, E_c r_{\Omega^c}u$ and its first normal derivative $\partial_{\textsl{n}}$ 
%\RED (( correct for $\partial_n$??))\BLACK
vanish on $\partial B$; moreover, $-u+ r_{\Omega\cap B} \,E_c r_{\Omega^c}u$ vanishes on $\partial \Omega$,  including $\partial_{\textsl{n}}$.   
Thus  $\mathbb B\,\div r_{\Omega\cap B}\, E_c r_{\Omega^c}u \in H^{2,q}_{0}(\Omega\cap B)$,  and $\mathscr R u$ as well as $\partial_{\textsl{n}} \mathscr R$ vanish on $\partial\Omega$.  

To estimate second order derivatives of $\mathscr Ru$ in $\eqref{A-mathscr R}_2$ we see that 
\begin{align*}
\big\|\nabla^2 \mathbb B\,\div (r_{\Omega\cap B}\, E_c r_{\Omega^c} u)\big\|_{L^q(\Omega\cap B)} 
& \leq C \Big(\big\|\nabla^2 E_c r_{\Omega^c} u\big\|_{L^q(B)} + \big\|\nabla E_c r_{\Omega^c} u \big\|_{L^{q}(\Omega\cap B)}\Big)\\
&  \leq C\Big(\|\nabla^2 u\big\|_{L^q(\Omega^c)} + \|\nabla u\big\|_{L^q(\Omega^c)} + \|u\big\|_{L^q(\Omega^c)}\Big) \\
&  \leq C\Big(\|\nabla^2 u\big\|_{L^q_\ell(\IR^n)} + \|\nabla u\big\|_{L^{q}_{-1+\ell}(\IR^n)} + \|u\big\|_{L^{q}_{-2+\ell}(\IR^n)}\Big) \\
%& \leq DELETE: C\big\|E_c r_{\Omega^c} u\big\|_{H^{2,q}(\Omega\cap B)}\\
& \leq C\|u\|_{\widehat H^{2,q}_\ell(\IR^n)}
\end{align*}
where we used Hardy's inequality \eqref{Hardy-ineq} and Rellich's inequality \eqref{Rellich-ineq} on the whole space $\IR^n$. 
Recall that the weight $\langle\cdot\rangle^{\ell q}$ does not play any role on $\Omega^c$. 
 Hence $\|\nabla^2 \mathscr R u\|_{L^q_\ell(\Omega)} \leq C\|u\|_{\widehat H^{2,q}_\ell(\IR^n)}$. \BLACK
The estimate of first order derivatives of $\mathscr Ru$ in $\eqref{A-mathscr R}_2$ is similar, but exploits only \eqref{Hardy-ineq}. \BLACK

Furthermore, note that Bogovski\u{\i}'s operator is also well defined and bounded as an operator  $\mathbb B: \widehat H^{-1,q}_0(\Omega\cap B) := (\widehat H^{1,q}(\Omega\cap B))^* \to L^q(\Omega\cap B)$, see \cite[Theorem III.3.5]{Galdi-steady}.  
Thus
\begin{align*}
\big\|\mathbb B\,\div (r_{\Omega\cap B} \,E_c r_{\Omega^c} u)\big\|_{L^q(\Omega)} & =  \big\|\mathbb B\,\div (r_{\Omega\cap B} \,E_c r_{\Omega^c}u) \big\|_{L^q(\Omega\cap B)} \\
& \leq C\big\|\div (r_{\Omega\cap B}\, E_c r_{\Omega^c} u)\big\|_{H^{-1,q}_0(\Omega\cap B)}\\
& \leq C\big\|E_c r_{\Omega^c} u\big\|_{L^q(B)} \leq C\|u\|_{L^q_\ell(\IR^n)}. 
\end{align*}
These estimates yield $\mathscr R: L^q_{\sigma,\ell}(\IR^n) \to L^q_{\sigma,\ell}(\Omega)$, {\em i.e.} $\eqref{A-mathscr R}_1$,  as well as
$\eqref{A-mathscr R}_2$ since by Lemma \ref{Helmh} $\|A u\|_{L^{q}_\ell(\Omega)} \leq c \|\nabla^2 u\|_{L^{q}_\ell(\Omega)}$.
Hence by \eqref{interp hatH sigma}, \eqref{BIP-homogeneous-domain-Rn} and complex interpolation $\mathscr R$ is also a bounded operator $\mathscr R:\widehat H^{1,q}_{\sigma,\ell}(\IR^n) \to \widehat \D(A_{q,\ell}^{1/2}).$ 
Since obviously $\mathscr R E_0=I$ on $\widehat H^{1,q}_{ \sigma,\ell,0\BLACK}(\Omega)$ where $E_0$ denotes the trivial extension by $0$ from $\Omega$ to $\IR^n$, we obtain that 
\begin{equation}\label{A12 by nabla}  
\big\|A_{q,\ell}^{1/2} u\big\|_{L^q_\ell(\Omega)} = \big\|A_{q,\ell}^{1/2} \mathscr R E_0 u\big\|_{L^q_\ell(\Omega)} \leq C\|E_0u\|_{\widehat H^{1,q}_{\sigma,\ell}(\IR^n)} = C\|\nabla u\|_{L^q_\ell(\Omega)}.  
\end{equation}
Combining \eqref{nabla by A12} and \eqref{A12 by nabla}
we proved (ii).

(iii) By definition, $ \mathscr R$ is bounded from $\widehat H^{1,q}_{\sigma,\ell}(\IR^n)$ to $\widehat H^{1,q}_{\sigma,\ell,0}(\Omega)$ provided that $u\in \widehat H^{1,q}_{\sigma,\ell}(\IR^n)$ can be identified with a function. Thus assuming $1-\frac{n}{q}<\ell<\frac{n}{q'}$, also $\mathscr R \mathbb P$ is bounded from $\widehat H^{1,q}_{\ell}(\IR^n)$ to $\widehat H^{1,q}_{\sigma,\ell,0}(\Omega)$. 
Then by Lemma \ref{interp-hatH-Rn in q},  for $0\leq\theta \leq 1$, $1<q_0<q_1< \infty\BLACK$ and $\frac{1}{q} = \frac{1-\theta}{q_0} + \frac{\theta}{q_1}$, such that  $1-\frac{n}{q_0} < 1-\frac{n}{q} <1-\frac{n}{q_1} <\ell$ and $\ell < \frac{n}{q_0'} < \frac{n}{q} <\frac{n}{q_1'}$, \BLACK
$$ \mathscr R \mathbb P: \big[\widehat H^{1,q_0}_{\ell}(\IR^n),\widehat H^{1,q_1}_{\ell}(\IR^n)\big]_\theta = \widehat H^{1,q}_{\ell}(\IR^n) \to \big[\widehat H^{1,q_0}_{\sigma,\ell,0}(\Omega), \widehat H^{1,q_1}_{\sigma,\ell,0}(\Omega)\big]_\theta $$ 
is bounded as well. 
Thus for any $u\in \widehat H^{1,q}_{\sigma,\ell,0}(\Omega)$ we get that 
\begin{align}\label{interp-ZP}
\|u\|_{\big[\widehat H^{1,q_0}_{\sigma,\ell,0}(\Omega),\, \widehat H^{1,q_1}_{\sigma,\ell,0}(\Omega)\big]_\theta} & = \|\mathscr R \mathbb P E_0 u\|_{\big[\widehat H^{1,q_0}_{\sigma,\ell,0}(\Omega),\, \widehat H^{1,q_1}_{\sigma,\ell,0}(\Omega)\big]_\theta} \nonumber\\
& \leq C\|\nabla E_0 u\|_{L^q_\ell(\IR^n)} = C\|\nabla u\|_{L^q_\ell(\Omega)}.
\end{align}

Conversely, we interpolate $\nabla: \widehat H^{1,q}_{\sigma,\ell,0}(\Omega) \to L^q_\ell(\Omega)$ with $q=q_0$, $q=q_1$ and see that 
$$ \nabla: \big[\widehat H^{1,q_0}_{\sigma,\ell,0}(\Omega), \widehat H^{1,q_1}_{\sigma,\ell,0}(\Omega)\big]_\theta \to L^q_\ell(\Omega)$$
is bounded. The corresponding estimate is converse to \eqref{interp-ZP}. Summarizing, we proved (iii).
\end{proof}

\vspace*{2ex}

Next we extend the range of admissible exponents $q,\ell$ of Proposition \ref{HhatOmega} to get with Theorem \ref{A1/2-nabla_A1/2} 
%and Corollary \ref{domains:A1/2-nabla_A1/2} 
one of the main results of this article.

\begin{thm}\label{A1/2-nabla_A1/2}
Let $\Omega$ be a $C^2$ exterior domain in $\IR^n$ of class $C^2$ and assume $1<q<\infty$ and $-\frac{n}{q} <\ell< \frac{n}{q'}$. 

(i) There exists a constant $C>0$ independent of $u\in \widehat H^{1,q}_{\sigma,\ell,0}(\Omega)$ such that 
\begin{equation}\label{A1/2-nabla}
    \big\|A_{q,\ell}^{1/2} u\big\|_{L^q_\ell(\Omega)} \leq C\|\nabla u \|_{L^q_\ell(\Omega)}, \quad u\in \widehat H^{1,q}_{\sigma,\ell}(\Omega).
\end{equation} 

(ii) Additionally let $1-\frac{n}{q} <\ell< \frac{n}{q'}$. Then there exists $C>0$ independent of $u$ such that
\begin{equation}\label{nabla-A1/2}
    \|\nabla u \|_{L^q_\ell(\Omega)} \leq C\big\|A_{q,\ell}^{1/2} u\big\|_{L^q_\ell(\Omega)}, \quad u\in \D(A_{q,\ell}^{1/2}).
    \end{equation}

(iii) Under the assumption $1-\frac{n}{q} <\ell< \frac{n}{q'}$ there holds - with equivalent norms -
$$\widehat \D(A_{q,\ell}^{1/2}) = \widehat H^{1,q}_{\sigma,\ell,0}(\Omega).$$
\end{thm}

\begin{proof}
By Proposition \ref{HhatOmega} (ii) \eqref{A1/2-nabla} and \eqref{nabla-A1/2} hold when $2-\frac{n}{q}<\ell<\frac{n}{q'}$. Moreover, for sufficiently smooth  vector fields $u,v$ there holds  $\langle A^{1/2}u, A^{1/2}v \rangle = \langle \nabla u, \nabla v \rangle $.

(i) Due to \eqref{nabla-A1/2} the operator $\nabla A^{-1/2}$ extends to a bounded operator $\nabla A^{-1/2}: L^{q}_{\sigma,\ell}(\Omega) \to L^{q}_\ell(\Omega)$ for all $q,\ell$ such that $2-\frac{n}{q}<\ell<\frac{n}{q'}$. 
%\RED(and for $q_1=2$, $\ell=0$. 
%By complex interpolation (\cite[Theorem 1.18.5]{Triebel}) we obtain that $\nabla A^{-1/2}$ is bounded on $L^{q_\theta}_{\sigma,\theta \ell}(\Omega)$ with $\frac{1}{q_\theta} = \frac{1-\theta}{2} + \frac{\theta}{q_0}$, $0<\theta<1$. This proves \eqref{nabla-A1/2} for the pair $q_\theta,\theta\ell$.)\BLACK

To prove \eqref{A1/2-nabla} note that the sectoriality of $A_{q,\ell}$ implies that $\mathcal R(A_{q,\ell})$ and $\mathcal R(A_{q,\ell}^{1/2})$ are dense in $L^q_{\sigma,\ell}(\Omega)$. 
Applying this density and \eqref{nabla-A1/2} for the pair $(q',-\ell)$  satisfying $2-\frac{n}{q'}< -\ell< \frac{n}{q}$ \BLACK we obtain that 
\begin{align}\label{A1/2-nabla-final}
\begin{aligned}
\big\|A_{q,\ell}^{1/2} u\big\|_{L^q_\ell} & = \sup_{v\neq 0}  \frac{\big|\langle A_{q,\ell}^{1/2} u, A_{q',-\ell}^{1/2} v\rangle\big|}{\|A_{q',-\ell}^{1/2} v\|_{L^{q'}_{-\ell}}} =
\sup_{v\neq 0} \frac{|\langle\nabla u, \nabla v\rangle|}{\|A_{q',-\ell}^{1/2}v\|_{L^{q'}_{-\ell}}} \\
& \leq 
\|\nabla u\|_{L^q_\ell}\; \sup_{v\neq 0}  \frac{ \|\nabla v\|_{L^{q'}_{-\ell}}}{\big\|A_{q',-\ell}^{1/2} v\big\|_{L^{q'}_{-\ell}}}
\leq C\|\nabla u\|_{L^q_\ell}; 
\end{aligned}\end{align}
here, $v\neq 0$ is running through all of  $\widehat\D(A_{q',-\ell}^{1/2})\; %not: \widehat H^{1,q'}_{\sigma,-\ell,0}(\Omega)
$. \BLACK By this argument, \eqref{A1/2-nabla} also holds when $-\frac{n}{q}<\ell<-2+\frac{n}{q'}$ 
%to all $-\frac{n}{q'} < - \ell < -2+\frac{n}{q}$ due to \eqref{A1/2-nabla-final}, 
and $u\in\widehat\D(A_{q,\ell})$. 

However, when $n=3$ or $n=4$, the interval $[-2+\frac{n}{q'}, 2-\frac{n}{q}]$ is nonempty.  In this case we use complex interpolation as in the proof of Corollary \ref{theorem-H-infty-2}; we also refer to Fig. 1, but ignore the conditions $\ell\geq 0$ and $\ell\leq 0$. \BLACK
%
%to prove the result also for any fixed $\ell \in [-2+\frac{n}{q'}, 2-\frac{n}{q}]$ where $1<q<\infty$. Indeed, choose $q_0=q=q_1$ and $\ell_ 0<-2+\frac{n}{q'}$ close to $-2+\frac{n}{q'}$ as well as $\ell_1 > 2-\frac{n}{q}$ close to $2-\frac{n}{q}$. Then there exists $0<\theta<1$ such that $\ell =(1-\theta)\ell_0 + \theta\ell_1$ 
%\RED (Proposition \ref{HhatOmega} (iii)) \BLACK
%so that complex interpolation applied to the pairs of indices $(q_0,\ell_0)$ and $(q_1,\ell_1)$ completes the proof.
%\RED(Proposition \ref{HhatOmega} (iii)) \BLACK and a density argument complete the proof of \eqref{A1/2-nabla} for all $-\frac{n}{q} < - \ell < \frac{n}{q'}$ and $u\in\widehat H^{1,q}_{0,\sigma,\ell}(\Omega)$. \RED (I did not check the range for $(q,\ell)$ after this complex interpolation)\BLACK

(ii) Assume $1-\frac{n}{q} <\ell< \frac{n}{q'}$. Then the variational inequality for $\|\nabla u\|_{L^q_\ell}$, see Theorem \ref{var-ineq-sol},
%Proposition \ref{lem-var-ineq-nabla_Omega0}, 
and (i) prove that 
\begin{align}\label{nablau A12u}
\begin{aligned}
\|\nabla u\|_{L^q_\ell} 
& \leq C \sup_{v\neq 0}   \frac{|\langle \nabla u,\nabla v\rangle\big|}{\|\nabla v\|_{L^{q'}_{-\ell}}} = C
\sup_{v\neq 0}  \frac{\big|\langle A_{q,\ell}^{1/2} u, A_{q',-\ell}^{1/2} v\rangle\big|} {\|\nabla v\|_{L^{q'}_{-\ell}}} \\
& \leq C
\big\|A_{q,\ell}^{1/2}u\big\|_{L^q_\ell}\; \sup_{v\neq 0}  \frac{ \big\|A_{q',-\ell}^{1/2} v\big\|_{L^{q'}_{-\ell}} }{ \|\nabla v\|_{L^{q'}_{-\ell}} }
\leq C \big\|A_{q,\ell}^{1/2} u\big\|_{L^q_\ell}   
\end{aligned}\end{align}
for $u\in \D(A_{q,\ell})$, where $v\neq 0$ is running through all of $\widehat H^{1,q'}_{\sigma,-\ell,0}(\Omega)$.  A density argument extends the result to $u\in \widehat \D(A_{q,\ell}^{1/2})$.

(iii) is an immediate consequence of (i), (ii).
\end{proof}

\begin{cor}\label{domains:A1/2-nabla_A1/2}
Let $\Omega$ be a $C^2$ exterior domain in $\IR^n$, and assume $1<q<\infty$ and $1-\frac{n}{q} <\ell< \frac{n}{q'}$. Then there holds $\widehat \D((-\Delta_{q,\ell})^{1/2}) = \widehat H^{1,q}_{\ell,0}(\Omega)$, and in particular,  
$$ \|(-\Delta_{q,\ell})^{1/2} u\|_{L^q_\ell} \sim \|\nabla u\|_{L^q_\ell} $$
with embedding constants independent of $u\in \widehat H^{1,q}_{\ell,0}(\Omega)$.
\end{cor}

\begin{proof} 
The proof follows the lines of Theorem \ref{A1/2-nabla_A1/2} and former results, but is shorter, since solenoidality does not enter the problem. 
\end{proof}

\subsection{Applications}\label{S3.4} 

Let $A=A_{q,\ell}$ denote the Stokes operator on an exterior domain $\Omega\subset \IR^n$, $n\geq 3$, with boundary of class $C^3$. \BLACK
By definition of the graph norm and the triangle inequality we see that $\mathcal D(A) = \mathcal D(I+A)$ with equivalent graph norms. The following results will imply that these identities and norm equivalencies can be generalized to arbitrarily powers $\theta\in(0,1)$. 

Indeed, the property $BIP$ also implies that for any $0<\theta<1$
\begin{equation}\label{DA=DI+A} 
\mathcal D(A^\theta) = [L^q_{\sigma,\ell}(\Omega), \mathcal D(A)]_\theta = [L^q_{\sigma,\ell}(\Omega), \mathcal D(I+A)]_\theta  = \mathcal D((I+A)^\theta) 
\end{equation}
and $\|u\|_{A^\theta} \simeq \|u\|_{(I+A)^\theta}$;  here we note that with $A$ also $A+I$ has property $BIP$, see \cite[Proposition 2.6 (iv)]{DHP2}.  
Moreover, these identities and norm equivalencies hold when the identity $I$ is replaced by the operator $\epsilon I$ for any $\epsilon>0$.  However, since $\widehat \D(A)\neq \widehat \D(I+A)  = \D(A)$, %\BLUE ((Please check the last "="))  \eqref{DA=DI+A}
will not hold for $\dot A^\theta$. 

%Since $A_q$ is a sectorial operator, $\D(A_q^\theta) = \D((I+A_q)^\theta)$ for $\theta>0$. 
The $\mathscr H^\infty$-calculus of $A_{q,\ell}$ will imply important embedding estimates of fractional powers of $A_{q,\ell}$.

\vspace{2ex}

\begin{prop}\label{domain-of-fractional-power}
Let $1<q < \infty$ and $A=A_{q,\ell} = -\mathbb P \Delta$ be the Stokes operator on $L^q_{\sigma, \ell }(\Omega)$ for $2-\frac{n}{q} < \ell < \frac{n}{q'}$. %\RED with $0\leq \ell$ (By Cor. 3.7 $\ell\geq 0$ is no longer needed). 
\BLACK 
%\BLUE (I think that due to \eqref{estimate-forth-term}, Theorem 3.5 uses $0 \leq \ell$ and thus we need this condition.) 

(i) For $0<\theta<1$, there holds %\RED (( $2\theta-\frac{n}{q} < \ell < \frac{n}{q'}$: not needed in this step, but below)) 
\BLACK
\begin{align}\label{BIP-homog-ext-domain-0}
\widehat{\D}(A^\theta) & = [L^q_{\ell,\sigma}, \widehat{\D}(A)]_{\theta},\\
\widehat{\D}((-\Delta_{q,\ell})^\theta) & = 
[L^q_{\ell}, \widehat{\D}(-\Delta_{q,\ell})]_\theta. \label{BIP-homog-ext-domain-u}  
\end{align} 

(ii) If $ 1< q\leq r<\infty$, 
\begin{equation} \label{restrictions} 
0<\theta\leq \frac12,\quad  2\theta-\frac{n}{q} < \ell < \frac{n}{q'},\quad \BLACK \ell-\ell' = 2\theta + \frac{n}{r} -\frac{n}{q}\geq 0, \end{equation}
then
\begin{align}\label{BIP-homog-ext-domain}
\begin{aligned}
\widehat{\D}(A^\theta) & = \big[L^q_{\ell,\sigma}, \widehat{\D}(A)\big]_{\theta} =  \big[L^q_{\ell,\sigma}, \widehat{\D}(A^{1/2})\big]_{2\theta} =
\big[L^q_{\ell}, \widehat{\D}((-\Delta_{q,\ell})^{1/2})\big]_{2\theta} \cap L^r_{\ell',\sigma}(\Omega)\\[1ex]
& =\big[L^q_{\ell}, \widehat{\D}(-\Delta_{q,\ell})\big]_\theta \cap L^r_{\ell',\sigma}(\Omega) = \widehat{\D}((-\Delta_{q,\ell})^\theta) \cap L^r_{\ell',\sigma}(\Omega) 
\end{aligned}\end{align}
with the norm equivalence 
\begin{equation}\label{fract-norm-equiv}
 \|A^{\theta} u\|_{L^q_{\ell}(\Omega)} \simeq \|u\|_{[L^q_{\sigma, \ell}, \widehat{\D}(A)]_{\theta} } \simeq \|u\|_{\widehat{H}^{2\theta}_{q,\ell}(\Omega)} \simeq \|(-\Delta)^\theta u\|_{L^q_{\ell}(\Omega)}.\end{equation}

% (iii) \RED (All ideas I tried did not work, maybe we have to delete it) 
%If $ 1< q\leq r<\infty$, 
%\begin{equation} \label{restrictions-2} 
%\frac12<\theta\RED \leq 1\BLACK,\quad   2\theta-\frac{n}{q} < \ell < \frac{n}{q'},\quad \BLACK \ell-\ell' = 2\theta + \frac{n}{r} -\frac{n}{q}\geq 0, \end{equation}
%then 
%\begin{equation}\label{fract-norm-equiv2}
%\|A^{\theta} u\|_{L^q_{\ell}(\Omega)} \simeq \|(-\Delta)^\theta u\|_{L^q_{\sigma, \ell}}.
%\end{equation}
\end{prop}

\begin{proof}[Proof of Proposition \ref{domain-of-fractional-power}]
(i) The identities \eqref{BIP-homog-ext-domain-0}, \eqref{BIP-homog-ext-domain-u} follow directly by \eqref{classesSEC}, Proposition \ref{Complex Interpolation} (i) and Corollary \ref{theorem-H-infty-2}. 

(ii) We use Theorem \ref{A1/2-nabla_A1/2}, Corollary \ref{domains:A1/2-nabla_A1/2} and the reiteration theorem for complex interpolation. In the third step note that 
$$ [L^q_\ell, \widehat \D((-\Delta_{q,\ell})^{1/2}]_{2 \theta} = [L^q_\ell, \widehat H^{1,q}_{\ell,0}]_{2 \theta} = \widehat H^{2\theta,q}_\ell \hookrightarrow L^r_{\ell'} $$ 
due to \eqref{spaces2}, Proposition \ref{Hardy-Rellich} (iii) and the restriction \eqref{restrictions}. 
%
%\RED DELETE?: For (iii) we interpolate the standard norm equivalence \eqref{fract-norm-equiv2} for $\theta=\frac12$ and $\theta =1$, {\em i.e.} the spaces $\widehat \D(A^{k/2})$ and $\widehat \D((-\Delta_{q,l})^{k/2})$, $k=1,2$, to obtain \eqref{fract-norm-equiv2} for $\frac12<\theta < 1.$ (???) \BLACK
%
%follows directly by Theorems \ref{theorem-H-infty}, \ref{angles},  Proposition \ref{Complex Interpolation} (i). The embedding into $L^r_{\ell',\sigma}(\Omega)$ is based on Proposition \ref{Hardy-Rellich} (iii).
%\BLUE (I think that Theorems \ref{theorem-H-infty}, \ref{angles},  Proposition \ref{Complex Interpolation} (i) are not enough to prove. Indeed, on the BIP on the homogeneous domains, we need Cor. \ref{domains:A1/2-nabla_A1/2}, and thus $\theta \leq 1/2$. )\RED (The pure  BIP part will hold, but for the identification with $\D(-\Delta)$ we need  $\theta \leq 1/2$) \BLACK
\end{proof}

%\RED Needed? We note that  due to the condition $\theta<\frac1{2q}$ vector fields $u$ in $\widehat{H}^{2\theta}_q(\Omega)^n\cap L^q_\sigma(\Omega)$ do not possess any trace except for the condition $u\cdot \textsl{n}=0$ for the normal component of $u$ on $\partial\Omega$. \BLACK 

\vspace{1ex}

\begin{cor}\label{Lp-Lq At} 
%\RED OLD \BLACK Let $1<q \leq p< \infty$ and $2-\frac{n}{q} < \ell < \frac{n}{q'}$ with $0 \leq \ell$.
%Assume $\frac{1}{p} = \frac{1}{q} - \frac{2\theta}{n}$ where $0\leq\theta\leq 1$. 
%Then there holds the $L^q-L^p$-estimate 
%$$ \|e^{- tA_{q,\ell}}u\|_{L^p_{\ell }(\Omega)} \leq C t^{-\frac{n}{2}\big(\frac{1}{q}-\frac{1}{p}\big)} \|u\|_{L^q_{\ell }(\Omega)}, \quad u\in L^q_{\sigma,\ell }(\Omega),$$
%with a constant $C>0$ independent of $t>0$.

Let $1< q\leq r<\infty$, and assume that 
$$ \kappa-\frac{n}{q} < \ell < \frac{n}{q'}, \quad \ell-\ell' = \kappa + \frac{n}{r} -\frac{n}{q}\geq 0. $$
%\RED (Note that these conditions imply that $-\frac{n}{r} < \ell' < \frac{n}{r'}$, $\kappa$ may be arbitrarily large, but $\kappa<n$ by the above first condition. But in the proof using an iteration and the semigroup property, even $\kappa\geq n$ is possible, isn't it?) 
Then there holds the $L^q$-$L^r$-estimate
$$ \|e^{- tA_{q,\ell}}u\|_{L^r_{\ell'}(\Omega)} \leq C t^{-\frac{n}{2}\big(\frac{1}{q}-\frac{1}{r}\big) -\frac{\ell-\ell'}{2}} \|u\|_{L^q_{\ell}(\Omega)}, \quad u\in L^q_{\sigma,\ell }(\Omega),$$
with a constant $C>0$ independent of $t>0$.
%\BLUE( What condition does  $\kappa$ satisfy related to $2-n/q<\ell$?) \RED (See RED above, $\kappa$ replaces $2$, see also $2\theta$ in Prop. 3.10) \BLACK 
\end{cor}

\begin{proof} 
%\RED OLD \BLACK With $\theta=\frac{n}{2}\big(\frac{1}{q}-\frac{1}{p}\big)$ the result follows from Proposition \ref{domain-of-fractional-power} and the classical estimate $\|A^\theta e^{-tA}\|\leq C t^{-\theta}$. %Indeed,
%$$ \|A^\theta e^{-tA}\| \leq C\|e^{-tA}\|^{1-\theta}\|Ae^{-tA}\|^{\theta} \leq Ct^{-\theta} $$
%by the moment inequality.
%
By the fractional Sobolev embedding $\widehat H^{\kappa,q}_{\ell}(\Omega) \hookrightarrow L^r_{\ell'}(\Omega)$, see Proposition \ref{Hardy-Rellich}, and the norm equivalence \eqref{fract-norm-equiv} we obtain that
$$\|e^{-tA}u\|_{L^r_{\ell'}(\Omega)} \leq  C\|e^{-tA}u\|_{\widehat H^{\kappa,q}_{\ell}(\Omega)} 
\leq C \|A^{\kappa/2} e^{-tA}u\|_{L^q_{\ell}(\Omega)} \leq Ct^{-\frac{n}{2}\big(\frac{1}{q}-\frac{1}{r}\big) -\frac{\ell-\ell'}{2}} \|u\|_{L^q_{\ell}(\Omega)},$$
provided $\kappa\leq 1$. However, if $\kappa > 1$, we decompose $\kappa$ into finitely many pieces $\leq 1$ and use the previous estimate as well as the semigroup property iteratively. 

\end{proof}

\begin{prop}\label{maxreg-L_qell}
 Let $1<q< \infty$ and  $-\frac{n}{q} < \ell < \frac{n}{q'}$. Then $A_{q,\ell}$ possesses for each $1<p<\infty$ maximal $L^p$-regularity on $L^q_{\ell,\sigma}(\Omega)$, {\em i.e.,} for $f\in L^p(\IR_+;L^q_{\ell,\sigma}(\Omega))$ the instationary Stokes system $u_t -\Delta u + \nabla\pi = f,\; \div u=0,\; u(0)=0,\; u\big|_{\partial\Omega} = 0$ possesses a unique solution $(u,\nabla\pi)$ such that 
\begin{align}\label{maxreg-f-u}
     \|u_t\|_{L^p(\IR_+;L^q_\ell)}  + \|A_{q,\ell}u\|_{L^p(\IR_+;L^q_\ell)}  
     %+ \|\nabla\pi\|_{L^p(\IR_+;L^q_\ell)} 
     \leq C\|f\|_{L^p(\IR_+;L^q_\ell)}
\end{align}
with a constant $C>0$ independent of $f$. Under the stronger assumption  $2-\frac{n}{q} < \ell < \frac{n}{q'}$ also $\nabla^2 u$ and $\nabla\pi$ admit the bound
\begin{align}\label{maxreg-pi}
    \|\nabla^2 u\|_{L^p(\IR_+;L^q_\ell)}  + \|\nabla\pi\|_{L^p(\IR_+;L^q_\ell)} & \leq C\|f\|_{L^p(\IR_+;L^q_\ell)}
\end{align}
 with an associated pressure function $\pi$. \BLACK In case of the whole space both estimates hold for all $-\frac{n}{q} < \ell < \frac{n}{q'}$.
\end{prop}

\begin{proof} 
By Theorem \ref{theorem-H-infty} and Corollary \ref{theorem-H-infty-2} $A_{q,\ell}$ possesses a bounded $\mathscr H^\infty$-calculus.
%for each pair $(q,\ell)$ satisfying $2-\frac{n}{q} < \ell < \frac{n}{q'}, \; 0 \leq \ell$, and for   
%$-\frac{n}{q} < \ell < \frac{n}{q'}-2, \; \ell\leq 0.$ 
Since $L^q_{\ell,\sigma}(\Omega)$ as a closed subspace of $L^q_{\ell}(\Omega)$ has Pisier's property $(\alpha)$ (\cite[Definition 3.11]{CdPSW}), the Stokes operator possesses also an $\mathcal R$-bounded $\mathscr H^\infty$-calculus. Thus 
$A_{q,\ell}$ has maximal $L^p$-regularity yielding \eqref{maxreg-f-u}; 
%for each pair $(q,\ell)$ as above; 
for such results see \cite[Theorem 12.8, Remark 12.9]{KuWe04}. 
%Finally, as in the proof of Theorem \ref{A1/2-nabla_A1/2} (i), complex interpolation shows maximal $L^p$-regularity also for the remaining pairs $(q,\ell)$.
The estimate \eqref{maxreg-pi} is based on  \eqref{equ:nabla2-A} \BLACK for an exterior domain and on \eqref{St-res-wRn} together with considering the limit $\lambda\to 0$ for the whole space.
\end{proof}

\BLACK
%\\[1ex]
\noindent {\bf Acknowledgements.}
The second author is supported by JSPS grant whose number is 22K13946.\vspace{1ex} 

\noindent {\bf Statements and Declarations}\vspace{1ex}
 
\noindent {\bf Conflicts of interest statement.}  There is no conflict of interest. \vspace{1ex} 

\noindent {\bf Data Availability statement.} No datasets were generated or analysed during the current study.

%\eqref{n-etAomega>-initial} yeilds that 
%\begin{align}
%\begin{aligned}
%\|\nabla u\|_{L^n(\Omega)} 
%&\leq C t^{-\frac{1}{2}}\|u_0\|_{L^n_{1+2s_1}(\Omega)}
%+ \int_0^{\frac{t}{2}}(t-\tau)^{-\frac{n}{2q}-\frac{1}{2}}\|u\|_{L^n_{s_1}(\Omega)}\|\nabla u\|_{L^q_{s_2}(\Omega)} 
%d\tau\\
%&\quad + \int_{\frac{t}{2}}^t(t-\tau)^{-\frac{1}{2}}\|u\|_{L^n_{s_1}(\Omega)}\|\nabla u\|_{L^q_{s_2}(\Omega)} 
%d\tau
%\end{aligned}
%\end{align}

\end{document}